%% file: PhDThesis.tex
\documentclass[11pt]{book}
\usepackage[utf8x]{inputenc}
\usepackage{amsmath,amscd,amsthm,amsfonts,amssymb}
\usepackage[small, nohug, heads=vee]{diagrams}
\diagramstyle[labelstyle=\scriptstyle]

\usepackage{enumerate}
\usepackage{stmaryrd}

\usepackage[danish,english]{babel}
\usepackage[linkcolor=blue]{hyperref}
\usepackage{setspace}
\hypersetup{
  colorlinks = true,
  urlcolor = black,
  pdfauthor = {Chiara Esposito},
  pdfkeywords = {physics, mathematics},
  pdftitle = {Chiara Esposito: thesis},
  pdfsubject = {PhD Thesis},
  pdfpagemode = UseNone
}
\usepackage[a4paper, bindingoffset=20mm]{geometry}

\usepackage{fancyhdr}
\makeatletter
\renewcommand{\chaptermark}[1]{\markboth{\@chapapp\ \thechapter.\ #1}{}}
\renewcommand{\tableofcontents}{%
    \if@twocolumn
      \@restonecoltrue\onecolumn
    \else
      \@restonecolfalse
    \fi
    \chapter*{\contentsname
        \@mkboth{%
           \itshape\contentsname}{\itshape\contentsname}}%
    \@starttoc{toc}%
    \if@restonecol\twocolumn\fi
    }

\makeatother

\newcommand{\bs}{\boldsymbol}
\newcommand{\ra}{\rightarrow}

\newcommand{\st}{^*}
\newcommand{\inv}{^{-1}}
\newcommand{\mk}{\mathfrak}
\newcommand{\mc}{\mathcal}
\newcommand{\sh}{^{\sharp}}
\newcommand{\R}{\mathbb{R}}
\newtheorem{theorem}{Theorem}[section]
\newtheorem{definition}[theorem]{Definition}
\newtheorem{lemma}[theorem]{Lemma}
\newtheorem{corollary}[theorem]{Corollary}
\newtheorem{proposition}[theorem]{Proposition}
\newtheorem{remark}[theorem]{Remark}
\newtheorem{example}[theorem]{Example}
\newcommand{\calL}{\mathcal{L}}

\begin{document}
\frontmatter % book mode only
\pagenumbering{roman}
\include{cover}
\thispagestyle{empty}
\include{Dedication/dedication}
\thispagestyle{empty}
\include{Acknowledgement/acknowledgement}
\thispagestyle{plain}
\include{Abstract/abstract}
\thispagestyle{plain}

\tableofcontents
\mainmatter
\include{Introduction/introduction}
\include{Chapter1/chapter1}

\include{Chapter2/chapter2}

\include{Chapter3/chapter3}

\include{Conclusions/conclusions}

\bibliographystyle{plain}
\bibliography{references1,references2,references3}	
\end{document}

%% file: cover.tex
{
\thispagestyle{empty}

\noindent \normalsize{Department of Mathematical Sciences\\
The PhD School of Science\\
Faculty of Science\\
University of Copenhagen\\
Denmark}

\vfill
{
\Huge\center{\bfseries On the classical and quantum momentum map}
\vspace{2cm}

\noindent\Large\center{\bfseries A thesis submitted in fulfillment of the requirements for the degree of Doctor of Philosophy}\\

\vspace{1cm}

\Large Chiara Esposito\\
}
\vfill
\noindent\large Advisor : Ryszard Nest\\
\large Discussed : 27 January 2012
}

%% file: Dedication/dedication.tex
% Thesis Dedictation ---------------------------------------------------
\chapter*{}
\begin{flushright}
\textit{To Giovanni and Isabel}
\end{flushright}
\thispagestyle{empty}
% ----------------------------------------------------------------------

%%% Local Variables: 
%%% mode: latex
%%% TeX-master: "../thesis"
%%% End: 

%% file: Acknowledgement/acknowledgement.tex
% Thesis Acknowledgements ------------------------------------------------

%\begin{acknowledgementslong} %uncommenting this line, gives a different acknowledgements heading
\chapter*{Acknowledgments}     %this creates the heading for the acknowlegments

First of all, I would like to thank my supervisor Prof. Ryszard Nest for his continuous encouragement
and his enthusiasm: without our funny discussions, our fights with pens, eraser and blackboard and his patience, this thesis would have never been completed.

Many thanks to Prof. Alan Weinsten and his students for the hospitality and for the beautiful experience I had at the University
of California, Berkeley, both scientific and personal. A special thanks to Benoit Jubin for his help with my paper and all the funny
 ``french conversations'' and to Sobhan Seyfaddini, for all the nice persian food he made me taste.

Many thanks to Prof. Eva Miranda for her nice hospitality at the Universitat Polit\`ecnica de Catalunya, Barcelona, for her support with applications and all the
 ideas and suggestions. And a special thanks to my dear friend Romero Solha, who introduced me to Eva, for our
 crazy discussions about girls and food. Thanks to Prof. Rui Loja Fernandes for his interesting comments on my work
 and his suggestions.

Thanks to all my friends: Giulietto, for all the deep and silly conversations and the projects we share; Rita, Diego, Ninfa and
 Sergio and all my ``neapolitan physics-family''. Thanks to my former advisor, zio Fedele, for his help and support in any moment of
 my life. To Alexandre Albore for the fun I had during these last months with his wrong proofs.
 To Hany and Bryndis, because their tango classes have been one of the nicest moments of the week in this last year.

And thanks to George, for all the intense moments we had in these three years, for his patience in bearing me (it's hard...),
for the beautiful trips and the lazy days, for always offering me a different point of view and for all the laughs we have together.

%\end{acknowledgmentslong}

% ------------------------------------------------------------------------

%%% Local Variables: 
%%% mode: latex
%%% TeX-master: "../thesis"
%%% End: 

%% file: Abstract/abstract.tex
\chapter*{}

\section*{Abstract}
In this thesis we study the classical and quantum momentum maps and the theory of
reduction.
We focus on the notion of momentum map in Poisson geometry and we discuss
the classification of the momentum map in this framework. Furthermore, we
describe the so-called Poisson Reduction, a technique that allows us to
reduce the dimension of a manifold in presence of symmetries
 implemented by Poisson actions.

Using techniques of deformation quantization and quantum groups, we
introduce the quantum momentum map as a deformation of the classical
momentum map, constructed in such a way that it factorizes the quantum
action. As an application we discuss some examples of quantum reduction.

\section*{Resum\'e}

I denne afhandling studerer vi den klassiske impulsafbildning og 
kvanteimpulsafbildningen samt reduktionsteori. Vi fokuserer på 
impulsafbildningen i Poisson-geometri og diskuterer klassifikation af 
impulsafbildningen inden for den ramme. Endvidere beskriver vi den 
såkalte Poisson-reduktion, som er en teknik, der tillader os at reducere 
dimensionen af en mangfoldighed med symmetrier implementeret af 
Poisson-virkninger.

Ved at bruge metoder fra deformationskvantisering og kvantegrupper 
introducerer vi kvanteimpulsafbildningen som en deformation af den 
klassiske impulsafbildning, der er konstrueret så den faktoriserer 
kvantereduktionen. Resultaterne anvendes til at undersøge eksempler på 
kvantereduktioner.

% ----------------------------------------------------------------------

%%% Local Variables: 
%%% mode: latex
%%% TeX-master: "../thesis"
%%% End: 

%% file: Introduction/introduction.tex
%%% Thesis Introduction --------------------------------------------------
\chapter*{Introduction}
\addcontentsline{toc}{chapter}{Introduction}
\markboth{Introduction}{}

The modern notion of momentum map was first introduced by Kostant \cite{Ks} and Souriau \cite{So}, as a refinement of the idea developed by Lie \cite{Lie} in 1890. The momentum map provides a mathematical formalization of the notion of conserved
quantity associated to symmetries of a dynamical system. The definition of momentum map  only requires a canonical
Lie algebra action and its existence is guaranteed whenever the infinitesimal generators of the Lie algebra action are
Hamiltonian vector fields.
    
It is important to underline that the momentum map is a fundamental tool for the study of the so-called symplectic reduction
of certain manifolds \cite{MsWe}.
This method is a synthesis of different techniques of reduction of the phase space (generally a symplectic manifold) associated to a dynamical system, in which the symmetries are divided out. In particular, 
the Marsden-Weinstein reduction \cite{MsWe} gives a description of the symplectic leaves on the orbit space, obtained by a Hamiltonian action on a symplectic manifold.

These theories can be further generalized to the Poisson geometry framework, considering more general structures which
 reduce to the symplectic ones under certain conditions. This generalization has been performed with several approaches, see \textit{e.g.}
 \cite{FOR}, \cite{Lu1} and \cite{Md}. Poisson geometry was introduced by Lie in \cite{Lie} as a geometrization of Poisson's
 studies of classical mechanics \cite{P}. During the past 40 years, Poisson geometry has become an interesting and active field
of research, motivated by connections with many research fields as mechanics of particles and continua (Arnold \cite{A},
Lichnerowicz \cite{Li}, Marsden–Weinstein \cite{MsWe1}) and integrable systems (Gel'fand-Dickey \cite{GD} and Kostant
 \cite{Ks1}).

Furthermore, the theory of Poisson Lie groups has been developed through the work of Drinfel'd \cite{Drinfeld} and 
Semenov-Tian-Shansky \cite{STS}, \cite{ST1} on completely integrable systems and quantum groups. 
These new structures can be naturally used to define Poisson actions, \textit{i.e.} actions of Poisson Lie groups on Poisson manifolds,
as a generalization of Lie group actions. 

In this thesis we mainly focus on a generalization of the momentum map provided by Lu \cite{Lu3}, \cite{Lu1}. In particular, we give a detailed discussion about existence and uniqueness of Lu's momentum map. More precisely, we introduce a
 weaker momentum map, called infinitesimal momentum map, and we study the conditions under which the infinitesimal
  momentum map determines a momentum map in the usual sense. We describe the theory of reconstruction of the momentum
  map from the infinitesimal one in two explicit cases. Moreover, we provide the conditions which ensure the uniqueness
  of the momentum map.
  
 Furthermore, we  exploit Lu's momentum map to construct a theory of reduction for Poisson actions. Indeed, the local description of Poisson manifolds given by Weinstein in \cite{We1} and the properties of Lu's momentum map provide
  an explicit description of the infinitesimal generator of a Poisson action, which allows us to define a Poisson reduced space.

A similar mathematical construction can be implemented for the study of symmetries in quantum mechanics. The passage
between classical and quantum systems can be performed with the theory of deformation quantization 
\cite{Ko}, \cite{BFLS}, \cite{Fe}. We use this approach to study the relation between symmetries in classical and quantum
 mechanics. The key idea resides in defining a quantum momentum map such that the classical conserved quantities
 can be regarded as a classical limit of the quantum ones.
 
The problem of quantization of the momentum map and the reduction has been the main topic of many works, see \textit{e.g.} \cite{Fe1} and \cite{Lu4}. In \cite{Fe1}, Fedosov  uses the theory of deformation quantization to define a quantum momentum map and, as a consequence, a quantum reduction. The author proves that the quantum reduced space is isomorphic to the algebra obtained by canonical deformation quantization of the symplectic reduced manifold. In \cite{Lu4}, Lu defines the quantization of a Poisson action in terms of quantum group action and the quantum momentum map as a map which induces the quantum group action.

Motivated by these two works, in the present thesis we discuss a quantization procedure for the momentum map associated to a Poisson action, which uses quantum group and deformation quantization techniques.

\section*{Plan of Work}
We give here a brief description of the contents of this thesis, underlining the main new results.

Chapter 1 gives some background about momentum map in symplectic geometry.
 After a short overview of Lie group actions and
Hamiltonian systems, we present the momentum map with some basic examples and we describe the
Marsden-Weinstein reduction. The chapter aims to introduce the language that we use throughout this  thesis.

Chapter 2 starts by  recalling some elements of Poisson geometry. In particular, we focus on Poisson Lie groups
and Lie bialgebras and we discuss some properties of Poisson manifolds. Within this framework, we introduce the Poisson action and the momentum map, and, as a warm up, we see the way they generalize the  discussion in Chapter 1.

Afterwards we commence the study of the part of the momentum map, the infinitesimal momentum map, given by a certain family of differential forms associated to the momentum map and which has the advantage that it can be
generalized to the quantized case - see the definition \ref{def: inf}.

In Theorem \ref{thm: rec} we give  explicit, computable conditions for the infinitesimal momentum map to determine a global momentum map. In  Section \ref{sec: rec} we study concrete cases of  this globalization question and we 
prove the existence and uniqueness/nonuniqueness of a momentum map associated to a given infinitesimal momentum map, for the particular case when the dual Poisson Lie group is abelian, respectively the Heisenberg group.

In the following section we study the question of  infinitesimal deformations of a given momentum map. The main result is Theorem \ref{thm: idef}, which describes explicitly the space tangent to the space of momentum maps at a 
given point. In particular, modulo the trivial infinitesimal deformations (which are given by vector fields on $M$ commuting with the Poisson bivector and  invariant under the group action), the tangent space is given by a subspace of the cohomology group $H^1({\mathfrak g},C^\infty (M))$. As an application, in the case of a compact and semisimple Poisson Lie group
 $G$ acting on a Poisson manifold $M$, the only infinitesimal deformations of a momentum map are the trivial ones.

The following section (Sec. \ref{sec: poisson reduction}) is devoted to  the theory of Poisson reduction. The main result of this section is the theorem \ref{thm: pred}, which gives a generalization of the Marsden-Weinstein 
reduction  to the general case of an arbitrary Poisson Lie group action on a Poisson manifold - the results that can be found in the literature are restricted to the case when the Poisson manifold in question is in fact a 
symplectic manifold.

Chapter 3 is dedicated to the study of the quantum momentum map. As proved by Kontsevich \cite{Ko}, any Poisson manifold
 admits a canonical quantization. We start this chapter by giving a short introduction to Kontsevich's theorem and its corollary describing  the deformation quantization theory for Poisson manifolds. Afterwards we give some background information on  Hopf algebras, quantum groups and the quantization of  Poisson Lie groups and Lie bialgebras.

The quantum action is naturally expressed in terms of the action of the Hopf algebra on the deformation quantization of the smooth manifold.

Since we work with deformation quantization of the Lie bialgebra of our Poisson Lie group, the natural context is quantization of the infinitesimal momentum map. This is defined in Section \ref{sec: def}. One can briefly describe it as follows. Let $\mathcal{U}_{\hbar}$ denote the quantized Lie bialgebra. Then the coproduct $\Delta_{\hbar}$ of $\mathcal{U}_{\hbar}$ extends to an odd  derivation on the tensor algebra $T(\mathcal{U}_{\hbar}[1])$. If  $\mathcal{A}_{\hbar}$ denotes the quantized algebra of functions on the manifold, the action of $\mathcal{U}_{\hbar}$ on $\mathcal{A}_{\hbar}$ becomes  a morphism of complexes
$$
\Phi_{\hbar}: (T(\mathcal{U}_{\hbar}[1]),\delta )\rightarrow (C^{\bullet} (\mathcal{A}_{\hbar},\mathcal{A}_{\hbar}),b ),
$$
and quantum momentum map is a factorization of this morphism through the complex of formal differential forms $\Omega^{\bullet}(\mathcal{A}_{\hbar}^+)$, which reduces, modulo $\hbar^2$, to the classical infinitesimal momentum map.

Section \ref{sec: exs} is devoted to the discussion of some examples of quantum momentum map and of quantum reduction. The interesting observation is that the existence of the quantum momentum map in the above sense dictates the formulas for the quantization of the Lie bialgebra. In fact, in the low dimensional examples studied in this section, the quantization of the Lie bialgebra is essentially uniquely  determined by existence of universal formulas for the quantum momentum map. 

The corresponding quantum reduction seems to be much more subtle then in the case of the undeformed action of the group. The problem is as follows. In the case of undeformed group action, the quantum momentum map has the form $X\rightarrow dH_X$, and the reduction has the form
$(\mathcal{A}/ \langle H_X,\; X\in {\mathfrak g}\rangle)^{\mathcal{U}({\mathfrak g})}$. In the general case, the ``Hamiltonian forms'' become of the form $adb$ with, as the examples show, some of the $a$'s (and $b$'s) invertible, hence the quotients of this form vanish. We can construct the quantum reduction in our case, but it seems, at the moment, somewhat ad hoc. 
 
%%% ----------------------------------------------------------------------

%%% Local Variables: 
%%% mode: latex
%%% TeX-master: "../thesis"
%%% End: 

%% file: Chapter1/chapter1.tex
\chapter{Momentum Map in Symplectic Geometry}
\label{ch: one}

In this chapter we explain how the symmetries of a Hamiltonian dynamical system can be used to simplify the study of that system. The description of symmetries is implemented via Lie group actions, discussed in the first part of the chapter. The treatment of Lie group actions and Hamiltonian systems is necessary to introduce the key concept of this chapter, the momentum map. This is a mathematical construction introduced by Lie \cite{Lie}, Kostant \cite{Ks} and Souriau \cite{So} that describes the conservation laws associated to the symmetries of a Hamiltonian dynamical system.
 
As will be seen, the momentum map is the basic ingredient for the construction of the symplectic reduction.
This is a procedure that uses symmetries and conserved quantities to reduce the dimensionality of a Hamiltonian system.

\section{Lie Group Actions}\label{sec_1.1}

This section presents a brief review of the theory of Lie group actions on a manifold, with some examples. 
This provides the background necessary to introduce the concepts of symmetries of a Hamiltonian system and momentum map.
For details the reader should consult \textit{e.g.} \cite{RO} and \cite{MR}.

\begin{definition}
Let $M$ be a manifold and $G$ a Lie group. A \textbf{left action} of $G$ on $M$ is a smooth mapping $\Phi:G\times M\rightarrow M$ such that
\begin{enumerate}
 \item $\Phi(e,m)=m, \qquad e \in G, \quad \forall m\in M$,
 \item $\Phi(g,\Phi(h,m))=\Phi(gh,m),\qquad\forall g,h\in G, m\in M$,
\end{enumerate}
where $e$ denotes the identity of $G$.
\end{definition}
We often use the notation $g\cdot m:=\Phi (g,m):=\Phi_g(m):=\Phi^m(g)$.
Similarly as for the left action, a \textbf{right action} is a smooth map $\Phi:M\times G\rightarrow M$, such that $\Phi(m,e)=m$, for all $m\in M$, and $\Phi(\Phi(m,g),h)=\Phi(m,gh)$, for all $g,h\in G$ and $m\in M$.

\begin{example}
An example of left action of a Lie group $G$ on itself is given by the left translation $L_g:G\rightarrow G: h\mapsto gh$. Similarly, the right translation $R_g:G\rightarrow G$, $h\mapsto hg$ defines a right action. The inner automorphism $AD_g\equiv I_g:G\rightarrow G$, given by $I_g:=R_{g^{-1}}\circ L_g$, defines a left action of $G$ on itself called conjugation.
\end{example}

\begin{example}
The differential at the identity of the conjugation map defines a linear left action of a Lie group $G$ on its Lie algebra $\mathfrak{g}$, called the adjoint representation of $G$ on $\mathfrak{g}$, that is
\begin{equation}
 Ad_g:= T_e I_g:\mathfrak{g}\rightarrow\mathfrak{g}.
\end{equation}
If $Ad^*_g:\mathfrak{g}^*\rightarrow\mathfrak{g}^*$ is the dual of $Ad_g$, then the map
\begin{equation}
\Phi: G\times\mathfrak{g}^*\rightarrow\mathfrak{g}^* :(g,\nu)\mapsto Ad^*_{g^{-1}}\nu
\end{equation}
defines also a linear left action of $G$ on $\mathfrak{g}^*$ called the coadjoint representation of $G$ on $\mathfrak{g}^*$.
\end{example}

Given a (left) action $\Phi:G\times M\rightarrow M$, the \textbf{infinitesimal generator} $\xi_M\in TM$ associated to $\xi\in\mathfrak{g}$ is the vector field on $M$ defined by

\begin{equation}
  \xi_M(m):=\left.\frac{d}{dt}\right|_{t=0} \Phi_{\exp (- t\xi)}(m)=T_e \Phi^m\cdot\xi
\end{equation}

An infinitesimal generator is a complete vector field. Indeed, the flow of $\xi_M$ equals $(t,m)\mapsto \exp t\xi\cdot m$. Moreover, the map $\xi\in\mathfrak{g}\mapsto\xi_M\in TM$ is a Lie algebra homomorphism, that is,
\begin{enumerate}
\item $(a\xi+b\eta)_M=a\xi_M+b\eta_M$
\item $[\xi,\eta]_M=[\xi_M,\eta_M]$,
\end{enumerate} 
for all $a,b\in\mathbb{R}$ and $\xi,\eta\in\mathfrak{g}$. Motivated by these properties, we introduce the following definition.

\begin{definition}
 Let $\mathfrak{g}$ be a Lie algebra and $M$ a smooth manifold. A right (left) \textbf{Lie algebra action} of $\mathfrak{g}$ on $M$ is a Lie algebra homomorphism $\xi\in\mathfrak{g}\mapsto\xi_M\in TM$ such that the mapping $(m,\xi)\in M\times \mathfrak{g}\mapsto\xi_M(m)\in TM$ is smooth.
\end{definition}

Given a Lie group action, we refer to the Lie algebra action induced by its infinitesimal generators as the associated Lie algebra action. An example of Lie algebra action is given by the \textbf{adjoint representation} of the algebra $\mathfrak{g}$, defined by the map
\begin{equation}\label{eq: ad}
ad: \mathfrak{g}\rightarrow End(\mathfrak{g}):\xi\mapsto ad_{\xi}=[\xi,\cdot]
\end{equation}

Consider a Lie group $G$ acting on a manifold $M$.
The \textbf{isotropy subgroup} or \textbf{stabilizer} of an element $m\in M$ acted upon by the Lie group $G$ is the closed subgroup

\begin{equation}
 G_m:=\left\lbrace g\in G|\, g\cdot m=m\right\rbrace \subset G
\end{equation}
whose Lie algebra $\mathfrak{g}_m$ equals
\begin{equation}
 \mathfrak{g}_m=\left\lbrace \xi\in\mathfrak{g}\vert\, \xi_M(m)=0 \right\rbrace 
\end{equation}

The \textbf{orbit} $\mathcal{O}_m$ of the element $m\in M$ under the group action $\Phi$ is the set 
\begin{equation}
 \mathcal{O}_m\equiv G\cdot m:= \left\lbrace g\cdot m\vert\, g\in G\right\rbrace .
\end{equation}
The notion of orbit can be used to characterize an equivalence relation on the manifold $M$. Two elements $x,y\in M$ 
are equivalent, $x\sim y$, if and only if they are in the same orbit, hence if there exists an element $g\in G$
such that $\Phi_g(x)=y$. The \textbf{orbit space} is the space of these equivalence classes and is denoted by $M/G$.
 
In the following we discuss the conditions which ensure that the orbit space is a regular quotient manifold.  

A group action on $M$ is said to be
\begin{enumerate}
\item[-] transitive, if there is only one orbit,
\item[-] free, if the isotropy of every element in $M$ consists only of the identity element, 
\item[-] faithful, if $\xi_M=id_M$ implies that $g=e$
 \end{enumerate}

Let $X$ and $Y$ be two topological spaces with $Y$ first countable. A continuous map $f:X\rightarrow Y$ is called proper if for any sequence $\lbrace x_n\rbrace_{n\in\mathbb{N}}$ such that $f(x_n)\rightarrow y$ there exist a convergent subsequence $\lbrace x_{n_k}\rbrace$ such that $x_{n_k}\rightarrow x$ and $f(x)=y$. A map $f:X\rightarrow Y$ is proper if and only if it is closed and $f^{-1}(y)$ is compact, for any $y\in Y$.

\begin{definition}
Let $G$ be a Lie group acting on the manifold $M$ via the map $\Phi:G\times M\rightarrow M$. We say that $\Phi$ is \textbf{proper} whenever the map $\Theta:G\times M\rightarrow M\times M$ defined by $\Theta (g,m)=(m,\Phi(g,m))$ is proper. 
\end{definition}
The properness of the action is equivalent to the following condition: for any two convergent sequences $\lbrace m_n\rbrace$ and $\lbrace g_n\cdot m_n\rbrace$ in $M$, there exists a convergent subsequence $\lbrace g_{n_k}\rbrace$ in $G$. We say that the action $\Phi$ is proper at the point $m\in M$ when for any two convergent sequences $\lbrace m_n\rbrace$ and $\lbrace g_n\cdot m_n\rbrace$ in $M$ such that $m_n\rightarrow m$ and $g_n\cdot m_n\rightarrow m$, there exists a convergent subsequence $\lbrace g_{n_k}\rbrace$ in $G$.

The following proposition is crucial in the theory of symplectic reduction. A proof can be found \textit{e.g.} in \cite{RO} and \cite{Bo}.

\begin{proposition}\label{1.1_orbit}
Let $\Phi:G\times M\rightarrow M$ be a proper action of the Lie group $G$ on the manifold $M$. Then
\begin{enumerate}
\item For any $m\in M$, the isotropy subgroup $G_m$ is compact.
\item The orbit space $M/G$ is a Hausdorff topological space
\item If the action is free, $M/G$ is a smooth manifold, and the canonical projection $\mathfrak{p}:M\rightarrow M/G$ defines on $M$ the structure of a smooth left principal $G$-bundle.
\end{enumerate}
\end{proposition}

\section{Hamiltonian systems}\label{sec: 1.2}

We start this section with a brief review of some basic notions about symplectic manifolds and invariant Hamiltonian dynamics. This provides us all the necessary background that we use in the next sections.
\begin{definition}
A \textbf{symplectic manifold} is a pair $(M,\omega)$, where $M$ is a manifold and $\omega\in\Omega^2(M)$ is a closed nondegenerate two form on $M$, that is, $d\omega=0$ and, for every $m\in M$, the map $v\in T_m M\mapsto \omega(m)(v,\cdot)\in T^*_mM$ is a linear isomorphism between the tangent space $T_mM$ and the cotangent space $T^*_mM$.
\end{definition}

Since $\omega$ is a differential two-form, hence skew-symmetric, the dimension of $M$ is always even. 
A Hamiltonian dynamical system is a triple $(M,\omega, H)$, where $(M,\omega)$ is a symplectic manifold and $H\in C^{\infty}(M)$ is the Hamiltonian function of the system. By nondegeneracy of the symplectic form $\omega$, to each Hamiltonian system one can associate a \textbf{Hamiltonian vector field} $X_H\in TM$, defined by the identity

\begin{equation}
 i_{X_H}\omega=dH
\end{equation}

It's clear that $X\in TM$ is a Hamiltonian vector field if and only if the one form $i_X\omega$ is exact.

\begin{definition}
 Let $f,g\in C^{\infty}(M)$. The \textbf{Poisson bracket} of these functions is the function $\left\lbrace f,g\right\rbrace\in C^{\infty}(M)$ defined by
\begin{equation}
 \left\lbrace f,g\right\rbrace(m)=\omega(m)\left( X_f(m),X_g(m)\right)=X_g[f](m)=-X_f[g](m). 
\end{equation}
\end{definition}

Given a symplectic manifold $(M,\omega)$, the set $C^{\infty}(M)$ can be always equipped  with a real Lie algebra structure relative to the Poisson bracket:

\begin{definition}\label{1.2_pa}
 A \textbf{Poisson manifold} is a pair $(M,\left\lbrace \cdot,\cdot\right\rbrace )$, where $M$ is a smooth manifold and $\left\lbrace \cdot,\cdot\right\rbrace$ is a bilinear operation on $C^{\infty}(M)$, such that the pair $(C^{\infty}(M),\left\lbrace \cdot,\cdot\right\rbrace)$ is a Lie algebra and $\left\lbrace \cdot,\cdot\right\rbrace$ is a derivation in each argument. The pair $(C^{\infty}(M),\left\lbrace \cdot,\cdot\right\rbrace)$ is called Poisson algebra. The functions in the center of the Lie algebra $(C^{\infty}(M), \left\lbrace \cdot, \cdot\right\rbrace)$ are called Casimir functions.
\end{definition}

Since there is an isomorphism between derivations on $C^{\infty}(M)$ and vector fields on $M$, it follows that each $H\in C^{\infty}(M)$ induces a vector field on $M$ via the expression
\begin{equation}\label{eq: hvf}
 X_H=\left\lbrace \cdot,H\right\rbrace ,
\end{equation}
called the Hamiltonian vector field associated to the Hamiltonian function $H$. The Hamiltonian equations $\dot{z}=X_H(z)$ can be equivalently written in Poisson bracket form as
\begin{equation}
 \dot{f}=\left\lbrace f,H\right\rbrace,
\end{equation}
for any $f\in C^{\infty}(M)$. The triple $(M,\left\lbrace \cdot,\cdot\right\rbrace, H)$ is called a Poisson dynamical system. The Lie algebra map $(C^{\infty}(M),\left\lbrace \cdot,\cdot\right\rbrace)\rightarrow ( TM,\left[\cdot,\cdot \right])$ that assigns to each function $f\in C^{\infty}(M)$ the associated Hamiltonian vector field $X_f\in TM$ is a Lie algebra homomorphism:

\begin{equation}
 X_{\left\lbrace f,g\right\rbrace } = \left[ X_f,X_g\right] \qquad \forall f,g\in C^{\infty}(M).
\end{equation}

Any Hamiltonian system on a symplectic manifold is a Poisson dynamical system relative to the Poisson bracket induced by the symplectic structure. 
Given a Poisson dynamical system $(M,\left\lbrace \cdot,\cdot\right\rbrace, H)$, its conserved quantities or
integrals of motion are defined by the subalgebra $C^{\infty}(M)^G$ of $(C^{\infty}(M),\left\lbrace \cdot,\cdot\right\rbrace)$ consisting of the $G$-invariant functions on $M$, i.e $f\in C^{\infty}(M)$ such that $\left\lbrace f,H\right\rbrace=0$.

As mentioned above, the symmetries of a Hamiltonian system are encoded via Lie group actions consistent with the structure of the given dynamical system. This motivates the following definition:
\begin{definition}\label{def: can}
A \textbf{canonical} action is a map $\Phi:G\times M\rightarrow M$ such that
\begin{equation}
 \Phi_g^*\left\lbrace f,g\right\rbrace=\left\lbrace \Phi_g^*f,\Phi_g^*g\right\rbrace  \qquad (\text{resp.}	\quad\Phi_g^*\omega=\omega)
\end{equation}
\end{definition}
We say that the Hamiltonian system $(M,\left\lbrace \cdot,\cdot\right\rbrace, H)$ is $G$-symmetric when the Lie group $G$ acts canonically on $(M,\left\lbrace \cdot,\cdot\right\rbrace)$ and the Hamiltonian function $H$ is $G$-invariant, that is $\Phi^*_g(H)=H$.

The infinitesimal version of this concept is the canonical action of a Lie algebra. An action of the Lie algebra $\mathfrak{g}$ on the Poisson (respectively, symplectic) manifold $M$ is canonical if the vector fields $\xi_M\in  TM$ are infinitesimal Poisson automorphisms, that is, if $\pi$ is the Poisson tensor we have that $L_{\xi_M}\pi=0$ (respectively, $L_{\xi_M}\omega=0$). We say that the Hamiltonian system $(M,\left\lbrace \cdot,\cdot\right\rbrace, H)$ is $\mathfrak{g}$-symmetric if the Lie algebra $\mathfrak{g}$ acts canonically on $(M,\left\lbrace \cdot,\cdot\right\rbrace)$ and the Hamiltonian function $H$ is $\mathfrak{g}$-invariant.

\section{Momentum map}\label{sec_1.3}

In the previous sections we introduced the symmetries of a Hamiltonian systems via Lie group actions. The conservation laws
of this system can be described with a mathematical construction called the momentum map. The definition of momentum map 
only requires a canonical Lie algebra action and its existence is guaranteed when the infinitesimal generators of the action are Hamiltonian vector fields.

\begin{definition}\label{1.3_mms}
Let $\mathfrak{g}$ be a Lie algebra acting canonically on the Poisson manifold $(M,\left\lbrace \cdot,\cdot\right\rbrace)$. Suppose that for any $\xi\in\mathfrak{g}$ the vector field $\xi_M$ is Hamiltonian, with Hamiltonian function $\boldsymbol{\mu}^{\xi}\in C^{\infty}(M)$ such that
\begin{equation}
\xi_M=X_{\boldsymbol{\mu}^{\xi}}.
\end{equation}
The map $ \boldsymbol{\mu}:M\rightarrow\mathfrak{g}^*$ defined by the relation
\begin{equation}
\boldsymbol{\mu}^{\xi}(m)=\left\langle \boldsymbol{\mu}(m),\xi\right\rangle 
\end{equation}
for all $\xi\in\mathfrak{g}$ and $m\in M$, is called \textbf{momentum map} of the $\mathfrak{g}$-action.
\end{definition}

Notice that the momentum map is not uniquely determined; indeed, $\boldsymbol{\mu}^{\xi}_1$ and $\boldsymbol{\mu}^{\xi}_2$ are momentum maps for the same canonical action if and only if for any $\xi\in \mathfrak{g}$
\begin{equation}
\boldsymbol{\mu}^{\xi}_1-\boldsymbol{\mu}^{\xi}_2
\end{equation}
is a Casimir function; if $M$ is symplectic and connected, then $\boldsymbol{\mu}$ is determined up to a constant in $\mathfrak{g}^*$.

\begin{example}[Linear momentum]
We consider the phase space $T^*\mathbb{R}^{3N}$ of a $N$-particle system. The additive group $\mathbb{R}^3$ acts on it by applying spatial translation on each factor: $\mathbf{v}\cdot (\mathbf{q}_i,\mathbf{p}^i)=(\mathbf{q}_i+\mathbf{v},\mathbf{p}^i )$, with $i=1,\dots, N$. This action is canonical and has an associated momentum map that coincides with the classical linear momentum
\begin{align}
\boldsymbol{\mu}:& T^*\mathbb{R}^{3N}\rightarrow Lie(\mathbb{R}^3)\simeq \mathbb{R}^3\\
                       & (\mathbf{q},\mathbf{p})\mapsto \sum_{i=1}^N \mathbf{p}_i
\end{align}

\end{example}

\begin{example}[Angular momentum]
Let $SO(3)$ act on $\mathbb{R}^3$ and then, by lift, on $T^*\mathbb{R}^{3}$, that is, $A\cdot (\mathbf{q},\mathbf{p})=(A\mathbf{q}, A\mathbf{p})$. This action is canonical and has an associated momentum map
\begin{align}
\boldsymbol{\mu}:& T^*\mathbb{R}^{3}\rightarrow \mathfrak{so(3)}^*\simeq \mathbb{R}^3\\
                           & (\mathbf{q},\mathbf{p})\mapsto \mathbf{q}\times\mathbf{p},
\end{align}
which is the classical angular momentum.
\end{example}

It can be shown that the momentum map satisfies Noether's Theorem \cite{N}, as stated in the following theorem.

\begin{theorem}[\cite{RO}]
Let $G$ be a Lie group acting canonically on the Poisson manifold $(M,\lbrace\cdot,\cdot\rbrace)$. Assume that this action admits a momentum map $\boldsymbol{\mu}:M\rightarrow\mathfrak{g}^*$ and that $H\in C^{\infty}(M)$ is invariant under the action of $\Phi$. 
Then the momentum map is an integral for the Hamiltonian vector field $X_H$ (\textit{i.e.} if $F_t$ is the flow of $X_H$ then $\boldsymbol{\mu}(F_t(x))=\boldsymbol{\mu}(x)$ for all $x$ and $t$ where $F_t$ is defined).
\end{theorem}

Finally, we introduce the property of equivariance of the momentum map. Let $(M,\lbrace\cdot,\cdot\rbrace)$ be a Poisson manifold and $\mathfrak{g}$ act canonically on it with a momentum map $ \boldsymbol{\mu}:M\rightarrow\mathfrak{g}^*$.
The map $(\mathfrak{g}, \left[ \cdot,\cdot\right] )\rightarrow (C^{\infty}(M),\lbrace\cdot,\cdot\rbrace)$ defined by $\xi\mapsto \boldsymbol{\mu}^{\xi}$, $\xi\in\mathfrak{g}$ is a Lie algebra homomorphism if and only if
\begin{equation}
T_z \boldsymbol{\mu}\cdot \Phi_{\xi}(m) = ad^*_{\xi}\boldsymbol{\mu}(m)
\end{equation}
for any $\xi\in\mathfrak{g}$ and any $m\in M$. A momentum map that satisfies this relation is called infinitesimally equivariant. When the Lie algebra action is associated to the action of a Lie group $G$, we say that $\boldsymbol{\mu}$ is $G$-equivariant if
\begin{equation}
\boldsymbol{\mu}\circ\Phi_g=Ad_g^*\circ\boldsymbol{\mu}
\end{equation}
for all $g\in G$. A Lie algebra action with an infinitesimally equivariant momentum map is called \textbf{Hamiltonian action} and a Lie group action with an equivariant momentum map is called globally Hamiltonian.

Details about the problem of the existence of the momentum map can be found in \cite{RO}.

\section{Symplectic Reduction}\label{sec_1.4}

In this section we describe the simplest version of symplectic reduction that constructs a symplectic manifold out of a given
symmetric one, on which the conservation laws and degeneracies associated to the symmetries have been eliminated.
Given a symmetric Hamiltonian dynamical system, the Marsden-Weinstein reduced system is also a Hamiltonian system
with reduced dimensionality, as proved in the following theorem:

\begin{theorem}[Marsden-Weinstein Reduction  \cite{MsWe}]\label{1.4_redw}
 Let $\Phi:G\times M\rightarrow M$ be a canonical action of the Lie group $G$ on the connected symplectic manifold $(M,\omega)$. Suppose that the action has an associated equivariant momentum map $\boldsymbol{\mu}:M\rightarrow \mathfrak{g}^*$. Let $u\in\mathfrak{g}^*$ be a regular value of $\boldsymbol{\mu}$ and assume that  the isotropy group $G_{u}$ under the $Ad^*$ action on $\mathfrak{g}^*$ acts freely and properly on $\boldsymbol{\mu}^{-1}(u)$. Then:
\begin{enumerate}
\label{thm: red1}\item  the space $M_{u}:=\boldsymbol{\mu}^{-1}(u)/G_{u}$ is a regular quotient manifold and there is a symplectic structure $\omega_{u}$ on $M_{u}$ uniquely determined by $i^*_{u}\omega=\mathfrak{p}^*_{u} \omega_{u}$, where $i_u:\boldsymbol{\mu}^{-1}(u)\hookrightarrow M$ is the natural inclusion and $\mathfrak{p} $ is the natural projection of $\boldsymbol{\mu}^{-1}(u)$ onto $M_{u}$. The pair $(M_{u},\omega_{u})$ is called the \textbf{symplectic reduced space}.

\label{thm: red2}\item Let $H\in C^{\infty}(M)^G$ be a $G$-invariant Hamiltonian. The flow $F_t$ of the Hamiltonian vector field $X_H$ leaves the connected components of $\boldsymbol{\mu}^{-1}(u)$ invariant and commutes with the $G$-action, so it induces a flow $F_t^{u}$ on $M_{u}$ defined by 

\begin{equation}
\mathfrak{p}_{u}\circ F_t\circ i_{u}=F_t^{\nu}\circ\mathfrak{p}_{u}.
\end{equation}

\label{thm: red3}\item The vector field generated by the flow $F_t^{u}$ on $(M_{u},\omega_{u})$ is Hamiltonian with associated reduced Hamiltonian function $H_{u}\in C^{\infty}(M_{u})$ defined by

\begin{equation}
H_{u}\circ\mathfrak{p}_{u}=H\circ i_{u}.
\end{equation}

The vector fields $X_H$ and $X_{H_{u}}$ are $\mathfrak{p}_{u}$-related. The triple $(M_{u},\omega_{u},H_{u})$ is called reduced Hamiltonian system.

\label{thm: red4}\item Let $K\in C^{\infty}(M)^G$ be another $G$-invariant function. Then $\lbrace H,K\rbrace$ is also $G$-invariant and $\lbrace H,K\rbrace_{u}=\lbrace H_{u},K_{u}\rbrace_{M_{u}}$, where $\lbrace\cdot,\cdot\rbrace_{M_{u}}$ denotes the Poisson bracket associated to the symplectic form $\omega_{u}$ on $M_{u}$.
\end{enumerate} 
 
\end{theorem}

The symplectic reduction can be rephrased in terms of algebra of functions. In \cite{SW}, Sniatycki and Weinstein proposed a
different procedure that yields a reduced Poisson algebra; they proved that this algebra coincides with the Poisson algebra on the reduced space $M_u$. 
We briefly introduce the results obtained in \cite{SW}, which we use in the next chapters.

 Consider the Hamiltonian function $\boldsymbol{\mu}^{\xi}\in C^{\infty}(M)$, with $\xi\in\mathfrak{g}$ and let 
 $\boldsymbol{\mu}_i=\boldsymbol{\mu}^{e_i}$, $i=1,\dots ,n$ be the components of the Hamiltonian function on the
 basis $\{e_i\}$ of $\mathfrak{g}$. 
  Define the ideal $\mathcal{I}$ in  $C^{\infty}(M)^G$ generated by the momenta $\boldsymbol{\mu}_i$ as 
\begin{equation}
 \mathcal{I}=\lbrace f\in C^{\infty}(M)^G\vert f=\sum_i g_i\boldsymbol{\mu}_i, g_i\in  C^{\infty}(M) \rbrace .
\end{equation}
Then we have:

\begin{lemma}\label{lem: reda}
Let $\mathfrak{g}$ be a Lie algebra acting canonically on the Poisson manifold $(M,\left\lbrace \cdot,\cdot\right\rbrace)$ with
$G$-equivariant momentum map $\boldsymbol{\mu}:M\rightarrow \mathfrak{g}^*$. The the ideal $\mathcal{I}$ is a Poisson subalgebra of $C^{\infty}(M)^G$.
\end{lemma}

The action of $G$ on $C^{\infty}(M)$ induces an action of $G$ on the quotient algebra $C^{\infty}(M)^G$ such that the projection homomorphism $\rho:C^{\infty}(M)\rightarrow C^{\infty}(M)/\mathcal{I}$ is $G$-equivariant. In \cite{SW} the authors proved that the quotient $\mathcal{R}= C^{\infty}(M)^G/\mathcal{I}$ naturally inherits a Poisson algebra structure. More precisely, under the assumptions of  Lemma \ref{lem: reda}, we have

\begin{lemma}
 $\rho^{-1}(C^{\infty}(M)^G)$ is the normalizer of $ \mathcal{I}$ and it has the structure of a Poisson subalgebra of $C^{\infty}(M)$.
\end{lemma}

\begin{corollary}
 $C^{\infty}(M)^G$ inherits the structure of a Poisson algebra such that $\rho$ restricted to $\rho^{-1}(C^{\infty}(M)^G)$ is a Poisson algebra homomorphism.
\end{corollary}

The Poisson algebra $\mathcal{R}$ is called the reduced Poisson algebra of the considered system. Finally, we have:

\begin{theorem}
The Poisson algebra $\mathcal{R}$ is canonically isomorphic to the Poisson algebra of the reduced phase space $C^{\infty}(M_{u})$ with Poisson structure induced by $\omega_{u}$.
\end{theorem}

%% file: Chapter2/chapter2.tex
\chapter{Momentum Map in Poisson Geometry}\label{ch: two}

In this chapter we discuss the generalization of the theory of momentum map and reduction to the Poisson geometry case. Similarly to the previous chapter, in this formalism the description of symmetries is 
implemented via Poisson actions. In order to introduce Poisson actions we give some background about Poisson Lie groups and Lie bialgebras and we discuss a more complete definition of Poisson manifolds 
with certain properties. 
We introduce the generalization of momentum map given by Lu \cite{Lu3} and its related Hamiltonian action.
After this introductory part, we give a new definition of the momentum map in terms of one-forms and we study its properties. As in the symplectic case, the momentum map here plays a fundamental role in the 
construction of the Poisson reduction theory.

\section{Lie bialgebras}
\label{sec:lie bialgebras}

The first object under consideration is a generalization of the notion of Lie algebra, called Lie bialgebra. In general, given a Lie algebra $\mathfrak{g}$, its dual space $\mathfrak{g}^*$ is a vector space. In this section we see that $\mathfrak{g}$ can be endowed with a structure  which induces a Lie algebra structure on its dual $\mathfrak{g}^*$.
 The corresponding Lie groups carry a Poisson structure compatible with the group multiplication.
 For details on this topic see e.g. \cite{YK}. As will be seen in next chapter, the importance of these structures
 relies in the fact that they admit a standard procedure of quantization.

Let $\mathfrak{g}$ be a finite dimensional Lie algebra and $\delta$ a linear map from $\mathfrak{g}$ to $\mathfrak{g}\otimes \mathfrak{g}$ with transpose
${}^{t}\delta: \mathfrak{g^*}\otimes\mathfrak{g^*}\rightarrow\mathfrak{g^*}$. Recall that a linear map on $\mathfrak{g^*}\otimes\mathfrak{g^*}$ can be identified with a bilinear map on $\mathfrak{g}$.

\begin{definition}
A \textbf{Lie bialgebra} is a Lie algebra $\mathfrak{g}$ with a linear map $\delta:\mathfrak{g}\rightarrow\mathfrak{g}\wedge\mathfrak{g}$ such that
\begin{enumerate}
 \item ${}^{t}\delta: \mathfrak{g^*}\otimes\mathfrak{g^*}\rightarrow\mathfrak{g^*}$ defines a Lie bracket on $\mathfrak{g}^*$, and
 \item \label{cond2} $\delta$ is a 1-cocycle on $\mathfrak{g}$ relative to the adjoint representation of $\mathfrak{g}$ on $\mathfrak{g}\otimes\mathfrak{g}$
\end{enumerate}
\end{definition}
\noindent Condition \ref{cond2} means that the 2-cocycle
\begin{equation}\label{eq: cc}
 ad_{\xi}(\delta(\eta))-ad_{\eta}(\delta(\xi))-\delta([\xi,\eta]) = 0
\end{equation}
for any $\xi,\eta\in\mathfrak{g}$.

In the following we will adopt the notation
\begin{equation}
 [x,y]_{*}={}^{t}\delta(x\otimes y),
\end{equation}
for any $x,y\in\mathfrak{g}^*$. Thus, by definition
\begin{equation}
 \langle[x,y]_{*},\xi\rangle=\langle\delta(\xi), x\otimes y\rangle
\end{equation}
for $\xi\in\mathfrak{g}$.

As discussed in the previous chapter any Lie algebra $\mathfrak{g}$ acts on itself by the adjoint representation $ad:\xi\in\mathfrak{g}\mapsto ad_{\xi}\in End\ \mathfrak{g}$, defined by 
$ad_{\xi}(\eta)=[\xi,\eta]$ (see eq. (\ref{eq: ad})). 

We now introduce the definition of coadjoint representation of a Lie algebra on the dual vector space.
Let $\mathfrak{g}$ be a Lie algebra and let $\mathfrak{g}^*$ be its dual vector space. For $\xi\in\mathfrak{g}$, we set

\begin{equation}
 ad_{\xi}^*=-{}^{t}(ad_{\xi}).
\end{equation}
Thus $ad_{\xi}^*$ is the endomorphism of $\mathfrak{g}^*$ satisfying
\begin{equation}
 \langle x,ad_{\xi}(y)\rangle=-\langle ad_{\xi}^*x,y\rangle.
\end{equation}
The map $\xi\in\mathfrak{g}\mapsto ad_{\xi}^*\in End\ \mathfrak{g}^*$ is a representation of $\mathfrak{g}$ in $\mathfrak{g}^*$, that we call coadjoint representation.
Hence, eq. (\ref{eq: cc}) can be written as

\begin{equation}
\begin{split}
 \langle[x,y]_{\mathfrak{g}^*},[\xi,\eta]\rangle & + \langle[ad_{\xi}^*x,y]_{*},y\rangle + \langle [x, ad_{\xi}^*y],y\rangle\\
& -\langle[ad_{\xi}^*x,y]_{*},\xi\rangle-\langle [x, ad_{\xi}^*y],\xi\rangle=0
\end{split}
\end{equation}
It is important to stress that there is a symmetry between $\mathfrak{g}$, with Lie bracket $[\cdot,\cdot],$ and $\mathfrak{g}^*$, with Lie bracket $[\cdot,\cdot]_{*}$, defined by $\delta$. In fact, setting

\begin{equation}
 ad_{x}(y)=[x,y]_{\mathfrak{g}^*}
\end{equation}
and
\begin{equation}
 \langle ad_{x}y, \xi\rangle=-\langle y, ad_{x}^*\xi\rangle,
\end{equation}
the map $x\in\mathfrak{g}^*\mapsto ad_{x}^*\in End\ \mathfrak{g}$ is the coadjoint representation of $\mathfrak{g}^*$ in the dual of $\mathfrak{g}^*$, which is isomorphic to $\mathfrak{g}$. Hence, eq. (\ref{eq: cc}) 
is equivalent to

\begin{equation}\label{eq: cc1}
\begin{split}
 \langle[x,y]_{\mathfrak{g}^*},[\xi,\eta]\rangle &+ \langle ad_{\xi}^*x,ad_{y}^*\eta\rangle -\langle ad_{\xi}^*y,ad_{x}^*\eta\rangle\\
 &-\langle ad_{\eta}^*x,ad_{y}^*\xi\rangle+ \langle ad_{\eta}^*y,ad_{x}^*\xi\rangle=0.
 \end{split}
\end{equation}
The symmetry between $\mathfrak{g}$ and $\mathfrak{g}^*$ then follows from the fact that eq. (\ref{eq: cc1}) is equivalent to the condition  on ${}^{t}[\cdot,\cdot]:\mathfrak{g}^*\rightarrow \mathfrak{g}^*\otimes \mathfrak{g}^*$ to be a 1-cocycle on $\mathfrak{g}^*$ with values on $\mathfrak{g}^*\otimes \mathfrak{g}^*$, where $\mathfrak{g}^*$ acts on $\mathfrak{g}^*\otimes \mathfrak{g}^*$ by the adjoint action.

\begin{proposition}
 If $(\mathfrak{g},\delta)$ is a Lie bialgebra, and $[\cdot,\cdot]$ is a Lie bracket on $\mathfrak{g}$, then $(\mathfrak{g}^*, {}^{t}\delta)$ is a Lie bialgebra, where ${}^{t}[\cdot,\cdot]$ defines a Lie bracket 
on $\mathfrak{g}^*$.
\end{proposition}

By definition, $(\mathfrak{g}^*, {}^{t}[\cdot,\cdot])$ is the \textbf{dual} of the Lie bialgebra $(\mathfrak{g},\delta)$. It is easy to see that the dual of $(\mathfrak{g}^*, {}^{t}[\cdot,\cdot])$ coincides with $(\mathfrak{g},\delta)$.

\begin{proposition}
 Let $(\mathfrak{g},\delta)$ be a Lie bialgebra with dual $(\mathfrak{g}^*, {}^{t}[,])$. There exists a unique Lie algebra structure on the vector space $\mathfrak{g}\oplus\mathfrak{g}^*$ such that
\begin{enumerate}
\item it restricts to the given brackets on $\mathfrak{g}$ and $\mathfrak{g}^*$
\item the scalar product  $\langle\cdot, \cdot\rangle$ on $\mathfrak{g}\oplus\mathfrak{g}^*$ is invariant.
\end{enumerate}
It is given by
\begin{equation}\label{eq: dbra}
\left[ \xi+x,\eta+y\right]=\left[ x,y\right]-ad^*_{\eta}x+ad^*_{\xi}y +\left[ \xi,\eta\right]+ad^*_{x}\eta-ad^*_{y}\xi .
\end{equation}
\end{proposition}
Moreover, the structure (\ref{eq: dbra}) is a Lie bracket on $\mathfrak{g}\oplus\mathfrak{g}^*$ if and only if $\mathfrak{g}$ is a Lie bialgebra.

\begin{definition}
The \textbf{double} $\mathfrak{d}=\mathfrak{g}\bowtie\mathfrak{g}^*$ of the Lie bialgebra $\mathfrak{g}$ is defined by the vector space $\mathfrak{g}\oplus\mathfrak{g}^*$ together with the Lie bracket  given by 
(\ref{eq: dbra}).
\end{definition}

Note that $\mathfrak{d}=\mathfrak{g}\bowtie\mathfrak{g}^*$ is also the double of $\mathfrak{g}^*$·

\subsection{Classical Yang-Baxter equation and r-matrices}
\label{subsec: classical yang baxter}

We now introduce a particular class of Lie bialgebra structures, given by a coboundary of an element $r\in \mathfrak{g}\otimes \mathfrak{g}$, called \textbf{r-matrix}. An $r$-matrix defines a cocycle $\delta$ as follows: 

\begin{equation}
 \delta(x)= ad_x(r)=[x\otimes 1+1\otimes x,r].
\end{equation}
To each element $r$ in $\mathfrak{g}\otimes \mathfrak{g}$, we associate the map $\underline{r}:\mathfrak{g}^*\rightarrow \mathfrak{g}$ defined by
\begin{equation}
 \underline{r}(\xi)(\eta)=r(\xi,\eta),
\end{equation}
for $\xi,\eta\in\mathfrak{g}^*$. When $\delta$ is determined by $r$ we write $[\xi,\eta]^r$ instead of $[\xi,\eta]_{*}$.
\begin{proposition}
If $r$ is skew-symmetric, then
\begin{equation}
 [\xi,\eta]^r=ad_{\underline{r}\xi}^*\eta-ad_{\underline{r}\eta}^*\xi.
\end{equation}
\end{proposition}
In order to show when the $r$-matrix defines a Lie bialgebra, we introduce the Schouten bracket of an element $r\in\mathfrak{g}\otimes \mathfrak{g}$ with itself, denoted by $[r,r]$.
It is the element of $\bigwedge^3 \mathfrak{g}$ defined by
\begin{equation}
 [r,r](\xi,\eta,\varsigma)=-2\circlearrowleft\langle\varsigma,[\underline{r}\xi,\underline{r}\eta],
\end{equation}
where $\circlearrowleft$ denotes the summation over the circular permutation of $\xi$, $\eta$ and $\varsigma$.

\begin{proposition}
 The $r$-matrix defines a Lie bracket on $\mathfrak{g}^*$ if and only if $[r,r]$ is ad-invariant.
\end{proposition}
The condition for $[r,r]$ to be $ad$-invariant is sometimes called generalized Yang-Baxter equation. 
\begin{definition}
 Let $r$ be an element of $\mathfrak{g}\otimes \mathfrak{g}$, with symmetric part $s$ and skew-symmetric part $a$. If $s$ and $[a,a]$ are ad-invariant, then $r$ is called classical $r$-matrix.
If $r$ is skew-symmetric ($r=a$) and if $[r,r]=0$, then $r$ is called a triangular $r$-matrix.
\end{definition}

Let us define the map $\langle\underline{r,r}\rangle:\bigwedge^2 \mathfrak{g}^*\rightarrow \mathfrak{g}$, given by
\begin{equation}
 \langle\underline{r,r}\rangle(\xi,\eta)=[\underline{r}\xi,\underline{r}\eta]-\underline{r}[\xi,\eta]^r.
\end{equation}
Setting
\begin{equation}
 \langle r,r\rangle(\xi,\eta,\varsigma)=\langle\varsigma,\langle\underline{r,r}\rangle(\xi,\eta)\rangle,
\end{equation}
the map $\langle\underline{r,r}\rangle$ is identified with an element $\langle r,r\rangle\in \bigwedge^2\mathfrak{g}\otimes \mathfrak{g}$.

\begin{theorem}
Let $\mathfrak{g}$ a finite dimensional Lie algebra.
\begin{enumerate}
  \item Let $a$ be in $\mathfrak{g}\otimes \mathfrak{g}$ and skew-symmetric. Then $\langle a,a\rangle$ is in $\bigwedge^3 \mathfrak{g}$, and
\begin{equation}
	\langle a,a\rangle= -\frac{1}{2}[a,a],
\end{equation}
  \item Let $s$ be in $\mathfrak{g}\otimes \mathfrak{g}$, symmetric and ad-invariant. Then $\langle s,s\rangle$ is an ad-invariant element in $\bigwedge^3 \mathfrak{g}$, and
\begin{equation}
 \langle\underline{s,s}\rangle(\xi,\eta)=[\underline{s}\xi,\underline{s}\eta],
\end{equation}
  \item For $r=s+a$, where $a$ is skew-symmetric and $s$ is symmetric and ad-invariant, $\langle r,r\rangle$ is in $\bigwedge^3 \mathfrak{g}$ and
\begin{equation}
 \langle r,r\rangle= \langle a,a\rangle+\langle s,s\rangle.
  \end{equation}
\end{enumerate}
\end{theorem}
From this theorem we obtain
\begin{corollary}
 Let $r=a+s$ where $a$ is skew-symmetric and $s$ is symmetric and ad-invariant. Then $[a,a]$ is ad-invariant if $\langle r,r\rangle=0$.
\end{corollary}
Thus an element $r$ in $\mathfrak{g}\otimes \mathfrak{g}$ with ad-invariant symmetric part, satisfying $\langle r,r\rangle=0$ is an $r$-matrix. The condition $\langle r,r\rangle=0$ is called classical Yang-Baxter equation.

\begin{definition}\label{eq: yb}
An $r$-matrix satisfying the classical Yang-Baxter equation is called quasi-triangular.
If the symmetric part is invertible, then $r$ is called factorisable.
\end{definition}

\subsection{Tensor notation}
\label{subsec: tensor notation}

Given $r\in \mathfrak{g}\otimes \mathfrak{g}$, we define $r_{12}$, $r_{13}$, $r_{23}$ as elements in the third power tensor of the enveloping algebra of $\mathfrak{g}$ (i.e. an associative algebra with unit such that $[x,y]=x\cdot y-y\cdot x$),
\begin{equation}
\begin{split}
 r_{12}&=r\otimes 1\\
 r_{23}&=1\otimes r.
\end{split}
\end{equation}
If $r$ = $\sum_i u_i\otimes v_i$, it is clear that $r_{13}=\sum_i u_i\otimes 1\otimes v_i$, where $1$ is the unit of the enveloping algebra.

Let us define
\begin{equation}
\begin{split}
[r_{12},r_{13}]&=[\sum_i u_i\otimes v_i\otimes 1, \sum_j u_j\otimes 1\otimes v_j]=\sum_{i,j}[u_i,u_j]\otimes v_i\otimes v_j,\\
[r_{12},r_{23}]&=[\sum_i u_i\otimes v_i\otimes 1, \sum_j 1\otimes u_j\otimes v_j]=\sum_{i,j}u_i\otimes [v_i,u_j]\otimes v_j,\\
[r_{13},r_{23}]&=[\sum_i u_i\otimes 1\otimes v_i, \sum_j 1\otimes u_j\otimes v_j]=\sum_{i,j}u_i\otimes u_j\otimes [v_i,v_j].
\end{split}
\end{equation}
in $\mathfrak{g}\otimes \mathfrak{g}\otimes\mathfrak{g}$. With this notation, if the symmetric part of $r$ is ad-invariant, we have
\begin{equation}
\langle r,r\rangle=[r_{12},r_{13}]+[r_{12},r_{23}]+[r_{13},r_{23}]
\end{equation}
and
\begin{equation}
\langle s,s\rangle=[s_{12},s_{13}]=[s_{12},s_{23}]=[s_{13},s_{23}].
\end{equation}
Hence, the classical Yang-Baxter equation reads
\begin{equation}
[r_{12},r_{13}]+[r_{12},r_{23}]+[r_{13},r_{23}]=0.
\end{equation}

\begin{example}\label{ex: axb}
Let $\mathfrak{g}$ be the 2-dimensional Lie algebra with basis $X$, $Y$ and Lie bracket
\begin{equation}
[X,Y]=X.
\end{equation}
Then $r=X\wedge Y=X\otimes Y-Y\otimes X$ is a skew-symmetric solution of the classical Yang-Baxter equation. As a consequence, $\delta(X)=ad_X(r)=0$ and $\delta(Y)=ad_Y(r)=-X\wedge Y$. In terms of dual basis 
$X^*$, $Y^*$ of $\mathfrak{g}^*$, $[X^*,Y^*]^r=-Y^*$.
\end{example}

\begin{example}\label{eq: sl2r}
Let $\mathfrak{g}=\mathfrak{sl}(2,\mathbb{C})$ and consider the Casimir element
\begin{equation}
t=\frac{1}{8}H\otimes H+\frac{1}{4}(X\otimes Y+Y\otimes X).
\end{equation}
We set
\begin{equation}
t_0=\frac{1}{8}H\otimes H,\quad t_{\pm}=\frac{1}{4}X\otimes Y
\end{equation}
and we define
\begin{equation}
r=t_0+2t_{\pm}=\frac{1}{8}(H\otimes H+4X\otimes Y).
\end{equation}
Then the symmetric part of $r$ is $t$ and the skew symmetric part is $a=\frac{1}{4}X\wedge Y$, and $r$ is a factorisable $r$-matrix. Then $\delta (H)=ad_H(a)=0$, $\delta (X)=ad_X(a)=\frac{1}{4}X\wedge H$ and 
$\delta (Y)=ad_Y(a)=\frac{1}{4}Y\wedge H$. In terms of the dual basis $H^*$, $X^*$, $Y^*$ of $\mathfrak{g}^*$,
\begin{equation}
[H^*,X^*]^r=\frac{1}{4}X^*,\quad [H^*,Y^*]^r=\frac{1}{4}Y^*,\quad [X^*,Y^*]^r=0.
\end{equation}
\end{example}

\begin{example}\label{ex: sl2rt}
On $\mathfrak{sl}(2,\mathbb{C})$ we can consider also $r=X\otimes H-H\otimes X$, which is a triangular $r$-matrix. Then $\delta (X)=0$, $\delta (Y)=2Y\wedge X$ and $\delta (H)=X\wedge H$.
\end{example}

\section{Poisson manifolds}
\label{sec: poisson manifolds}

As discussed in the previous chapter, it is always possible to define a Poisson structure on a symplectic manifold, which give rise to Poisson brackets on the space of smooth functions on the manifold.
In the following, we rephrase the definition of Poisson manifold given above in a more convenient way \cite{V} and we discuss the theory of symplectic foliation \cite{CW},\cite{We1}.

Let $\pi$ be a bivector on a manifold $M$, i.e. a skew-symmetric, contravariant 2-tensor. At each point $m$, $\pi(m)$ can be viewed as a skew-symmetric bilinear form on $T_m^*M$, or as the skew-symmetric linear map 
$\pi^{\sharp}(m): T_m^*M\rightarrow T_mM$, such that

\begin{equation}\label{eq: pish}
  \pi (m)(\alpha_m,\beta_m)=\pi^{\sharp}(\alpha_m)(\beta_m), \quad \alpha_m,\beta_m\in T_m^*M.
\end{equation}

If $\alpha$, $\beta$ are 1-forms on $M$, we define $\pi(\alpha,\beta)$ to be the function in $C^{\infty}(M)$ whose value at $m$ is $\pi(m)(\alpha_m,\beta_m)$. Given $f,g\in C^{\infty}(M)$ we set

\begin{equation}\label{eq: pbra}
  \pi(m)(df,dg)=\lbrace f,g\rbrace(m).
\end{equation}
Note that $\pi^{\sharp}(df)$ is the Hamiltonian vector field defined in (\ref{eq: hvf}). It is clear that the bracket induced by $\pi$ satisfies the Leibniz rule.

\begin{definition}
 A \textbf{Poisson manifold} $(M,\pi)$ is a manifold $M$ with a Poisson bivector $\pi$ such that the bracket defined in eq. (\ref{eq: pbra}) satisfies the Jacobi identity.
\end{definition}

\begin{example}
 If $M=\mathbb{R}^{2n}$, with coordinates $(q^i, p_i)$, $i=1,\cdots, n$ and if
\begin{equation}
 \pi^{\sharp}(dq^i)=-\frac{\partial}{\partial q^i}, \quad  \pi^{\sharp}(dp_i)=-\frac{\partial}{\partial p_i},
\end{equation}
then
\begin{equation}
 X_f=\frac{\partial f}{\partial p_i}\frac{\partial}{\partial q^i}-\frac{\partial f}{\partial q^i}\frac{\partial}{\partial p_i}
\end{equation}
and
\begin{equation}
\lbrace f,g\rbrace=\frac{\partial f}{\partial p_i}\frac{\partial g}{\partial q^i}-\frac{\partial f}{\partial q^i}\frac{\partial g}{\partial p_i},
\end{equation}
is the standard Poisson bracket of functions on the phase space. The corresponding bivector is $\pi=\frac{\partial }{\partial p_i}\wedge \frac{\partial}{\partial q^i}$.
\end{example}
In local coordinates, a bivector $\pi$ is a Poisson bivector if and only if
\begin{equation}
 \pi^{hi}\partial_h \pi^{jk}+\pi^{hj}\partial_h \pi^{ki}+\pi^{hk}\partial_h \pi^{ij}=0.
\end{equation}
In terms of bivector field $\pi$ we have the following characterization of a Poisson bivector:
\begin{proposition}\label{pro: sn}
The bivector field $\pi\in\wedge^2 TM$ is a Poisson bivector if and only if $[\pi,\pi]_S=0$, where $[\cdot,\cdot]_S$ is the Schouten-Nijenhuis bracket.
\end{proposition}

\begin{definition}
A mapping $\phi: (M_1,\pi_1)\rightarrow (M_2,\pi_2)$ between two Poisson manifolds is called a \textbf{Poisson mapping} if $\forall f,g\in C^{\infty}(M_2)$ one has
\begin{equation}
\lbrace f\circ\phi, g\circ\phi\rbrace_1=\lbrace f,g\rbrace_2\circ \phi
\end{equation}
\end{definition}

\begin{definition}
Given two Poisson manifolds $(M_1,\pi_1)$ and $(M_2,\pi_2)$, the pair $(M_1\times M_2,\pi)$, is a Poisson manifold, called \textbf{Poisson product}, with $\pi=\pi_1\oplus\pi_2$.
\end{definition}

An interesting feature of Poisson manifolds is the existence of the differential calculus of forms, which can be resumed in the following result:
\begin{theorem}[\cite{V}]\label{thm: bf}
 Let $(M,\pi)$ be a Poisson manifold. Then there exists a unique bilinear, skew-symmetric operation $[\cdot,\cdot]_{\pi}: \Omega^1(M)\times \Omega^1(M)\rightarrow \Omega^1(M)$ such that
\begin{equation}
\begin{split}
 [df,dg]_{\pi}&=d[f,g]_{\pi}, \qquad f,g\in C^{\infty}(M),\\
 [\alpha,f\beta]_{\pi}&=f[\alpha,\beta]_{\pi}+(\iota_{\pi^{\sharp}(\alpha)}f) \beta  \qquad f\in C^{\infty}(M), \alpha,\beta\in \Omega^1(M).
\end{split}
\end{equation}
This operation is given by the general formula
\begin{equation}\label{eq: bff}
 [\alpha,\beta]_{\pi}=\mc{L}_{\pi\sh(\alpha)}\beta-\mc{L}_{\pi\sh(\beta)}\alpha-d(\pi(\alpha,\beta))
\end{equation}
Furthermore, it provides $\Omega^1(M)$ with a Lie algebra structure such that $\pi\sh:T\st M\ra TM$ is a Lie algebra homomorphism.
\end{theorem}

\subsection{Symplectic Foliation}
\label{subsec: symplectic foliation}
As mentioned above, every symplectic manifold admits a Poisson structure, but the converse does not hold. In fact, any Poisson manifold can be seen as a union of symplectic manifolds called symplectic leaves.

Locally, the symplectic foliation of $(M,\pi)$ can be described in terms of coordinates. 
More precisely, the local structure of a Poisson manifold at $O\in M$ is described by the Splitting Theorem \cite{We1}:

\begin{theorem}[Weinstein]\label{thm: split}
On a Poisson manifold $(M,\pi)$, any point $O\in M$ has a coordinate neighborhood with coordinates
$(q_1,\dots,q_k,p_1,\dots,p_k,\allowbreak y_1,\dots,y_l)$ centered at $O$, such that
\begin{equation}\label{eq: splitp}
\pi=\sum_i \frac{\partial}{\partial q_i}\wedge\frac{\partial}{\partial p_i}+\frac{1}{2}\sum_{i,j}\phi_{ij}(y) \frac{\partial}{\partial y_i}\wedge\frac{\partial}{\partial y_j}\quad \phi_{ij}(0)=0.
\end{equation}
The rank of $\pi$ at $O$ is $2k$. Since $\phi$ depends only on the $y_i$’s, this theorem gives
a decomposition of the neighborhood of $O$ as a product of two Poisson manifolds: one with rank $2k$, and the other with rank 0 at $O$.
\end{theorem}
 For any point $O$ of the Poisson manifold, if $(q, p, y)$ are normal
coordinates as in the previous theorem, then the symplectic leaf through $O$ is given locally
by the equation $y = 0$.
Hence, for any point $m\in M$, we have a symplectic leaf through it. Locally, this leaf has canonical
coordinates $(q_1,\dots ,q_k, p_1,\dots , p_k)$, where the bracket is given by canonical symplectic relations. Notice that the symplectic leaf is well-defined, but each choice of coordinates $(y_1,\dots y_l)$ in 
Theorem \ref{thm: split} gives rise to a different term

\begin{equation}
\frac{1}{2}\sum_{i,j}\phi_{ij}(y)\frac{\partial}{\partial y_i}\wedge \frac{\partial}{\partial y_j}
\end{equation}
called the \textbf{transverse Poisson structure} of dimension $l$.
The transverse structures are not uniquely defined, but they are all isomorphic. Locally, the transverse structure is determined by the structure functions $\pi_{ij}(y)=\lbrace y_i,y_j\rbrace$ (which vanishes at $y=0$).

Given a Poisson manifold the Casimir functions are functions which are constant on the symplectic leaves.

\section{Poisson Lie groups}
\label{sec: poisson lie groups}

Poisson Lie groups are a particular class of Poisson manifolds. Namely, a Poisson Lie group is a Poisson manifold that is also a
 Lie group and its corresponding infinitesimal object is a Lie bialgebra. In particular we discuss Poisson Lie groups defined by
  $r$-matrices and some basic examples.

Recall that, given a Lie group $G$, the left and right translations by an element $g\in G$ are defined by
\begin{equation}
\lambda_g(h)=gh,\quad \rho_g(h)=hg,
\end{equation}
for $h\in G$.

\begin{definition}
	A \textbf{Poisson Lie group} $(G,\pi_G)$ is a Lie group equipped with a multiplicative Poisson structure $\pi_G$.
\end{definition}

In terms of the Poisson tensor $\pi_G$, the Poisson structure is multiplicative if and only if

\begin{equation}\label{eq: pig}
\pi_G (gh)=\lambda_g\pi_G(h)+\rho_h\pi_G(g),\qquad \forall g,h\in G,
\end{equation}
This is equivalent to the following condition (see \cite{YK})

\begin{equation}
\lbrace\varphi\circ\lambda_g,\psi\circ\lambda_g\rbrace(h)+\lbrace\varphi\circ\rho_h,\psi\circ\rho_h\rbrace(g)=\lbrace\varphi,\psi\rbrace(gh),
\end{equation}
for all functions $\varphi,\psi$ on $G$, and for all $g,h$ in $G$. This condition means that the multiplication map $G\times G\rightarrow G$ is a Poisson map.
 Note that a nonzero multiplicative Poisson structure is in general neither left nor right invariant. In fact the left or right translation by $g\in G$ preserves $\pi_G$ if and only if $\pi_G(g)=0$.

\begin{example}
If $\pi_G=0$, it is obviously multiplicative, hence any Lie group $G$ with the trivial Poisson structure is a Poisson Lie group. The direct product of two Poisson Lie groups is again a Poisson Lie group.
\end{example}
A very important class of Poisson Lie groups arise from the $r$-matrix formalism. Let $r\in \mathfrak{g}\wedge \mathfrak{g}$. Define a bivector $\pi_G$ on $G$ by

\begin{equation}\label{eq: pir}
\pi_G(g)=\lambda_g r-\rho_g r\qquad \forall g\in G
\end{equation}
We see that this bivector is multiplicative. In particular we have:

\begin{theorem}[Drinfeld]
Let $G$ be a connected Lie group with Lie algebra $\mathfrak{g}$. Let $r\in \mathfrak{g}\wedge \mathfrak{g}$. Define a bivector field on $G$ by eq. (\ref{eq: pir}).
Then $(G,\pi_G)$ is a Poisson Lie group if and only if $\left[ r, r\right] \in \mathfrak{g}\wedge \mathfrak{g}\wedge \mathfrak{g}$ is Ad-invariant. In particular, when $r$ satisfies the Yang-Baxter equation,
it defines a multiplicative Poisson structure on $G$.
\end{theorem}	
The Poisson structure $\pi$ vanishes at the identity $e\in G$, and its linearization at $e$ is given by $d_e\pi:\mathfrak{g}\rightarrow\mathfrak{g}\wedge\mathfrak{g}$. The map $d_e\pi$ is a derivative and its dual map
$\left[ \cdot,\cdot\right]_{*}: \mathfrak{g^*}\wedge\mathfrak{g^*}\rightarrow\mathfrak{g^*}$ is given by $\left[ x,y\right]_{*}=d_e(\pi_G(\bar{x},\bar{y}))$, where $x,y\in\mathfrak{g}^*$ and $\bar{x}$ and $\bar{y}$
can be any 1-forms on $G$ with $\bar{x}(e)=x$ and $\bar{y}(e)=y$.
	
When $\pi_G$ is a Poisson bivector, $\left[ \cdot,\cdot\right]_{*}$ satisfies the Jacobi identity, so it makes $\mathfrak{g^*}$ into a Lie algebra. The Lie algebra $(\mathfrak{g^*},\left[ \cdot,\cdot\right]_{*})$ is just
the linearization of the Poisson structure at $e$.

\begin{theorem}
A multiplicative bivector field $\pi_G$ on a connected Lie group $G$ is Poisson if and only if its linearization at $e$ defines a Lie bracket on $\mathfrak{g^*}$.
\end{theorem}

The relation between Poisson Lie groups and Lie bialgebras is given by the following theorem:
\begin{theorem}[Drinfeld]\label{thm: dr}
If $(G,\pi_G)$ is a Poisson Lie group, then the linearization of $\pi_G$ at $e$ defines a Lie algebra structure on $\mathfrak{g}^*$ such that $(\mathfrak{g},\mathfrak{g}^*)$ form a Lie bialgebra over $\mathfrak{g}$,
called the tangent Lie bialgebra to $(G,\pi_G)$. Conversely, if $G$ is connected and simply connected, then every Lie bialgebra $(\mathfrak{g},\mathfrak{g}^*)$ over $\mathfrak{g}$ defines a unique multiplicative Poisson
structure $\pi_G$ on $G$ such that $(\mathfrak{g},\mathfrak{g}^*)$ is the tangent Lie bialgebra to the Poisson Lie group $(G,\pi_G)$.
\end{theorem}

It is important to emphasize that, given a connected Lie group $G$ with Lie algebra $\mathfrak{g}$, for an arbitrary 1-cocycle $\delta$ on $\mathfrak{g}$, there is no general way to integrate $\delta$ to the 1-cocycle 
$\epsilon$ on $G$ such that $d_e\epsilon=\delta$. In the case of a Lie bialgebra $(\mathfrak{g},\mathfrak{g}^*,\delta)$ the integration can be reduced to integrating Lie algebras to Lie groups and Lie algebra homomorphisms 
to Lie groups homomorphisms.

\subsection{Examples}
\label{subsec: examples}

\paragraph{Quasi-triangular structure of $SL(2, \mathbb{R})$}
Let $G=SL(2,\mathbb{R})$ be the group of real $2\times 2$ matrices with determinant 1. We consider $r=\frac{1}{8}(H\otimes H+4 X\otimes Y)$, as seen in Example \ref{eq: sl2r}, as element in 
$\mathfrak{sl}(2,\mathbb{R})\otimes\mathfrak{sl}(2,\mathbb{R})$, with skew-symmetric part $r_0=\frac{1}{4}(X \otimes Y - Y\otimes X)$. Then, given a generic element 
$g=\left( \begin{matrix} a & b \\ c & d \end{matrix}\right) \in G$ and using (\ref{eq: pir}) for $r_0$ we get the quadratic Poisson brackets,

\begin{align} 
\lbrace a,b\rbrace & =\frac{1}{4}ab&  \lbrace a,c\rbrace & =\frac{1}{4}ac &   \lbrace a,d\rbrace & =\frac{1}{2}bc\\
\lbrace b,c\rbrace &=0 & \lbrace b,d\rbrace &=\frac{1}{4}bd& \lbrace c,d\rbrace &=\frac{1}{4}cd.
\end{align}

It is easy to check that $ad-bc$ is a Casimir element for this Poisson structure, which is indeed defined on $SL(2,\mathbb{C})$.

\paragraph{Triangular structure of $SL(2, \mathbb{R})$}

Consider now the triangular $r$-matrix defined in Example \ref{ex: sl2rt},
\begin{equation}
r=X\otimes H-H\otimes X=\left( \begin{matrix} 0 & -1 & 1 & 0 \\ 0 & 0 & 0 & -1\\ 0 & 0 & 0 & 1\\ 0 & 0 & 0 & 0 \end{matrix}\right).
\end{equation}
We find
\begin{align} 
\lbrace a,b\rbrace &=1-a^2 &   \lbrace a,c\rbrace & =c^2 &   \lbrace a,d\rbrace & =c(-a+d)\\
\lbrace b,c\rbrace &=c(a+d)& \lbrace b,d\rbrace &=d^2-1 & \lbrace c,d\rbrace &=-c^2.
\end{align}
Also in this case $ad-bc$ is a Casimir function.

\paragraph{Quasi-triangular structure of $SU(2)$}\label{sub: su2}
Let $G=SU(2)$ and $\mathfrak{g}=\mathfrak{su}(2)$. Let

\begin{equation}
e_1=\frac{1}{2}\left( \begin{matrix} i & 0 \\ 0 & -i \end{matrix}\right) \quad e_2=\frac{1}{2}\left( \begin{matrix} 0 & 1 \\ -1 & 0 \end{matrix}\right) \quad e_3=\frac{1}{2}\left( \begin{matrix}0 & i \\ i & 0 \end{matrix}\right) .
\end{equation}
Then ${e_1,e_2,e_3}$ is a basis for $\mathfrak{g}$, and $[e_1,e_2]=e_3$, $[e_2,e_3]=e_1$, $[e_3,e_1]=e_2$. Any $r\in \mathfrak{g}\wedge\mathfrak{g}$ is such that $\left[ r,r\right]$ is $Ad_G$-invariant. 

Let $r=2(e_2\wedge e_3)$ and define the Poisson structure on $SU(2)$ by
\begin{equation}
\pi (g)=2\left( \rho_g(e_2\wedge e_3)-\lambda_g(e_2\wedge e_3)\right).
\end{equation}
Hence, given a generic element $g\in SU(2)$ as $g=\left( \begin{matrix} a & b \\ c & d, \end{matrix}\right)$, the Poisson brackets are given by
\begin{align}
\left\lbrace a,b\right\rbrace &=iab & \left\lbrace a,c\right\rbrace &=iac & \left\lbrace a,d\right\rbrace &=2ibc\\
\left\lbrace b,c\right\rbrace &=0 & \left\lbrace b,d\right\rbrace &=ibd & \left\lbrace c,d\right\rbrace &=icd.
\end{align}
Finally, the Lie bracket defined by $r$ on $\mathfrak{su}(2)^*$ is given by
\begin{equation}
[e_1^*,e_2^*]=e_2^*,\quad [e_1^*,e_3^*]=e_3^*,\quad [e_2^*,e_3^*]=0.
\end{equation}

\subsection{The dual of a Poisson Lie group}
\label{sec: dual}
In Section \ref{sec:lie bialgebras} we introduced the dual and the double of a Lie bialgebra, so now we discuss the corresponding constructions at the level of the Lie group.

Given a Poisson Lie group $(G,\pi_G)$, we consider its Lie bialgebra $\mathfrak{g}$ whose 1-cocycle is $\delta =d_e\pi_G$. Let us denote the dual Lie bialgebra by $(\mathfrak{g}^*,\delta)$. 
By Theorem \ref{thm: dr} we know that there is a unique connected and simply connected Poisson Lie group $(G^*,\pi_{G\st})$, called the \textbf{dual} of $(G,\pi_G)$, associated to the Lie bialgebra $(\mathfrak{g}^*,\delta)$. If $G$ is connected and simply connected, then the dual of $G^*$ is $G$.

When $(G,\pi_G)$ is a Poisson Lie group, its Lie bialgebra $(\mathfrak{g},\delta)$ has a double $\mathfrak{d}$. The connected and simply connected Lie group $\mathcal{D}$ with Lie algebra $\mathfrak{d}$ is 
called the \textbf{double} of $(G,\pi_G)$. Since $\mathfrak{d}$ is a factorisable Lie bialgebra with $r$-matrix $r_{\mathfrak{d}}$, $\mathcal{D}$ is a factorisable Poisson Lie group with Poisson structure

\begin{equation}
 \pi_{\mathcal{D}}(d)=\lambda_d r_{\mathfrak{d}}-\rho_d r_{\mathfrak{d}}
\end{equation}
where $d\in\mathcal{D}$.

\begin{example}
Let $G=SU(2)$ be equipped with the Poisson structure given in  Example \ref{sub: su2}. Then $G^*$ can be identified with $SB(2,\mathbb{C})$. The double Lie algebra $\mathfrak{d}=\mathfrak{g}\bowtie\mathfrak{g}^*$ 
is the Lie algebra $\mathfrak{sl}(2,\mathbb{C})$ considered as a real Lie algebra. The decomposition $\mathfrak{sl}(2,\mathbb{C})=\mathfrak{su}(2) \oplus \mathfrak{sb}(2,\mathbb{C})$ is the well-known Gram-Schmidt decomposition.
\end{example}

\begin{example}
Consider the Lie bialgebra $\mathfrak{g}=ax+b$ of Example \ref{ex: axb}. A matrix representation of $\mathfrak{g}$ is the Lie algebra $\mathfrak{gl}(2, \mathbb{R})$ via
\begin{equation}
X=\left( \begin{matrix} 1 & 0 \\ 0 & 0 \end{matrix}\right),\quad Y=\left( \begin{matrix} 0 & 1 \\ 0 & 0 \end{matrix}\right),\quad X^*=\left( \begin{matrix} 0 & 0 \\ 0 & 1 \end{matrix}\right),\quad Y^*=\left( \begin{matrix} 0 & 0 \\ 1 & 0 \end{matrix}\right)
\end{equation}
with metric
\begin{equation}
\gamma(A,B)=tr(AJBJ), \quad\text{where}\quad J=\left( \begin{matrix} 0 & 1 \\ 1 & 0 \end{matrix}\right).
\end{equation}
The subgroups $G$ and $G^*$ of the Lie group $GL^+(2, \mathbb{R})$ of matrices with determinant $>0$ are given by
\begin{equation}
G=\left\lbrace \left( \begin{matrix} x & y \\ 0 & 1 \end{matrix}\right): x>0\right\rbrace\quad G^*=\left\lbrace \left( \begin{matrix} 1 & 0 \\ a & b \end{matrix}\right): b>0\right\rbrace.
\end{equation}
The Poisson bivector on $GL^+(2, \mathbb{R})$ in the coordinates $\left( \begin{matrix} x & y \\ a & b \end{matrix}\right)$ reads
\begin{equation}
	\pi=xy\partial_x\wedge \partial_y+ab\partial_a\wedge\partial_b+xb(\partial_x\wedge\partial_b+\partial_a\wedge\partial_y).
\end{equation}
It is degenerate at points with $xb=0$ and vanishes at $x=b=0$.
\end{example}

\section{Poisson actions and Momentum maps}

A Poisson action is a key concept for the generalization of the theory of momentum map, since it generalizes the canonical action discussed in the previous chapter. 

Recall, from Definition \ref{def: can}, that a canonical action of a Lie group $G$ on a Poisson manifold $M$ is defined as a group action which preserves the Poisson structure. Instead, a Poisson action is an action of a Poisson Lie group on a Poisson manifold satisfying a different property of compatibility between the Poisson bivectors of both manifolds. When the Poisson structure is trivial we recover the canonical actions.
 
In the following we always assume that $G$ is connected and simply connected, such that Theorem \ref{thm: dr} holds.

\begin{definition}
 The action of $(G,\pi_G)$ on $(M,\pi)$ is called \textbf{Poisson action} if the map $\Phi:G\times M\rightarrow M$ is Poisson, where $G\times M$ is a Poisson  product with structure $\pi_G\oplus\pi$
\end{definition}
If $(G,\pi_G)$ is a Poisson Lie group, the left and right actions of $G$ on itself are Poisson actions. From the previous chapter, we have that, given an action $\Phi$ of $G$ on $M$, the infinitesimal generator $\xi_M$ associated to $\xi \in \mathfrak{g}$ is the vector field on $M$ defined by
\begin{equation}
\xi_M(m):=\left.\frac{d}{dt}\right|_{t=0} \Phi_{\exp (- t\xi)}(m).
\end{equation}
This defines an action of $\mathfrak{g}$ on $M$ by $\xi\in\mathfrak{g}\mapsto \xi_M$, in fact we have
\begin{equation}
[\xi,\eta]_M=[\xi_M,\eta_M],\quad \text{for}\quad \xi,\eta\in\mathfrak{g}.
\end{equation}
Note that if $G$ carries the zero Poisson structure $\pi_G=0$, the action is Poisson if and only if it preserves $\pi$. In general, when $\pi_G\neq 0$, the structure $\pi$ is not invariant with respect to the action.

\begin{proposition}
Assume that $(G,\pi_G)$ is a connected Poisson Lie group with associate 1-cocycle of $\mathfrak{g}$
\begin{equation}
\delta=d_e\pi_G:\mathfrak{g}\rightarrow \wedge^2 \mathfrak{g},
\end{equation}
and let $(M,\pi)$ be a Poisson manifold. The action $\Phi: G\times M\rightarrow M$ is a Poisson action if and
only if
\begin{equation}\label{eq: pa}
\mathcal{L}_{\xi_M}(\pi)=-(\delta(\xi))_M
\end{equation}
 for any $\xi\in\mathfrak{g}$, where $\mathcal{L}$ denotes the Lie derivative.
\end{proposition}

\begin{definition}
A Lie algebra action $\xi\mapsto \xi_M$ is called an \textbf{infinitesimal Poisson action} of the Lie bialgebra $(\mathfrak{g},\delta)$ on $(M,\pi)$ if it satisfies eq. (\ref{eq: pa}).
\end{definition}

In this formalism the definition of momentum map reads (Lu, \cite{Lu3}, \cite{Lu1}):

\begin{definition}\label{def: mm}
A \textbf{momentum map} for the Poisson action $\Phi:G\times M\rightarrow M$ is a map $\boldsymbol{\mu}: M\rightarrow G^*$ such that
\begin{equation}\label{eq: mmp}
\xi_M=\pi^{\sharp}(\boldsymbol{\mu}^*(\theta_{\xi}))
\end{equation}
where $\theta_{\xi}$ is the left invariant 1-form on $G^*$ defined by the element $\xi\in\mathfrak{g}=(T_eG^*)^*$ and $\boldsymbol{\mu}^*$ is the cotangent lift $T^* G^*\rightarrow T^*M$.
\end{definition}
If $G$ has trivial Poisson structure, then $G^*=\mathfrak{g}^*$, the differential 1-form $\theta_{\xi}$ is the constant 1-form $\xi$ on $\mathfrak{g}^*$, and
\begin{equation}
\boldsymbol{\mu}^*(\theta_{\xi})=d(\boldsymbol{\mu}^{\xi}),\quad\text{where}\quad \boldsymbol{\mu}^{\xi}(m)=\langle\boldsymbol{\mu}(m),\xi\rangle.
\end{equation}
Thus, in this case, we recover the usual definition of a momentum map for a Hamiltonian action $\boldsymbol{\mu}: M\rightarrow \mathfrak{g}^*$, that is
\begin{equation}
\xi_M=\pi^{\sharp}(d(\boldsymbol{\mu}^{\xi})).
\end{equation}
In other words, $\xi_M$ is the Hamiltonian vector field with Hamiltonian $\boldsymbol{\mu}^{\xi}\in C^{\infty}(M)$.

When $\pi_G$ is not trivial, $\theta_{\xi}$ is a Maurer-Cartan form, hence $\boldsymbol{\mu}^*(\theta_{\xi})$ can not be written as a differential of a Hamiltonian function.

\subsection{Dressing Transformations}\label{sec: dressing}

One of the most important example of Poisson action is the dressing action of $G$ on $G^*$. Consider a Poisson Lie group $(G,\pi_G)$, its dual $(G^*,\pi_{G^*})$ and its double $\mathcal{D}$, with Lie algebras 
$\mathfrak{g}$, $\mathfrak{g}^*$ and $\mathfrak{d}$, respectively.

\begin{theorem}\label{thm: drac}
Let $l(\xi)$ the vector field on $G^*$ defined by
\begin{equation}\label{eq: idr}
	l(\xi)=\pi_{G\st}\sh(\theta_{\xi})
\end{equation}
for each $\xi\in\mathfrak{g}$. Here $\theta_{\xi}$ is the left invariant 1-form on $G^*$ defined by $\xi\in\mathfrak{g}=(T_eG^*)^*$. Then
\begin{enumerate}[1.]
		\item The map $\xi\mapsto l(\xi)$ is an action of $\mathfrak{g}$ on $G^*$, whose linearization at $e$ is the coadjoint action of $\mathfrak{g}$ on $\mathfrak{g}^*$.
		\item The action $\xi\mapsto l(\xi)$ is an infinitesimal Poisson action of the Lie bialgebra $\mathfrak{g}$ on the Poisson Lie group $G^*$.
	\end{enumerate}
\end{theorem}
A proof of this theorem can be found in \cite{YK}. 

The action defined in (\ref{eq: idr}) is generally called left infinitesimal dressing action of $\mathfrak{g}$ on $G^*$. In
particular, when $G$ is a trivial Poisson Lie group, its dual group $G^*$ is the Abelian group $\mathfrak{g}^*$, and the left infinitesimal dressing action of $\mathfrak{g}$ on $\mathfrak{g}^*$ is given by the linear  vector fields $l(\xi): \eta\in\mathfrak{g}^* \mapsto -ad^*_{\xi}\eta\in\mathfrak{g}^*$, for each $\xi\in\mathfrak{g}$. It is easy to see   that $\xi\mapsto l(\xi)$ is a Lie algebra homomorphism from $\mathfrak{g}$ to the Lie algebra of linear vector fields on $\mathfrak{g}^*$, in fact we find $l([\xi,\eta])=l(\xi)l(\eta)-l(\eta)l(\xi)$.

Similarly, the right infinitesimal dressing action of $\mathfrak{g}$ on $G^*$ is defined by
\begin{equation}
r(\xi)=-\pi_{G\st}\sh(\theta_{\xi})
\end{equation}
where $\theta_{\xi}$ is the right invariant 1-form on $G^*$.

Let $l(\xi)$ (resp. $r(\xi)$) a left (resp. right) dressing vector field on $G^*$. If all the dressing vector fields are complete, we can integrate the $\mathfrak{g}$-action into a Poisson $G$-action on $G^*$ called the
\textbf{dressing action} and we say that the dressing actions consist of dressing transformations.

\begin{proposition}
The symplectic leaves of $G$ (resp. $G^*$) are the connected components of the orbits of the right or left dressing action of $G^*$ (resp. $G$).
\end{proposition}
The momentum map for the dressing action of $G$ on $G^*$ is the opposite of the identity map from $G^*$ to itself.

For the Poisson Lie group $G\st$
we identify $\mk{g}$ with the space of left invariant 1-forms on $G\st$. Then, this space is closed under the bracket defined by $\pi_{G\st}$ and the induced bracket on $\mk{g}$, by the above identification,
coincides with the original Lie bracket on $\mk{g}$ (see \cite{We}).
Given $\xi\in\mathfrak{g}$, we denote by $\theta_{\xi}$ the left invariant form on 
$G\st$ whose value at identity is $\xi$. The basic property of $\theta$'s is the Maurer-Cartan equation for $G\st$:
\begin{equation}\label{eq: mct}
    d\theta_{\xi}+\frac{1}{2}\theta\wedge\theta\circ\delta(\xi)=0
\end{equation}
In the following proposition we prove new relations satisfied by $\theta$:
\begin{proposition}
Let $\theta_{\xi},\theta_{\eta}$ be two left invariant 1-forms on $G\st$, such that $\theta_\xi(e)=\xi$, $\theta_{\eta}(e)=\eta$ then
\begin{equation}\label{eq: theta}
\theta_{[\xi,\eta]}=[\theta_{\xi},\theta_{\eta}]_{\pi_{G\st}}
\end{equation}
and
\begin{equation}\label{eq: theta2}
 \mc{L}_X\pi_{G\st}(\theta_{\xi},\theta_{\eta})=x([\xi,\eta])+\pi_{G^*}(\theta_{ad^*_x \xi},\theta_{\eta})+\pi_{G^*}(\theta_{\xi},\theta_{ad^*_x\eta})
\end{equation}

\end{proposition}

\begin{proof}
First, we prove that $[\theta_{\xi},\theta_{\eta}]_{\pi_{G\st}}$ is a left-invariant 1-form. Let us consider and element $x\in\mk{g}\st$ and the correspondent left invariant vector field $X\in TG\st$. 
We contract $X$ with the bracket $[\theta_{\xi},\theta_{\eta}]_{\pi_{G\st}}$ to show that we obtain a constant. More precisely, we contract $X$ term by term with the equation (\ref{eq: bff}) and we get
\begin{equation}\label{eq: xc}
\begin{split}
 \iota_X d\pi_{G\st}(\theta_{\xi},\theta_{\eta})&=(\mc{L}_X\pi_{G\st})(\theta_{\xi},\theta_{\eta})+\pi_{G\st}(\mc{L}_X\theta_{\xi},\theta_{\eta})+\pi_{G\st}(\theta_{\xi},\theta_{\eta})\\
     \iota_X \mc{L}_{\pi\sh_{G\st}(\theta_{\xi})}\theta_{\eta}&=(\mc{L}_X\pi_{G\st})(\theta_{\eta},\theta_{\xi})-\pi_{G\st}(\mc{L}_X\theta_{\xi},\theta_{\eta})\\
      \iota_X \mc{L}_{\pi\sh_{G\st}(\theta_{\eta})}\theta_{\xi}&=(\mc{L}_X\pi_{G\st})(\theta_{\xi},\theta_{\eta})-\pi_{G\st}(\theta_{\xi},\theta_{\eta})
\end{split}
\end{equation}
Then, substituting the relations (\ref{eq: xc}) in (\ref{eq: bff}) we obtain
\begin{equation}
\iota_X [\theta_{\xi},\theta_{\eta}]_{\pi_{G\st}}=(\mc{L}_X\pi_{G\st})(\theta_{\xi},\theta_{\eta}).
\end{equation}
Since $\mc{L}_X\pi_{G\st}(e)={}^{t}\delta(x)$, eq. (\ref{eq: theta}) is proved. Moreover, we have
 \begin{equation}
\begin{split}
\mc{L}_X\pi_{G\st}(\theta_{\xi},\theta_{\eta})&=(\mc{L}_X\pi_{G\st})(\theta_{\xi},\theta_{\eta})+\pi_{G\st}(\mc{L}_X\theta_{\xi}, \theta_{\eta})+\pi_{G\st}(\theta_{\xi},\theta_{\eta})\\
&={}^{t}\delta(x)(\xi,\eta)+\pi_{G^*}(\theta_{ad^*_x \xi},\theta_{\eta})+\pi_{G^*}(\theta_{\xi},\theta_{ad^*_x\eta}),
\end{split}
\end{equation}
since $\mc{L}_X\theta_{\xi}=\theta_{ad^*_x \xi}$. From ${}^{t}\delta(x)(\xi,\eta)=x([\xi,\eta])$, eq. (\ref{eq: theta2}) follows.
\end{proof}

For sake of completeness, we record an alternative way to define the dressing action. Consider $g\in G$ and $u\in G^*$ and let $ug\in \mathcal{D}$ be their product. Since $\mathfrak{d}=\mathfrak{g}\oplus\mathfrak{g}^*$, elements in $\mathcal{D}$ 
close to the unit can be decomposed in a unique way as a product of an element in $G$ and an element in $G^*$. Then, there exist elements $^u g\in G$ and $u^g\in G^*$ such that
\begin{equation}
ug=^ugu^g.
\end{equation}
Hence, the action of $u\in G^*$ on $g\in G$ (resp. the action of $g\in G$ on $u\in G^*$) is given by
\begin{equation}
(u,g)\mapsto (ug)_G\quad (\text{resp.}\quad (u,g)\mapsto (ug)_{G^*}),
\end{equation}
where $(ug)_G$ (resp. $(ug)_{G^*}$) denotes the $G$-factor (resp. $G^*$-factor) of $ug\in\mathcal{D}$ as $g'u'$, with $g'\in G$, $u'\in G^*$.
Accordingly, the product  $gu\in\mathcal{D}$ can be uniquely decomposed into $^gug^u$, where $^g u\in G^*$ and $g^u\in G$. So, by definition,
\begin{equation}
gu=^g u g^u.
\end{equation}
This defines locally a left action of $G$ on $G^*$ and a right action of $G^*$ on $G$.

\begin{definition}
A multiplicative Poisson tensor $\pi$ on $G$ is \textbf{complete} if each left (equiv. right) dressing vector field is complete on $G$.
\end{definition}

\begin{proposition}
A Poisson Lie group is complete if and only if its dual Poisson Lie group is complete.
\end{proposition}
Assume that $G$ is a complete Poisson Lie group. We denote respectively the left (resp. right) dressing action of $G$ on its dual $G^*$ by $g\mapsto l_g$ (resp. $g\mapsto r_g$).

\begin{definition}
A momentum map $\boldsymbol{\mu}:M\rightarrow G^*$ for a left (resp. right) Poisson action $\Phi$ is called \textbf{G-equivariant} if it is such with respect to the left dressing action of $G$ on $G^*$, that is, 
$\boldsymbol{\mu}\circ \Phi_g=\lambda_g\circ \boldsymbol{\mu}$ (resp. $\boldsymbol{\mu}\circ \Phi_g=\rho_g\circ \boldsymbol{\mu}$)
\end{definition}
A momentum map is $G$-equivariant if and only if it is a Poisson map, i.e. $\boldsymbol{\mu}_*\pi=\pi_{G^*}$. Given this generalization of the concept of equivariance introduced for Lie group actions, it is natural to call \textbf{Hamiltonian action} a Poisson action induced by an equivariant momentum map.

\subsection{Structure of the momentum map}
\label{sub: structure}
In this section we introduce a weaker definition of momentum map in infinitesimal terms. From Definition \ref{def: mm}, it follows that one can associate to a momentum map a 1-form $\alpha_{\xi}$. In the following, we discuss the properties of these forms and, using the infinitesimal momentum map, we analyze the conditions under which the momentum map is determined. Then, we introduce the concept of gauge equivalence for the $\alpha$'s and we show the relation of this equivalence class with a cohomological class in $H^1(M,\mathfrak{g})$.

As a direct consequence of the properties of $\theta$'s stated in Section \ref{sec: dressing}, we have the following proposition:
\begin{proposition}
Given a Poisson action $\Phi:G\times M\ra M$ with equivariant momentum map $\boldsymbol{\mu}:M\ra G\st$, the forms $\alpha_{\xi}=\boldsymbol{\mu}^*(\theta_{\xi})$ satisfy the following identities:
\begin{align}
%\xi_M &=\pi^{\sharp}(\alpha_{\xi}) \\
\label{eq: ah}\alpha_{[\xi,\eta]}&=[\alpha_{\xi},\alpha_\eta ]_{\pi}\\
\label{eq: mca} d\alpha_{\xi} &+\frac{1}{2}\alpha\wedge\alpha\circ\delta(\xi)=0
\end{align}
\end{proposition}

\begin{proof}
These identities are a direct consequence of the properties of $\theta\in\Omega(G\st)$ stated in the previous section.
In particular, eq. (\ref{eq: mca}) follows from eq. (\ref{eq: mct}) by simply recalling that the pullback and the differential commute. Eq. (\ref{eq: ah}) follows from eq. (\ref{eq: theta}), using the equivariance of the momentum map.
\end{proof}

In the following, we give the definition of infinitesimal momentum map that, as will be seen in the next chapter, plays a fundamental role in the quantization of the momentum map.
\begin{definition}\label{def: inf}
Let $M$ be a Poisson manifold and $G$ a Poisson Lie group. An \textbf{infinitesimal momentum map} is a morphism of Gerstenhaber algebras
\begin{equation}\label{eq: imm}
 \alpha: (\wedge^{\bullet}\mathfrak{g} ,\delta, [\;,\;])\longrightarrow (\Omega^{\bullet} (M), d_{DR},[\;,\;]_\pi ).
\end{equation}
\end{definition}

The following theorem is crucial in the study of the conditions in which an infinitesimal momentum map determines a momentum map in the usual sense.

\begin{theorem}\label{thm: rec}
Let $(M,\pi )$ be a Poisson manifold and $\alpha:{\mathfrak g}\rightarrow \Omega^1 (M)$ a linear map. Suppose that the following relations
\begin{equation}
\begin{split}
\alpha_{[\xi,\eta]}&=[ \alpha_{\xi},\alpha_{\eta}]_{\pi}\\
 d\alpha_{\xi} &=\alpha\wedge\alpha \circ \delta(\xi)
\end{split}
\end{equation}
are satisfied. Then:
\begin{enumerate}
    \item $\{ \alpha_{\xi}-\theta_{\xi},\ \xi\in {\mathfrak g}\}$ generate an involutive distribution $\mathcal D$ on $M \times G^*$. 
    \item Suppose, moreover, that $M$ is connected and simply connected. Then the leaves $\mathcal{F}$ of  $\mathcal D$ coincide with graphs of maps $\boldsymbol{\mu}_{\mathcal{F}}: M\rightarrow G^*$ satisfying 
$\alpha=\boldsymbol{\mu}^*_{\mathcal{F}} (\theta)$ and $G^*$ acts freely transitively on the space of leaves by left multiplication on the second factor.
    \item Vector fields $\pi^{\sharp} (\alpha_{\xi})$ give a homomorphism
\begin{equation}
 \mathfrak{g} \rightarrow TM.
\end{equation}

Suppose that they integrate to the action of $G$ on $M$ (which is automatically the case when $M$ is compact and $G$ simply connected). Then the induced action of $G$ on $M$ is a Poisson action of the Poisson group $G$
and $\boldsymbol{\mu}_{\mathcal{F}}$ is a momentum map for this action if and only if the functions
\begin{equation}
 \phi (\xi,\eta)=\pi (\alpha_{\xi} ,\alpha_{\eta})-\pi_{G^*}(\theta_{\xi},\theta_{\eta})
\end{equation}
satisfy
\begin{equation}
\label{eq:vanishing}
\phi (\xi,\eta)|_{\mathcal{F}}=0
\end{equation}
for all $\xi,\eta\in \mathfrak{g}$.
\end{enumerate}

\end{theorem}

\begin{proof}
\begin{enumerate}
    \item Using the eqs. (\ref{eq: mct})-(\ref{eq: mca}), the $\mathfrak{g}$-valued form $\alpha -\theta$ on $M\times G^*$ satisfies
\begin{equation}
 d(\alpha -\theta )=(\alpha -\theta )\wedge (\alpha -\theta )
\end{equation}
and hence it defines a distribution on $M\times G^*$. Let $\mathcal{F}$ be any of its leaves. Let $p_i$, $i=1,2$ denote the projection onto the first (resp. second) factor in $M\times G^*$. Since the linear span of 
$\theta_{\xi} ,\xi\in \mathfrak{g}$ at any point $u\in G^*$ coincides with $T_u^*{G^*}$, the restriction of the projection $p_1:M\times G \rightarrow M$ to $\mathcal{F}$ is an immersion. 
Moreover, since $dim(M)=dim(\mathcal{F})$, $p_1$ is a covering map. 
    \item As we assumed that $M$ is simply connected, $p_1$ is a diffeomorphism and
\begin{equation}
 \boldsymbol{\mu}_{\mc{F}}=p_2 \circ p_1^{-1}
\end{equation}
is a smooth map whose graph coincides with $\mathcal{F}$. It is immediate, that $\alpha =\boldsymbol{\mu}_{\mathcal{F}}^* (\theta)$.

The statement about the action of $G^*$ on the space of leaves follows from the fact that $\theta$'s are left invariant.
    \item Suppose that the condition (\ref{eq:vanishing}) is satisfied. Then
\begin{equation}
 \pi (\alpha_{\xi},\alpha_{\eta}) = \boldsymbol{\mu}^*_{\mathcal{F}} (\pi_{G^*}(\theta_{\xi},\theta_{\eta}))
\end{equation}
and $Ker {\boldsymbol{\mu}_{\mathcal{F}}}_*$ coincides with the set of zero's of $\alpha_{\xi},\ {\xi}\in \mathfrak{g}$. Hence, $\boldsymbol{\mu}_{\mathcal{F}}$ is a Poisson map and, in particular
\begin{equation}
	{\boldsymbol{\mu}_{\mathcal{F}}}_* ( \pi^{\sharp} (\alpha_{\xi}))=\pi_{G^*}^\sharp (\theta_{\xi}),
\end{equation}
i.e. it is a $G$-equivariant map.
\end{enumerate}
\end{proof}

We are interested in understanding when the condition 
\begin{equation}\label{eq:MC}
d\alpha_{\xi} +\frac{1}{2}\alpha\wedge\alpha\circ\delta(\xi)=0 ,
\end{equation}
is satisfied. We show that it can be solved explicitly, at least  in the case when $M$ is a K\"{a}hler manifold.
\begin{definition}
Two solutions $\alpha$ and $\alpha'$ of eq. (\ref{eq:MC}) are said to be \textbf{gauge equivalent}, if there exists a smooth function $H:M\rightarrow \mk{g}^*$ such that
\begin{equation}
\alpha'=\exp(ad H)(\alpha)+\int_0^1 dt \exp t(ad H)(dH)
\end{equation}
\end{definition}
\begin{theorem}
Suppose that $M$ is is a K\"{a}hler manifold. The set of gauge equivalence classes of $\alpha\in \Omega^1 (M,\mk{g}^*)$ satisfying the equation  
\begin{equation}
	d\alpha_{\xi} +\frac{1}{2}\alpha\wedge\alpha\circ\delta(\xi)=0	
\end{equation}
is in bijective correspondence with the set of the cohomology classes $c\in H^1(M,\mk{g}^*)$ satisfying
\begin{equation}
	[c,c]=0.
\end{equation}
\end{theorem}
\begin{proof} 
Since $M$ is is a K\"{a}hler manifold,  $(\Omega^{\bullet}(M),d)$ is a formal CDGA (commutative differential graded algebra) \cite{GDMS}. As a consequence,
\begin{equation}
Hom({\mathfrak g^*},\Omega^{\bullet} (M)),d,[\cdot ,\cdot ])
\end{equation}
is a formal DGLA (some elements of DGLA's will be given in the next chapter) and, in particular, there exists a bijection between the equivalence classes of Maurer Cartan elements of $Hom({\mathfrak g^*},\Omega^{\bullet} (M),d,[\cdot ,\cdot ])$
and Maurer Cartan elements of $Hom({\mathfrak g^*},H_{DR}^{\bullet} (M),[\cdot ,\cdot ])$.

A Maurer-Cartan element in $Hom({\mathfrak g^*},H_{DR}^{\bullet} (M),[\cdot ,\cdot ])$ is an element $c\in H^1(M,\mk{g}^*)$ satisfying
\begin{equation}
	[c,c]=0,
\end{equation}
and the claim is proved.
\end{proof}

\subsection{Reconstruction problem}
\label{sec: rec}
In this section we discuss the conditions under which the distribution $\mathcal D$ defined in Theorem \ref{thm: rec}
admits a leaf satisfying eq. (\ref{eq:vanishing}).
In particular, we analyze the case where the structure on $G\st$ is trivial and the Heisenberg group case.
In the following we keep the assumption that $M$ is connected and simply connected.

\paragraph{Abelian case}
Suppose that $G^*=\mathfrak{g}^*$ is abelian. Then, the forms $\alpha_{\xi}$ satisfy $d\alpha_{\xi}=0$, hence  $\alpha_{\xi}=dH_{\xi}$ (since $H^1 (M)=0$), for some $H_{\xi}\in C^{\infty}(M)$.

Let us denote by $ev_{\xi}$ the linear functions $\mathfrak{g}^*\ni l\rightarrow z(\xi)$. Then $\theta_{\xi}=d(ev_{\xi})$ and the leaves of the distribution $\mathcal D$ coincide with the level sets (on $M\times \mathfrak{g}^*$) of the functions
\begin{equation}
	\{ H_{\xi} -ev_{\xi} \mid \xi\in\mathfrak{g} \}.	
\end{equation}
Furthermore, we have
\begin{equation}
\phi (\xi,\eta)(m,z)=\{ H_{\xi} ,H_{\eta} \} - z[[\xi,\eta]).	
\end{equation}

In this case, the basic identity (\ref{eq: bff}) reduces to
\begin{equation}
	d\{ H_{\xi} ,H_{\eta} \}=dH_{[\xi,\eta]},
\end{equation}
hence
\begin{equation}
	\{ H_{\xi} ,H_{\eta} \}-H_{[\xi,\eta]}=c(\xi,\eta),
\end{equation}
for some constants $c(\xi,\eta)$. By the Jacobi identity, the constants $c(\xi,\eta)$ define a class $[c]\in H^2 (\mathfrak{g} ,\R)$. Suppose that this class vanishes (for example if $\mathfrak{g}$ semisimple). 
Then, there exists a $z_0\in\mathfrak{g}^*$ such that $c(\xi,\eta)=z_0([\xi,\eta])$. Hence, given a leaf $\mathcal{F}$,
\begin{equation}
	\phi (\xi,\eta)|_{\mathcal{F}}=0	
\end{equation}
if and only if $\mathcal{F}$ is given by
\begin{equation}
	H_{\xi}-ev_{\xi} -z_0(\xi)=0.	
\end{equation}

In other words, the space of leaves of $\mathcal D$ which give a momentum map coincides with the affine space modeled on
$\{ z\in \mathfrak{g}^* | z|_{[\mathfrak{g} ,\mathfrak{g} ]}=0\}$ (which again vanishes when $\mathfrak{g}$ is semisimple). This proves the following theorem.
\begin{theorem}
Suppose that $G$ is a connected and simply connected Lie group with trivial Poisson structure and $M$ is compact. Then an infinitesimal momentum map is a map $\mathfrak{g} \ni \xi\rightarrow C^\infty (M)$ such that 
\begin{equation}
	d\{ H_{\xi} ,H_{\eta}\} =dH_{[\xi,\eta]},\ \forall \xi,\eta\in \mathfrak{g} .
\end{equation}
$c(\xi,\eta)=\{ H_{\xi} ,H_{\eta}\} -H_{[\xi,\eta]} $ is a two cocycle $c$ on $\mathfrak{g}$ with values in $\R$. The infinitesimal momentum map $\alpha$ is generated by a momentum map 
$\boldsymbol{\mu}$ if this cocycle vanishes and, in this case, $\boldsymbol{\mu}$ is unique.
\end{theorem}

\paragraph{Heisenberg group case}
Suppose now that $G^*$ is the Heisenberg group. Let $x,y,z$
be a basis for $\mathfrak{g}^*$, where $z$ is central and $[x,y]=z$. Let $\xi,\eta,\zeta$ be the dual basis of $\mathfrak{g}$. Recall that the cocycle $\delta$ on $\mathfrak{g}$ dual to the Lie algebra 
structure on $\mathfrak{g}^*$ is given by
\begin{equation}
	[l_1 ,l_2 ](\xi)=(l_1\wedge l_2) \delta (\xi).	
\end{equation}
Hence, we have
\begin{equation}
	\delta (\xi)=\delta (\eta)=0 \mbox{ and } \delta (\zeta)=\xi\wedge \eta.
\end{equation}
and
\begin{equation}
\begin{split}
d\alpha_{\xi} &=d\alpha_{\eta} =0\\
d\alpha_{\zeta}&=\alpha_{\xi}\wedge \alpha_{\eta} .
\end{split}
\end{equation}

There are essentially two possibilities for the Lie bialgebra structure on $\mathfrak{g}^*$, which give the following two possibilities for the Lie algebra structure on $\mathfrak{g}$. Either
\begin{equation}
[\xi,\eta]=0,[\xi,\zeta]=\xi,[\eta,\zeta]=\eta
\end{equation}
or
\begin{equation}
[\xi,\eta]=0,[\xi,\zeta]=\eta,[\eta,\zeta]=-\xi.
\end{equation}
The result below will turn out to be independent of the choice - the computations will be done using the second choice, which corresponds to $G=\R \ltimes \R^2$, with $\R$ acting by rotation on $\R^2$. 
Below we use the notation
\begin{equation}
\delta(\xi)=\sum_i \xi^1_i\wedge \xi^2_i.
\end{equation}

Applying the Cartan formula $\calL =[\iota ,d]$ and the identity $[ \alpha_{\xi}, \alpha_{\eta}]_\pi =\alpha_{[\xi,\eta]}$ to the basic equation (\ref{eq: bff}), we get
\begin{equation}
\sum_i \pi (\alpha_\eta ,\alpha_{\xi_i^1})\alpha_{\xi^2_i}- \sum_i \pi (\alpha_\xi ,\alpha_{\eta^1_i})\alpha_{\eta^2_i}=\alpha_{[\eta,\xi]} - d \pi (\alpha_\eta,\alpha_\xi).
\end{equation}

In our case it gives the following equations
\begin{equation}
\begin{split}
d\pi (\alpha_{\xi}, \alpha_{\eta})&=\alpha_{[\xi,\eta]}\\
d\pi(\alpha_{\zeta}, \alpha_{\eta})&=\alpha_{[\zeta,\eta]}+\pi(\alpha_{\eta}, \alpha_{\xi}) \alpha_{\eta}\\
d\pi (\alpha_{\zeta}, \alpha_{\xi})&=\alpha_{[\zeta,\xi]}-\pi(\alpha_{\xi} ,\alpha_{\eta})\alpha_{\xi}.
\end{split}
\end{equation}
which are also satisfied after replacing $\alpha$ with $\theta$. Let $\mathcal I $ denote the ideal generating our distribution $\mathcal D$. Then, from above,
\begin{align}\label{eq:constant}
& d\phi (\xi,\eta)\in \mathcal I \\
&\phi (\xi,\eta)|_{\mathcal{F}}=0\Longrightarrow d\phi (\zeta,\eta)|_{\mathcal{F}},d\phi (\zeta,\xi)|_{\mathcal{F}}\in\mathcal I.
\end{align}
Here, as before, $\mathcal{F}$ is a leaf of $\mathcal D$.
Using the relation (\ref{eq: theta2}), we get
\begin{equation}
\begin{split}
&\calL_z^*( \pi_{G^*}(\theta_{\xi},\theta_{\eta}))=\calL_x^*( \pi_{G^*}(\theta_{\xi},\theta_{\eta}))=\calL_y^*( \pi_{G^*}(\theta_{\xi},\theta_{\eta}))=0\\
&\calL_z^*( \pi_{G^*}(\theta_{\xi},\theta_{\zeta}))=\calL_y^*( \pi_{G^*}(\theta_{\xi},\theta_{\zeta}))=0 \qquad\calL_x^*( \pi_{G^*}(\theta_{\xi},\theta_{\zeta}))=1\\
&\calL_z^*( \pi_{G^*}(\theta_{\eta},\theta_{\zeta}))=\calL_x^*( \pi_{G^*}(\theta_{\eta},\theta_{\zeta}))=0 \qquad\calL_y^*( \pi_{G^*}(\theta_{\eta},\theta_{\zeta}))=1
\end{split}
\end{equation}
In particular, $\pi_{G^*} (\theta_{\xi},\theta_{\eta})$ is invariant under left translations. Since $\pi_{G^*}$ is zero at the identity, we get
\begin{equation}
\pi_{G^*} (\theta_{\xi},\theta_{\eta})=0
\end{equation}
By the first equations (\ref{eq:constant}), $\phi (\xi,\eta)$ is leafwise constant, hence so is $\pi(\alpha_{\xi}, \alpha_{\eta})$. Hence we have
\begin{lemma}
 $\pi (\alpha_{\xi}, \alpha_{\eta})=c$ is constant on $M$ and necessary condition for existence of the momentum map is $c=0$.
\end{lemma}

Let us continue under the assumption that $c=0$. Then, given a leaf $\mathcal{F}$, by eq. (\ref{eq:constant}),  
\begin{equation}
  \phi(\eta,\zeta)|_{\mathcal{F}}=c_1 \quad \text{and} \quad  \phi(\xi,\zeta)|_{\mathcal{F}}=c_2
\end{equation}
for some constants $c_1$ and $c_2$. Setting $\mathcal{F}_1 =id \times \exp(c_1 x) \exp (c_2 y)$ to $\mathcal{F}$, we get
\begin{equation}
  \phi(\eta,\zeta)|_{\mathcal{F}_1}= \phi(\xi,\zeta)|_{\mathcal{F}_1}= \phi(\xi,\eta)|_{\mathcal{F}_1}=0.
\end{equation}

\begin{theorem} 
Let $G$ be a Poisson Lie group acting on a Poisson manifold $M$ with an infinitesimal momentum map $\alpha$ and such that $G^*$ is the Heisenberg group. 
Let $\xi,\eta,\zeta$ denote the basis of $\mathfrak{g}$ dual to the standard basis $x,y,z$ of $\mathfrak{g}^*$ (with $z$ central and $[x,y]=z$. Then
\begin{equation}
\pi(\alpha_{\xi}, \alpha_{\eta})=c
\end{equation}
is constant on $M$. The form $\alpha$ lifts to a momentum map $\boldsymbol{\mu}:M\rightarrow G^*$ if and only if $c=0$. When $c=0$ the set of momentum maps with given $\alpha$ is one dimensional with free 
transitive action of $\R$.
\end{theorem}

\subsection{Infinitesimal deformations of a momentum map}
\label{sec: inf}

In the following we study the behavior of deformations of a momentum map, close to the identity. Indeed, we consider a deformation of $\boldsymbol{\mu}$ given by the map $X: M\rightarrow \mathfrak {g}\st$
and we discuss the property of the infinitesimal generator of the action induced by this deformed momentum map. 

\begin{theorem}\label{thm: idef}
Let $(M,\pi)$ be a Poisson manifold with a Poisson action of a Poisson Lie group $(G,\pi_G)$. Suppose that $[-\epsilon ,\epsilon ]\ni t\rightarrow \boldsymbol{\mu}_t:M\rightarrow G^*$ is a differentiable 
path of momentum maps for this action. Let $\exp :{\mathfrak g}^*\rightarrow G^*$ denote the exponential map. We can assume that $\boldsymbol{\mu}_t$ is of the form $m\rightarrow \boldsymbol{\mu} (m)\exp (t X_m) +o(\epsilon)$ 
for some differentiable map $X:M\rightarrow {\mathfrak g}^*: m\mapsto X_m $. Then, for all $\xi, \eta \in \mathfrak g$,
\begin{align}
\label{eq: inf1}\mc{L}_\xi X (\eta)-\mc{L}_\eta X (\xi) &=X( [\xi,\eta])\\
\label{eq: inf2}\{ X(\xi),\ \cdot \} &= -\mc{L}_{ad^*_X\xi}.
\end{align}
\end{theorem}
\begin{proof}
Assuming that the deformed momentum maps can be written as $\boldsymbol{\mu}_t(m)=\boldsymbol{\mu} (m)\exp (t X_m)$ then we have $\alpha^t_\xi =\boldsymbol{\mu}_t^* (\theta_\xi )=\langle d\boldsymbol{\mu}_t,\theta_\xi \rangle$. 
We want to figure out its behavior close to the identity so we calculate $\left.\frac{d}{dt}\right|_{t=0}\langle d\boldsymbol{\mu}_t,\theta_\xi \rangle$. First notice that
\begin{equation}
 d\boldsymbol{\mu}_t = (\rho_{\exp (t X)})_*d\boldsymbol{\mu} +(\lambda_{\boldsymbol{\mu}})_*d \exp (t X),
\end{equation}
so we get:
\begin{equation}
\begin{split}
 \left.\frac{d}{dt}\right|_{t=0}\langle (\rho_{\exp (t X)})_*d\boldsymbol{\mu},\theta_{\xi} \rangle &=\left.\frac{d}{dt}\right|_{t=0}\langle d\boldsymbol{\mu},(\rho_{\exp (t X)})^*\theta_{\xi}\rangle\\
&=\langle d\boldsymbol{\mu},\mc{L}_X\theta_{\xi}\rangle\\
&=\langle d\boldsymbol{\mu},\theta_{ad^*_X\xi}\rangle=\alpha_{ad^*_X\xi}
\end{split}
\end{equation}
and
\begin{equation}
\begin{split}
 \left.\frac{d}{dt}\right|_{t=0}\langle (\lambda_{\boldsymbol{\mu}})_*d \exp (t X),\theta_{\xi} \rangle &=\left.\frac{d}{dt}\right|_{t=0}\langle d \exp (t X),(\lambda_{\boldsymbol{\mu}})^*\theta_{\xi} \rangle\\
&=\left.\frac{d}{dt}\right|_{t=0}\langle d \exp (t X),\theta_{\xi} \rangle
\end{split}
\end{equation}
The differential of the exponential map $\exp: {\mathfrak g}^*\rightarrow G^*$ is a map from the cotangent bundle of ${\mathfrak g}^*$ to the cotangent bundle of $G\st$. It can be trivialized as 
$d\exp:{\mathfrak g}^*\times {\mathfrak g}^*\rightarrow G\st\times {\mathfrak g}^* $. Furthermore, $(\exp^{-1},id): G\st\times {\mathfrak g}^*\rightarrow {\mathfrak g}^*\times {\mathfrak g}^*$ hence the map 
${\mathfrak g}^*\times {\mathfrak g}^*\rightarrow {\mathfrak g}^*\times {\mathfrak g}^*$ is given by $tX+o(t^2)$. We get
\begin{equation}
 \left.\frac{d}{dt}\right|_{t=0}\langle d \exp (t X),\theta_{\xi} \rangle=\left.\frac{d}{dt}\right|_{t=0}\langle d(tX+o(t)),\theta_{\xi} \rangle=d\langle X,\theta_{\xi} \rangle=d\langle X,\xi \rangle
\end{equation}
and finally
\begin{equation}\label{eq: beta}
 \beta_\xi = \left.\frac{d}{dt}\right|_{t=0}\alpha^t_\xi  =\alpha_{ad^*_X\xi }+dX (\xi ).
\end{equation}
Since $\pi^{\sharp} (\alpha^t_\xi) =\mc{L}_\xi$ is independent of $t$, we get the identity (\ref{eq: inf2}).

In order to prove the relation (\ref{eq: inf1}), recall that, since $\boldsymbol{\mu}_t$ is a family of Poisson maps, one has
\begin{equation}
\pi (\alpha^t_\xi ,\alpha^t_\eta)(m)=\pi_{G^*}(\theta_\xi ,\theta_\eta)(\boldsymbol{\mu}_t (m)).
\end{equation}
Applying $\left.\frac{d}{dt}\right|_{t=0}$ to both sides, we get
\begin{equation}
\pi (\beta_\xi ,\alpha_\eta)(m)+\pi (\alpha_\xi ,\beta_\eta)(m)=\mc{L}_X (\pi_{G^*}(\theta_\xi ,\theta_\eta))(\boldsymbol{\mu} (m)).
\end{equation}
Substituting the expression of $\beta$'s (\ref{eq: beta}) and using the following identity
\begin{equation}
 \mc{L}_X (\pi_{G^*}(\theta_\xi ,\theta_\eta))(\boldsymbol{\mu} (m))=X[\xi ,\eta ]+\pi_{G^*}(\theta_{ad^*_X\xi} ,\theta_\eta)+\pi_{G^*}(\theta_\xi ,\theta_{ad^*_X\eta})
\end{equation}
the claimed equality follows.
\end{proof}

We consider the case of a compact and semisimple Poisson Lie group $G$ to obtain a uniqueness condition for the momentum map. From the relation (\ref{eq: inf1})
we can conclude that there exists a function $\Phi$ such that
\begin{equation}
 \mc{L}_{\xi}\Phi=X(\xi).
\end{equation}
Using this expression we get $\mc{L}_{ad^*_X\xi}f=\mc{L}_{\xi'}\Phi\xi''(f)$, where $\delta(\xi)=\xi'\otimes \xi''$. Now observe that
\begin{equation}
 \xi\{\Phi,f\}=\mc{L}_{\xi}\pi(d\Phi,df)=(\mc{L}_{\xi}\pi)(d\Phi,df)+\{\mc{L}_{\xi}\Phi,f\}+\{\Phi,\mc{L}_{\xi}f\}
\end{equation}
hence
\begin{equation}
\begin{split}
\{X(\xi),f\}&=\{\mc{L}_{\xi}\Phi,f\}=\xi\{\Phi,f\}-(\mc{L}_{\xi}\pi)(d\Phi,df)-\{\Phi,\mc{L}_{\xi}f\}\\
&=\xi\{\Phi,f\}-\delta(\xi)(\Phi,f)-\{\Phi,\mc{L}_{\xi}f\}\\
&=\xi\{\Phi,f\}-\mc{L}_{\xi'}\Phi\;\xi''(f)-\{\Phi,\mc{L}_{\xi}f\}.
\end{split}
\end{equation}
Substituting these results in (\ref{eq: inf2}) we get
\begin{equation}
 \xi\{\Phi,f\}-\{\Phi,\mc{L}_{\xi}f\}=0.
\end{equation}
This means that there exists a vector field $H_{\Phi}$ associated with the deformation of the momentum map which commutes with the action:
\begin{equation}
 [H_{\Phi},\mc{L}_{\xi}]=0.
\end{equation}
In other words, given the momentum map $\boldsymbol{\mu}:M\rightarrow G\st$, if there exists an endomorphism on $M$ such that the associated vector field commutes with the action,
then we get another momentum map, as discussed in the Theorem (\ref{thm: idef}).

\section{Poisson Reduction}
\label{sec: poisson reduction}

Here we present the main result of this chapter. We show that, given a Poisson action we can define a reduced manifold in terms of momentum map.
A first generalization of the Marsden-Weinstein reduction has been given by Lu in \cite{Lu1}, where it is shown that, given a Poisson Lie group acting on a symplectic manifold $M$,
the symplectic structure on $M$ induces a symplectic structure on the leaves of $M/G$ generated by the momentum map. 

Given a Poisson action $\Phi:G\times M\rightarrow M$ with momentum map $\boldsymbol{\mu}: M\rightarrow G^*$, we define a $G$-invariant foliation $\mathcal{F}$ of $M$.
The leaves are not Poisson manifolds, but considering the action of $G$ on the space of leaves, we prove that, for each leaf $\mathcal{L}$, the Poisson structure on $M$ induces a
Poisson structure on the orbit space $\mathcal{L}/G_{\mathcal{L}}$, where $G_{\mathcal{L}}$ is the isotropic group at any point of $\mathcal{L}$.
This shows that we can reduce $M$ to another Poisson manifold $\mathcal{L}/G_{\mathcal{L}}$ that we call the \textbf{Poisson reduced space}.

\subsection{Poisson structure on $M/G$}

In this section we prove that, given a Poisson Lie group $(G,\pi_G)$ acting on a Poisson manifold $(M,\pi)$, with equivariant momentum map $\boldsymbol{\mu}: M\rightarrow G^*$, the orbit space inherits a Poisson structure from $M$. From now on we further assume that the Poisson Lie group $G$ is complete.

In \cite{STS} Semenov-Tian-Shansky showed that, given a Poisson action, if the orbit space is a smooth manifold, it carries a Poisson structure such that the natural 
projection $\text{pr}:M\rightarrow M/G$ is a Poisson mapping. More precisely, given $f,h\in C^{\infty}(M)$ with the definitions
\begin{equation}
\hat{f}(m,g):=f(g\cdot m),\quad \hat{h}(m,g):=h(g\cdot m),
\end{equation}
 for any $f,h$ one finds
\begin{equation}
\{f,h\}_M(g\cdot m)=\{\hat{f}(m,\cdot),\hat{h}(m,\cdot)\}_G(g)+\{\hat{f}(\cdot, g),\hat{h}(\cdot, g)\}_M(m)
\end{equation}
Then $M/G$ inherits a Poisson structure from the Poisson structure on $M$:
\begin{equation}
f,h\in C^{\infty}(M)^G\Longrightarrow \{f,h\}_M\in C^{\infty}(M)^G.
\end{equation}

Consider a Poisson action $\Phi:G\times M\rightarrow M$ with equivariant momentum map $\boldsymbol{\mu}: M\rightarrow G^*$ and assume that the orbit space $M/G$ is a smooth manifold. 
Recall that when the Poisson structure of $G$ is trivial, the infinitesimal Poisson action $\xi_M$ is a Hamiltonian vector field, but in general this does not hold. The first goal of this section is to provide an explicit formulation for $\xi_M$, in terms of local coordinates. We use the properties of the momentum map and dressing action to obtain such a formulation.

We observe that the Poisson Lie group $G^*$ can be described locally in terms of coordinates $(q,p,y)$ such that $\pi_{G\st}$
 is given by the Splitting Theorem \ref{thm: split}.
In particular, the transverse structure is determined by the structure functions 
$\pi_{G^{*}}^{ij}(y)=\{ y_i,y_j\}$, which vanishes on the symplectic leaves. As discussed in Section \ref{sec: dressing} the Poisson Lie group $G$ acts on $G\st$ by dressing action and the dressing orbits are the same as the symplectic leaves. Hence, the generic orbit $\mathcal{O}_x$ through $x\in G\st$ is a closed submanifold of $G^*$ and $y_i$ are transversal 
coordinates such that $\{ y_i,y_j\}=0$. Note that $\boldsymbol{\mu}$ is a submersion, hence it has open image. In particular, the image of $M$ is an open neighborhood of a generic orbit of $G$ on $G^*$.

Define the functions $H_i$ as the pullbacks by $\boldsymbol{\mu}$ of the transversal coordinates $y_i$ to the orbit on $G^*$:
\begin{equation}\label{eq: hi}
 H_i:=y_i\circ\boldsymbol{\mu}.
\end{equation}
$H_i$ are defined locally in a $G$-invariant open neighborhood $U$ of the preimage $N=\boldsymbol{\mu}^{-1}(\mathcal{O}_x)$. We can assume that $x$ is a regular value of $\boldsymbol{\mu}$, hence $N$ is a 
closed $G$-invariant submanifold of $M$. Since $\{ y_i,y_j\}$ vanishes on the orbit $\mathcal{O}_x$, $\{ H_i,H_j\}$ vanishes on the preimage $N$.
The 1-forms $\alpha_{\xi}=\boldsymbol{\mu}^*(\theta_{\xi})$ are in the span of the $dH_i$'s. Since the left invariant 1-form $\theta_{\xi}$ in local coordinates can be expressed as a linear combination of $dy_i$,
using the definition (\ref{eq: hi}) we have
\begin{equation}\label{eq: al}
\alpha_{\xi}=\boldsymbol{\mu}^*(\theta_{\xi})=\sum_{i}c_{i}(\xi)dH_i
\end{equation}
for any $\xi\in\mathfrak{g}$.
As a consequence, the infinitesimal generators $\xi_M$ of the Hamiltonian action $\Phi$, induced by $\boldsymbol{\mu}$, can be written as a linear combination of Hamiltonian vector fields:
\begin{equation}\label{eq: xis}
\xi_M=\pi^{\sharp}(\alpha_{\xi})=\sum_{i}c_{i}(\xi)\lbrace H_j,\cdot\rbrace.
\end{equation}	

We use this explicit formulation  to prove that $M/G$ inherits a Poisson structure from $M$:

\begin{theorem}\label{thm: m/g}
 Let $\Phi: G\times M\rightarrow M$ be a Poisson action with equivariant momentum map $\boldsymbol{\mu}$. The algebra $C^{\infty}(M)^G$ of $G$-invariant functions on $M$ is a Lie subalgebra of $C^{\infty}(M)$.
\end{theorem}

\begin{proof}
Let $f,g\in C^{\infty}(M)^G$, then $\xi_M[f]=\xi_M[g]=0$ for any $\xi\in\mathfrak{g}$. Applying the relation (\ref{eq: xis}) we have that
\begin{equation}
 \sum_{i}c_{i}(\xi)\lbrace H_i,f\rbrace=\sum_{i}c_{i}(\xi)\lbrace H_i,g\rbrace=0
\end{equation}
that implies $\lbrace H_i,f\rbrace=\lbrace H_i,g\rbrace=0$ for any $i$. Then, using the Jacobi identity we get
 $\lbrace H_i,\lbrace f,g\rbrace\rbrace=0$. Since $G$ is connected we proved that
\begin{equation}
 \xi_M\left[\lbrace f,g\rbrace\right]=0.
\end{equation}
Hence $\lbrace f,g\rbrace$ is $G$-invariant and the claim is proved.
\end{proof}

\subsection{Poisson structure on $\mathcal{L}/G_{\mathcal{L}}$}

Consider the $\mathfrak{g}^{*}$-valued 1-forms $\alpha_{\xi}$ defined by $\boldsymbol{\mu}$ by eq. (\ref{eq: al}). The distribution $\{\alpha_{\xi}\vert \xi\in\mathfrak{g}\}$ defines a $G$-invariant foliation 
$\mathcal{F}$ on $M_{reg}$, the open submanifold of regular values of $\boldsymbol{\mu}$ of $M$ by
\begin{equation}
 T_m\mathcal{L}= \ker {\alpha_\xi}(m)=\bigcap_i \ker dH_i(m)
\end{equation}
for any leaf $\mathcal{L}$, which is of the form $\mathcal{L}=\boldsymbol{\mu}^{-1}( x)$. The leaf $\mathcal{L}$ is not a Poisson submanifold but we prove that, considering the action of $G$ on the space of leaves, 
the quotient $\mathcal{L}/G_{\mathcal{L}}$ inherits a Poisson structure by $M$, where
\begin{equation}
 G_{\mathcal{L}}=\{g\in G\vert g\cdot\mathcal{L}=\mathcal{L}\}
\end{equation}
is the stabilizer of the action of $G$ on $\mathcal{L}$.

In order to prove this statement we observe that, since $\pi_{G^{*}}$ restricted to $\mathcal{O}_x$ does not depend on the transversal coordinates $y_i$'s, the Poisson structure $\pi$ on $M$ depends on the coordinates $H_i$ defined in (\ref{eq: hi}) only in the combination $\partial_{x_i}\wedge\partial_{H_i}$. This is evident because the differential $d\boldsymbol{\mu}$ between $TM\vert_N/TN$ and $TG^{*}/T\mathcal{O}_x$ is a bijective map, so using the definition (\ref{eq: hi}) the claim is proved.

Now consider the ideal $\mathcal{I}$ generated by $H_i$. The coordinates $H_i$ are locally defined but we can show that $\mathcal{I}$ is globally defined.
Considering a different neighborhood on the orbit of $G^{*}$ we have transversal coordinates $y_i^{\prime}$ and their pullback to $M$ will be $H_i^{\prime}=y_i^{\prime}\circ \boldsymbol{\mu}$.
The coordinates $H_i^{\prime}$ are defined in a different open neighborhood $V$ of $N$, but we can see that the ideal $\mathcal{I}$ generated by $H_i$ coincides with $\mathcal{I}^{\prime}$ generated by $H_i^{\prime}$ 
on the intersection of $U$ and $V$, then it is globally defined.
Moreover, since $\boldsymbol{\mu}$ is a Poisson map we have:
\begin{equation}
 \{ H_i,H_j\}=\{ y_i\circ \boldsymbol{\mu},y_j\circ \boldsymbol{\mu}\}=\{ y_i,y_j\}\circ \boldsymbol{\mu}.
\end{equation}
Hence the ideal $\mathcal{I}$ is closed under Poisson brackets.

\begin{lemma}\label{lem: id2}
 Suppose that $N/G$ is an embedded submanifold of the smooth manifold $M/G$, then
\begin{equation}
 (C^{\infty}(M)/\mathcal{I})^G=(C^{\infty}(M)^G + \mathcal{I})/\mathcal{I}
\end{equation}
\end{lemma}

\begin{proof}
 Let $f$ be a smooth function on $M$ satisfying $[f]\in (C^{\infty}(M)/\mathcal{I})^G$. If the equivalence class $[f]$ is $G$-invariant, we have
\begin{equation}
f(G\cdot m)=f(m)+i(m),
\end{equation}
where $i\in \mathcal{I}=\{f\in C^{\infty}(M): f\vert_N =0 \}$. It is clear that $f\vert_N$ is $G$-invariant and hence it defines a smooth function $\bar{f}\in C^{\infty}(N/G)$.
Since $N/G$ is a $k$-dimensional embedded submanifold of the $n$-dimensional smooth manifold $M/G$, the inclusion map $\iota: N/G\rightarrow M/G$ has local coordinates representation:
\begin{equation}
 (x_1,\dots,x_k)\mapsto (x_1,\dots,x_k,c_{k+1},\dots,c_n)
\end{equation}
where $c_i$ are constants. Hence we can extend $\bar{f}$ to a smooth function $\phi$ on $M/G$ by setting $\bar{f}(x_1,\dots,x_k)=\phi(x_1,\dots,x_k,0,\dots,0)$.
The pullback $\tilde{f}$ of $\phi$ by $\text{pr}:M\rightarrow M/G$ is $G$-invariant and satisfies
\begin{equation}
 \tilde{f}-f\vert_N=0,
\end{equation}
hence $\tilde{f}-f\in \mathcal{I}$.
\end{proof}

\begin{theorem}\label{thm: pred}
Let $\Phi:G\times M\rightarrow M$ be a Poisson action of $(G,\pi_G)$ on a Poisson manifold $(M,\pi)$ with equivariant momentum map $\boldsymbol{\mu}:M\rightarrow G^*$. For each leaf, the orbit space 
$\mathcal{L}/G_{\mathcal{L}}$ has a Poisson structure induced by $\pi$.
\end{theorem}

\begin{proof}
First we prove that the Poisson bracket of $M$ induces a well defined Poisson bracket on $(C^{\infty}(U)^G+\mathcal{I})/\mathcal{I}$.
In fact, for any $f+i\in C^{\infty}(U)^{G}/\mathcal{I}$ and $j\in \mathcal{I}$ the Poisson bracket $\{ f+i,j\}$ still belongs to the ideal $\mathcal{I}$.
Since the ideal $\mathcal{I}$ is closed under Poisson brackets, $\{ i,j\}$ belongs to $\mathcal{I}$.
The function $j$, by definition on the ideal $\mathcal{I}$, can be written as a linear combination of $H_i$, so $\{ f,j\}=\sum_i a_i\{ f,H_i\}$.
By the Theorem \ref{thm: m/g}, we have $\{ f,H_i\}=0$, hence $\{ f+i,j\}\in \mathcal{I}$ as stated.
Finally, using the isomorphism proved in the Lemma (\ref{lem: id2}) and the identifications
\begin{equation}\label{eq: id1}
 C^{\infty}(\mathcal{L}/G_{\mathcal{L}})\simeq C^{\infty}(N/G)\simeq(C^{\infty}(U)/\mathcal{I})^{G}.
\end{equation}
the claim is proved.
\end{proof}

We refer to $\mathcal{L}/G_{\mathcal{L}}$ as the \textbf{Poisson reduced space}.

\section{An example: $\mathbb{R}^2$ action}
\label{sec: ex}

Here we want to discuss a concrete example for the Poisson reduction.
Consider the Lie bialgebra $\mathfrak{g}=\mathbb{R}^2$ with generators $\xi$ and $\eta$ such that
\begin{equation}
[\xi,\eta]=\eta
\end{equation}
and cobracket given by
\begin{equation}
\delta(\xi)=0\quad \delta(\eta)=\xi\wedge\eta.
\end{equation}

The matrix representation of $\mathfrak{g}$ is the Lie algebra $\mathfrak{gl}(2,\mathbb{R})$ and the subgroups $G$ and $G^*$
of $GL(2,\mathbb{R})$ of matrices with positive determinant are given by
\begin{equation}
G=\left\lbrace \left(\begin{matrix} 1 & 0 \\ x & y \end{matrix}\right)\; :y>0\right\rbrace \qquad G^*=\left\lbrace \left(\begin{matrix} a & b \\ 0 & 1 \end{matrix}\right)\; :a>0\right\rbrace
\end{equation}
and we remark that the Poisson bivector on $G^*$ is
\begin{equation}
\pi_{G^*}=a b\partial_a\wedge\partial_b.
\end{equation}

In this simple case it is clear that $\{a,b\}$ are global coordinates on $G^*$. 
We analyze the orbits of the dressing action of $G$ on $G^*$ for this example.

Remember that the dressing orbits $\mathcal{O}_x$ through a point $x\in G^*$ are the same as the symplectic leaves, hence it 
is clear that they are generated by the equation $b=0$. The symplectic foliation of the manifold $G^*$ is now given by two open orbit, determined by the conditions
$b>0$ and $b<0$ respectively, and a closed orbit given by $b=0$ and $a\in\mathbb{R}$.

Consider a Poisson action of $G$ on a generic Poisson manifold $M$ induced by the equivariant momentum map $\boldsymbol{\mu}:M\rightarrow G^*$.
Its pullback
\begin{equation}
\boldsymbol{\mu}\st: C^{\infty}(G^*)\longrightarrow C^{\infty}(M)
\end{equation}
maps the coordinates $a$ and $b$ on $G^*$ to $\hat{a}(x)=a(\boldsymbol{\mu}(x))$ and $\hat{b}(x)=b(\boldsymbol{\mu}(x))$ resp. In order to simplify the notation we denote the coordinates on $M$ only with $a$ and $b$. 
It is important to underline that we have no information on the dimension of $M$, so $a$ and $b$ are just a couple of the possible coordinates. Nevertheless, since $\boldsymbol{\mu}$ is a Poisson map,  we have
\begin{equation}
\lbrace a,b\rbrace=ab
\end{equation}
on $M$. The infinitesimal action of $\mathfrak{g}=\mathbb{R}^2$ on $M$ that we consider can be written in terms of these 
coordinates $a,b$ as
\begin{equation}
\Phi(\xi)=a\{b,\cdot \}\quad \Phi(\eta)=a\{a^{-1},\cdot \}.
\end{equation}

In the previous section we proved that the Poisson reduction is given equivalently either as the Poisson algebra $C^{\infty}(N/G)$ on the quotient $N/G$, with $N=\boldsymbol{\mu}^{-1}(\mathcal{O}_x)$ or as 
$(C^{\infty}(M)/\mathcal{I})^{G}$.
In the following, we discuss 3 different cases of dressing orbit.

\paragraph{Case 1: $b>0$.} Consider the dressing orbit $\mathcal{O}_x$ generated by the condition $b>0$. Since $a$ and $b$ are both positive, we can put
\begin{equation}
  a=e^p,\quad b=e^q
\end{equation}
and we have
\begin{equation}
	\lbrace p,q\rbrace=1
\end{equation} 
since $\lbrace a,b\rbrace=ab$. For this reason we can claim that
the preimage of the dressing orbit can be split as $N=\mathbb{R}^2\times M_1$ and $C^{\infty}(N)$ is given explicitly by the set of functions generated by $b^{-1}$. The infinitesimal action is given by
\begin{equation}
\Phi(\xi)=e^p\{e^q,\cdot \}\qquad \Phi(\eta)=e^p\{e^{-p},\cdot \}
\end{equation}
which is just the action of $G$ on the plane. Hence the Poisson reduction in this case is given by
\begin{equation}
 (C^{\infty}(M)[b^{-1}])^G.
\end{equation}

\paragraph{Case 2: $b<0$.} This case is similar, with the only difference that $b=-e^q$. 

\paragraph{Case 3: $b=0$.} This case is slightly different. The orbit $\mathcal{O}_x$ is given by fixed points on the line $b=0$, then we choose the point $a=1$. Clearly, in this case we can not define $b=e^p$.

Consider the ideal $\mathcal{I}=\langle a-1,b\rangle$ of functions vanishing on $N$. It is easy to check that it is $G$-invariant, hence the Poisson reduction in this case is:
\begin{equation}
 (C^{\infty}(M)/\mathcal{I})^G.
\end{equation}

%% file: Chapter3/chapter3.tex
\chapter{Quantum Momentum Map}
\label{ch: three}

The main goal of this chapter is the definition of the quantum momentum map as a deformation of the classical momentum map, introduced in the previous chapter. In the first part we introduce the reader to the theory of deformation 
quantization of Poisson manifolds developed by M. Kontsevich \cite{Ko}. Then we present some basic results about quantum groups and their connection with Poisson Lie groups and Lie bialgebras \cite{C}. Finally, we discuss the quantization of the momentum map and analyze some examples of quantum reduction.

\section{Deformation quantization of Poisson manifolds}\label{sec_3.1}

We start this section with a survey of deformation quantization sketching the physical motivations which underlie such theory.
In general, deformation quantization establishes a correspondence between classical and quantum mechanics.

In the Hamiltonian formalism of classical mechanics, physical observables are represented by smooth functions on a certain space, called phase space. This generally is a symplectic or Poisson manifold $M$. On the other hand, a quantum system is usually described by a Hilbert space and the observables are self-adjoint operators on it. However, a formal correspondence between the two theories is still missing, despite the fact that many progresses in that direction have been done.

The problem of finding a precise mathematical procedure to associate to a classical observable (smooth function on $M$) a quantum analog, was first approached by trying to construct a correspondence between the commutative algebra $C^{\infty}(M)$ and the non-commutative algebra of operators. Starting from the quantization of $\mathbb{R}^{2n}$, the first result was achieved by Groenewold \cite{Gro}, which states that  the Poisson algebra $C^{\infty}(\mathbb{R}^{2n})$ can not be quantized in such a way that the Poisson bracket of two classical observables is mapped into the Lie bracket of the correspondent operators. 

The idea of Bayen, Flato et al. \cite{BFLS},\cite{FLS1}, \cite{FLS2} was a change of perspective: instead of mapping functions to operators, the algebra of functions can be deformed into a non-commutative one.
In particular, they proved
that on the symplectic vector space $\mathbb{R}^{2n}$, there exists a standard deformation quantization, or star product, known as the Moyal-Weyl product. 
The origins of the Moyal-Weyl product can be found in the works of Weyl \cite{Wy} 
and Wigner \cite{Wi}, where they give an explicit correspondence between functions and operators, and of Groenewold \cite{Gro} and Moyal \cite{Mo}, where 
the product and the bracket of operators defined by Weyl have been introduced. The existence of an associative star product has been generalized to a symplectic manifold admitting a flat connection $\nabla$ in \cite{BFLS}. 
The first proof of the existence of star product for any symplectic manifold was given by De Wilde and Lecomte \cite{DL} and few years later by Fedosov \cite{Fe}. In subsequent works (e.g. \cite{NT}, \cite{Gu}) the equivalence classes of star products on symplectic manifolds and the connection with de Rham cohomology has been studied. It came out that the equivalence classes of star products and elements in $H_{dR}^2(M)  \llbracket \epsilon  \rrbracket $ are in a one-to-one correspondence.

The existence and classification of star product culminated with Kontsevich's Formality Theorem, that was first formalized in a conjecture in \cite{Ko1} and then proved in \cite{Ko}. Kontsevich showed that any finite dimensional Poisson manifold $M$ admits a canonical deformation quantization and established a correspondence between the set of isomorphism classes of deformations of $C^{\infty}(M)$ 
and the set of equivalence classes of formal Poisson structures on $M$.

\subsection{Classification of Star Products}

We start discussing the problem of the existence of a formal deformation for an arbitrary Poisson manifold. 
First, we recall the basic notion of formal deformation of an algebra $A$ and then we explain the connection of deformations with Poisson structures. 

Kontsevich in \cite{Ko} solved the problem of classifying star products on a given Poisson manifold $M$ by proving that there is a one-to-one correspondence between equivalence classes of star products and equivalence classes of Poisson structures.

Let $k$ be a commutative ring and $A$ a $k$-algebra, associative and unital. 
Denote by $k  \llbracket \hbar  \rrbracket $ the ring of formal power series in $\hbar$ and by $A  \llbracket \hbar  \rrbracket $ the $k  \llbracket \hbar  \rrbracket $-module of formal power
series 
\begin{equation}
\sum_{n=0}^{\infty} \hbar^n a_n
\end{equation}
with coefficients in $A$.
A \textbf{formal deformation} of the algebra $A$ is a formal power series 

\begin{equation}\label{eq: star}
a\star b=ab+\sum_{k=1}^{\infty}\hbar^k P_k(a,b)
\end{equation}
where $P_m:A\times A\rightarrow A$ are $k$-bilinear maps such that

\begin{enumerate}
\item The product $\star$ is associative
\item $P_k(1,f)=P_k(f,1)=0$ for any $f\in A$
\end{enumerate}

An isomorphism of two deformations $\star$, $\star'$ is a formal power series \linebreak $T(a)=a+\sum_{m=0}^{\infty}t^m T_m(a)$ such that
\begin{equation}
 T(a\star b)=T(a)\star' T(b)\quad\forall a,b\in A.
\end{equation}
Let $M$ be a smooth manifold. It has been proven in \cite{BFLS} that a deformation quantization of $C^{\infty}(M)$, or a \textbf{star product}, is a deformation of $\mathcal{A}=C^{\infty}(M)$ 
such that $P_m$ are bidifferential operators. An isomorphism of two star products is an isomorphism of the corresponding deformations such that the operators $T_m$ are differential.

Given a star product on a smooth manifold $M$, we can define a Poisson bracket on $C^{\infty}(M)$ by setting
\begin{equation}\label{eq: pbd1}
\lbrace f,g\rbrace=P_1(f,g)-P_1(g,f).
\end{equation}
Recall from the previous chapter that we can associate a bivector $\pi$ to the Poisson bracket, putting 
\begin{equation}\label{eq: pbd2}
\lbrace f,g\rbrace_{\pi}=\pi( df,dg)
\end{equation}
From the associativity of $\star$ we obtain that the Poisson bracket (\ref{eq: pbd1}) is necessarily of the form $\lbrace f,g\rbrace_{\pi_0}$ for some bivector field $\pi_0$. 

From  Proposition \ref{pro: sn} we know that a bivector $\pi$ is a Poisson bivector if and only if the Schouten bracket $\left[ \pi,\pi\right]_S $ is zero. It is easy to show that for any star product the bivector field $\pi_0$ in eq. (\ref{eq: pbd2}) is a Poisson structure. Moreover, we can define a formal Poisson structure as a formal power series $\pi_{\hbar}= \sum_{m=0}^{\infty}\hbar^m\pi_m$ such that $\left[ \pi_{\hbar},\pi_{\hbar}\right]_S =0$. For any Poisson structure $\pi$ it is possible to define a formal Poisson structure $\pi_{\hbar}=\hbar\pi$. Two formal Poisson structures $\pi_{\hbar}$ and $\pi'_{\hbar}$ are equivalent if there is a formal power series $X=\sum_{m=0}^{\infty}\hbar^m X_m$ such that $\pi'_{\hbar}=\exp(\mathcal{L}_X)\pi_{\hbar}$. The connection between formal Poisson structures and star products motivates the following result:

\begin{theorem}[Kontsevich, \cite{Ko}]
There is a bijection, natural with respect to diffeomorphisms, between the set of equivalence classes of formal Poisson structures on $M$ and the set of isomorphism classes of deformation quantizations of $C^{\infty}(M)$.
\end{theorem}

In other words, this theorem states that classes of star products corresponds to classes of deformations of the Poisson structure, i.e. any Poisson manifold admits a deformation quantization. This result follows from a more general one, called Formality theorem.

\subsection{Formality Theory}
As mentioned in the previous section, a Poisson structure is completely defined by the choice of a bivector field satisfying certain properties; on the other hand a star product is specified by a family of bidifferential operators. In order to work out the correspondence between these two objects, we introduce the two differential graded Lie algebras they belong to: multivector fields $\mathfrak{g}_S^{\bullet}(M)$ and multidifferential operators $\mathfrak{g}_G^{\bullet}(C^{\infty}(M))$.
In the Formality theorem, Kontsevich constructed a $L_{\infty}$ quasi-isomorphism between these differential graded Lie algebras. 

\begin{definition}
A \textbf{graded Lie algebra} (GLA) is a graded vector space \linebreak 
$\mathfrak{g}=\oplus_{i\in\mathbb{Z}}\mathfrak{g}^i$ endowed with a bilinear operation
\begin{equation}
[\cdot,\cdot]:\mathfrak{g}\otimes\mathfrak{g}\rightarrow \mathfrak{g}
\end{equation}
satisfying the following conditions:
\begin{enumerate}
\item homogeneity, $[a,b]\in \mathfrak{g}^{\alpha+\beta}$
\item skew-symmetry, $[a,b]=-(-1)^{\alpha\beta}[b,a]$
\item Jacoby identity, $[a,[b,c]]=[[a,b],c]+(-1)^{\alpha\beta}[b,[a,c]]$
\end{enumerate}
for any $a\in \mathfrak{g}^{\alpha}$, $b\in \mathfrak{g}^{\beta}$ and $c\in \mathfrak{g}^{\gamma}$.
\end{definition}

As an example, any Lie algebra is a GLA concentrated in degree 0.

\begin{definition}
A \textbf{differential graded Lie algebra} (DGLA) is a GLA $\mathfrak{g}$ together with a differential $d:\mathfrak{g}\rightarrow \mathfrak{g}$, i.e. a linear operator of degree 1 which satisfies the Leibniz rule
\begin{equation}
d[a,b]=[da,b]+(-1)^{\alpha\beta}[a,db]\qquad a\in \mathfrak{g}^{\alpha},\quad b\in \mathfrak{g}^{\beta}
\end{equation}
 and $d^2=0$.
\end{definition}
Given a DGLA we can define immediately the cohomology of $\mathfrak{g}$ as
\begin{equation}
H^i(\mathfrak{g}):=Ker(d:\mathfrak{g}^i\rightarrow \mathfrak{g}^{i+1})/Im(d: \mathfrak{g}^{i-1}\rightarrow \mathfrak{g}^i)
\end{equation}
The set $H:=\oplus_i H^i(\mathfrak{g})$ has a natural structure of graded Lie algebra.

The morphism $f:\mathfrak{g}_1\rightarrow \mathfrak{g}_2$ of DGLA's induces a morphism $H(f):H_1\rightarrow H_2$ between cohomologies. Recall that a \textbf{quasi-isomorphism} is a morphism of DGLA's inducing isomorphisms in cohomology.

\begin{definition}\label{def: fdgla}
A differential graded Lie algebra $\mathfrak{g}$ is \textbf{formal} if it is quasi-isomorphic to its cohomology, regarded as a DGLA with zero differential and the induced bracket.
\end{definition}

\subsection*{Multivector fields and Multidifferential operators}

In the following, we discuss two DGLA's which will play a fundamental role in deformation quantization (a useful presentation can be found in \cite{CI}). We start with the DGLA of multidifferential operators, which is a subalgebra of the Hochschild DGLA. In the following, we explain how this algebra is  constructed. 

The Hochschild complex of an associative unital algebra $A$ is the complex $C(A,A)$ with vanishing components in degree $n<0$ 
and whose $n$-th component, for $n\geq 0$ is the space
\begin{equation}
\tilde{C}(A,A):= \sum_{n=-1}^{\infty}\tilde{C}^n(A,A)\qquad     \tilde{C}^n(A,A)=Hom(A^{\otimes n+1},A).
\end{equation}
If $A=C^{\infty}(M)$, we require that $\tilde{C}^n(A,A)$ consists of those maps from $A^{\otimes n}$ to $A$ which are multi-differential.
By definition, the differential of a $n$-cochain $f$ is the $(n+1)$-cochain defined by
\begin{equation}
\begin{split}
(-1)^n(df)(a_0,\dots, a_n) &= a_0 f(a_1,\dots, a_n)-\sum_{i=0}^{n-1}(-1)^i f(a_0,\dots, a_i a_{i+1},\dots, a_n)\\
                                         &+(-1)^{n-1}f(a_0,\dots, a_{n-1})a_n
\end{split}
\end{equation}
The Hochschild cohomology $H(A,A)$ of $A$ is the homology associated to the Hochschild complex. The normalized 
Hochschild complex is 
\begin{equation}\label{eq: hcc}
C^n(A,A)=Hom(\bar{A}^{\otimes n},A)
\end{equation}
where $\bar{A}=A/k1$. Now we introduce a new structure on the Hochschild complex, the Gerstenhaber bracket \cite{G}.
The Gerstenhaber product of $f\in\tilde{C}^n(A,A)$ and $g\in \tilde{C}^m(A,A)$ is the $(n+m-1)$-cochain defined by
\begin{equation}
(f\circ g)(a_1,\dots, a_{n+m-1})=\sum_{j=0}^{n-1}(-1)^{(m-1)j}f(a_1,\dots, a_j, g(a_{j+1}, \dots, a_{j+m} ),\dots)
\end{equation}
that is not associative in general. As a consequence, we define the Gerstenhaber bracket as follows:
\begin{equation}
[D,E]_G=D\circ E-(-1)^{(n-1)(m-1)}E\circ D.
\end{equation}
We notice that the Hochschild differential can be expressed in terms of the Gerstenhaber bracket 
and the multiplication $m$ of $A$ as
\begin{equation}
d=[m, \cdot]_G:\tilde{C}^{\bullet}(A,A)\rightarrow \tilde{C}^{\bullet+1}(A,A)
\end{equation}

It follows that the Hochschild complex $\tilde{C}^{\bullet}(A,A)$ endowed with the Gerstenhaber bracket is a DGLA \cite{G}
 that we denote by $\mathfrak{g}^{\bullet}_G(A)$. It is well known that the embedding of $C^{\bullet}(A,A)$ into
 $\tilde{C}^{\bullet}(A,A)$ is a quasi-isomorphism \cite{CE}.

The second DGLA we are interested in is given by the multivector fields on $M$. By definition, a $k$-multivector field $X$ is a section of the $k$-th exterior power $\wedge^k TM$ of the tangent space $TM$.
We define the Schouten-Nijenhuis bracket:
\begin{equation}
[X,Y_1\wedge\dots\wedge Y_k]_S:=\sum_{i=1}^k(-1)^{i+1}[X,Y_i]\wedge Y_1\wedge\dots\wedge\hat{Y}_i\wedge\dots\wedge Y_k.
\end{equation}
such that
\begin{enumerate}
	\item $X\in \Gamma(M,T)$, $[X,\pi]_S=L_X\pi$,
	\item for $f,g\in\Gamma(M,\wedge^0 T)$, $[f,g]_S=0$,
	\item the bracket $[\cdot,\cdot]_S$ turns $\Gamma(M,\wedge^{\bullet+1}T)$ into a graded Lie algebra,
	\item for any $\pi,\psi,\varphi\in \Gamma(M,\wedge^{\bullet}T)$,
		\begin{equation}
[\pi, \varphi\wedge\psi]_S=[\pi, \varphi]_S\wedge\psi+(-1)^{\vert\pi\vert (\vert\varphi\vert+1)}\varphi\wedge[\pi,\psi]_S.
		\end{equation}	
\end{enumerate}

The space $\wedge^k TM$, endowed with the Schouten-Nijenhuis bracket and with differential $d=0$, is a DGLA, which we denote by 
\begin{equation}
\mathfrak{g}_S^{\bullet}(M)=\Gamma(M,\wedge^{\bullet+1}T).
\end{equation}
% We can see that $C^{\bullet}(A,A)$ is a subcomplex and a graded Lie subalgebra of $\tilde{C}^{\bullet}(A,A)$ and it is known that the embedding of $C^{\bullet}(A,A)$ into $\tilde{C}^{\bullet}(A,A)$ is a quasi-isomorphism. 
%Denote
%\begin{equation}
%=C^{\bullet+1}(A,A)
%\end{equation}
%The differential $\delta$ and the bracket $[\cdot,\cdot]_S$ make $\mathfrak{g}^{\bullet}_G(A)$ a differential graded Lie algebra. 
%\begin{remark}
%The cup product on $C^{\bullet}(A,A)$
%\begin{equation}
%(D\smile E)(a_1,\dots, a_{n+m})=(-1)^{nm}D(a_1, \dots, a_n)E(a_{n+1},\dots, a_{n+m})
%\end{equation}
%is associative and is compatible with the differential $\delta$, then $C^{\bullet}(A,A)$ is a differential graded algebra.

\subsection*{Formality Theorem}

As we mentioned above, Kontsevich's main result is that 
$\mathfrak{g}^{\bullet}_G(A)$ is a formal DGLA (see Def. \ref{def: fdgla}). This result relies on the existence of a previous result by Hochschild, Kostant and Rosenberg \cite{HKR} which establishes the existence of an isomorphism between the cohomology of the algebra of multidifferential operators and the algebra of multivector fields.

\begin{theorem}[Hochschild-Kostant-Rosenberg \cite{HKR}]\label{hkr}
The formula
\begin{equation}
D_{\pi}(a_1,\dots, a_n)=\langle\pi, da_1\dots da_n\rangle
\end{equation}
defines a quasi-isomorphism
\begin{equation}
(\Gamma(T,\wedge^{\bullet}T),0)\rightarrow C^{\bullet}(C^{\infty}(M),C^{\infty}(M))
\end{equation}
In particular, the cohomology groups $ H^{\bullet}(C^{\infty}(M),C^{\infty}(M))$ is isomorphic to
\begin{equation}
\Gamma(T,\wedge^{\bullet}T),
\end{equation}
where the bracket induced by $[\cdot,\cdot]_G$ becomes the Schouten bracket $[\cdot,\cdot]_S$.
\end{theorem}

The last tool we need is the notion of $L_{\infty}$-quasi isomorphism.
Let $L_1$ and $L_2$ be two DGLA. By definition, an $\mathbf{L_{\infty}}$-\textbf{morphism} $f:L_1\rightarrow L_2$ is given by a sequence of maps
\begin{equation}
f_n:L_1^{\otimes n}\rightarrow L_2,\quad n\geq 1,
\end{equation} 
homogeneous of degree $1-n$ and such that the following conditions are satisfied:
\begin{enumerate}
	\item The morphism $f_n$ is graded antisymmetric, i.e. we have
\begin{equation}
f_n(x_1,\dots, x_i,x_{i+1},\dots x_n)=-(-1)^{\vert x_i\vert\vert x_{i+1}\vert}f_n(x_1,\dots, x_{i+1},x_{i},\dots x_n)
\end{equation}
for all homogeneous $x_1,\dots, x_n$ of $L_1$. 
	\item We have $f_1\circ d=d\circ f_1$ i.e. the map $f_1$ is a morphism of complexes.
	\item $f_1$ is compatible with the brackets up to a homology given by $f_2$. In particular, $f_1$ induces a morphism of graded 
	Lie algebras from $H^{\bullet}(L_1)$ to $H^{\bullet}(L_2)$.
	\item More generally, for any homogeneous element $x_1,\dots, x_n$ of $\mathfrak{g}^{\bullet}$,
\begin{equation}
	\begin{split}
		\sum \pm & f_{q+1}([x_{i_1},\dots, x_{i_p}]_p,x_{j_1},\dots, x_{j_q})= \\
& \sum\pm \frac{1}{k!}[f_{n_1}(x_{i_{11}},\dots, x_{i_{1n_1}}),\dots,f_{n_k}(x_{i_{k1}},\dots, x_{i_{kn_k}})]
	\end{split}
\end{equation}
\end{enumerate} 
Roughly, an $L_{\infty}$-morphism is a map between DGLA which is compatible with the brackets up to a given coherent system
of higher homotopies. 

An $\mathbf{L_{\infty}}$-\textbf{quasi isomorphism} is an $L_{\infty}$-morphism whose first components is a quasi-isomorphism.

Kontsevitch's Formality Theorem can be stated as follows:

\begin{theorem}[Kontsevich \cite{Ko}]
There exists natural $L_{\infty}$ quasi-isomorphism
\begin{equation}
K:\mathfrak{g}_S^{\bullet}(M)\rightarrow \mathfrak{g}_G^{\bullet}(C^{\infty}(M))
\end{equation}
The component $K_1$ of $K$ coincides with the quasi-isomorphism defined in the Hochschild-Kostant-Rosenberg Theorem \ref{hkr}.
\end{theorem}

Kontsevich's formality map  induces a one-to-one map from formal Poisson structures on $M$
to star products on $C^{\infty}(M)$.

\section{Quantization of a Poisson Lie group}\label{sec_3.2}

The theory of Kontsevich provides a procedure to quantize a Poisson manifold; we now introduce 
a theory of quantization for Poisson Lie groups and Lie bialgebras. As defined in Section \ref{sec: poisson lie groups} a Poisson Lie group is a 
Poisson manifold endowed with a Lie group structure. The quantization of these structures can be done using the formalism of  quantum groups; this allows us to 
deform the Poisson manifold and group structures in a compatible way. 
More precisely, given a Poisson Lie group or a Lie bialgebra, an associated Hopf algebra can be defined and deformed to obtain the correspondent quantum group. 
In the following we introduce the definitions of Hopf algebra and Hopf algebra action and we explain how to quantize them. The interested reader can consult the standard books about quantum groups e.g. \cite{C} and \cite{Ma} for details.

\subsection{Hopf algebras}\label{sec_3.2.1}

An algebra with unit over a commutative ring $k$ is a $k$-module $A$ with a multiplication, bilinear over $k$ and associative, and with the unit element $1$ such that $a\cdot 1=1\cdot a=a$ for all $a\in A$. We reformulate this definition in terms of commutative diagrams.

\begin{definition}
An \textbf{algebra} over a commutative ring $k$ is a $k$-module $A$ equipped with $k$-module maps $m^A:A\otimes_k A\rightarrow A$, the multiplication, 
and $\iota^A:A\rightarrow A$, the unit, such that the following diagrams commute:

\begin{center}
$\begin{CD} 
A\otimes k @>id\otimes\iota >> A\otimes A\\ 
@VV\cong V @VVmV\\ 
A @>id>> A 
\end{CD}
\qquad\qquad 
\begin{CD} 
k\otimes A @>\iota\otimes id >> A\otimes A\\ 
@VV\cong V @VVmV\\ 
A @>id>> A 
\end{CD} $
\end{center}

\medskip

\begin{center}
$
\begin{CD} 
A\otimes A\otimes A @>m\otimes id >> A\otimes A\\ 
@VVid\otimes mV @VVmV\\ 
A\otimes A @>m>> A 
\end{CD}
$
\end{center}
\end{definition}

In terms of the traditional description of an algebra we have 
\begin{equation}
\iota(\lambda)=\lambda 1, \qquad m(a_1\otimes a_2)=a_1\cdot a_2.
\end{equation}
The first two diagrams express the properties of the unit element and the third the associativity of multiplication.

An algebra is commutative if the following diagram commutes
$$
\begin{CD} 
A\otimes A @>\tau >> A\otimes A\\ 
@VV mV @VVmV\\ 
 A @>id>> A 
\end{CD}
$$
where $\tau:A\otimes A\rightarrow A\otimes A$ is the flip map $\tau(a_1\otimes a_2)=a_2\otimes a_1$. If we set $m_{op}=m\circ\tau$, then $(A,\iota, m_{op})$ is the opposite algebra of $A$.

\begin{definition}
A \textbf{coalgebra} over a commutative ring $k$ is a $k$-module $A$ equipped with $k$-module maps $\Delta^A:A\rightarrow A\otimes A$, the coproduct, and $\epsilon: A\rightarrow k$, the counit, 
such that the following diagrams commute:
$$
\begin{CD} 
A\otimes k @<id\otimes\epsilon << A\otimes A\\ 
@AA\cong A @AA\Delta A\\ 
A @<id<< A 
\end{CD}
\qquad\qquad
\begin{CD} 
k\otimes A @<\epsilon\otimes id << A\otimes A\\ 
@AA\cong A @AA\Delta A\\ 
A @<id<< A 
\end{CD} 
$$

\medskip

$$
\begin{CD} 
A\otimes A\otimes A @<\Delta\otimes id << A\otimes A\\ 
@AA id\otimes \Delta A @AA\Delta A\\ 
A\otimes A @<\Delta << A 
\end{CD}
$$
\end{definition}

The commutativity of the third diagram is usually referred to as the coassociativity of $A$. The coalgebra $A$ is called cocommutative if the following diagram commutes:
$$
\begin{CD} 
A\otimes A @<\tau << A\otimes A\\ 
@AA\Delta A @AA\Delta A\\ 
 A @<id<< A 
\end{CD}
$$

If we set $\Delta^{op}=\tau\circ\Delta$, then $(A,\epsilon, \Delta^{op})$ is the opposite coalgebra.

Given two coalgebras $A$ and $B$, a $k$-module map $\varphi:A\rightarrow B$ is a coalgebra homomorphism if 

\begin{equation}
(\varphi\otimes\varphi)\circ\Delta^A=\Delta^B\circ\varphi, \quad \epsilon^B\circ\varphi=\epsilon^A.
\end{equation}

A Hopf algebra has compatible algebra and coalgebra structures and one extra structure map.

\begin{definition}
A \textbf{Hopf algebra} over a commutative ring $k$ is a $k$-module $A$ such that

\begin{enumerate}
\item $A$ is both an algebra and a coalgebra over $k$;
\item the coproduct $\Delta:A\rightarrow A\otimes A$ and the counit $\epsilon:A\rightarrow k$ are algebra homomorphisms;
\item the product $m:A\otimes A$ and the unit $\iota:k\rightarrow A$ are coalgebra homomorphisms;
\item $A$ is equipped with a bijective $k$-module map $S^A:A\rightarrow A$ called the antipode, such that the following diagrams commute:
$$
\begin{CD} 
A\otimes A @>S\otimes id >> A\otimes A\\ 
@AA\Delta A @VVmV\\ 
A @>\iota\circ\epsilon >> A 
\end{CD}
\qquad 
\begin{CD} 
A\otimes A @>id\otimes S >> A\otimes A\\ 
@AA\Delta A @VVmV\\ 
A @>\iota\circ\epsilon>> A 
\end{CD} 
$$
\end{enumerate}
\end{definition}

If $A$ and $B$ are Hopf algebras, a $k$-module map $\varphi:A\rightarrow B$ is a Hopf algebra homomorphism if it is a homomorphism of both the algebra and the coalgebra structures of $A$.

Let us consider two crucial examples.

\begin{example}\label{ex_3.2.1_cg}
Let $G$ be a compact topological group. Consider the space $C(G)$ of the continuous functions on $G$ together with the
 following maps:
\begin{itemize}
\item[-] $(f\cdot h)(g)=f(g)h(g)$
\item[-] $\Delta (f)(g_1\otimes g_2)=f(g_1g_2)$
\item[-] $\iota (x)=x1$ where $1(g)=1$ for any $g\in G$
\item[-] $\epsilon(f)=f(e)$ where $e$ is the unit element of $G$
\item[-] $S(f)(g)=f(g^{-1})$ 
\end{itemize}
where $g_1,g_2,g\in G$, $x\in k$ and $f,h\in C(G)$. The set $C(G)$ together with these maps is a Hopf algebra.
\end{example}

\begin{example}\label{ex_3.2.1_ug}
Let $\mathfrak{g}$ be a Lie algebra and $\mathcal{U}(\mathfrak{g})$ its universal enveloping algebra, then $\mathcal{U}(\mathfrak{g})$ becomes a Hopf algebra when
\begin{itemize}
	\item[-] the ordinary multiplication in $\mathcal{U}(\mathfrak{g})$
	\item[-] $\Delta(x)=x\otimes 1+1\otimes x$
	\item[-] $\iota (\alpha)=\alpha1$ 
	\item[-] $\epsilon(1)=1$ and zero on all the other elements
	\item[-] $S(x)=-x$ 
\end{itemize}
where $x\in\mathfrak{g}$ is considered as a subset of $\mathcal{U}(\mathfrak{g})$. To be precise, this defines $\Delta$, $\iota$, $\epsilon$ and $S$ only on the subset $\mathfrak{g}$ of the universal enveloping algebra, but these maps can be extended uniquely to all $\mathcal{U}(\mathfrak{g})$ such that the Hopf algebra axioms are satisfied everywhere.
\end{example}

These examples are in a sense dual each other. In general, suppose that $(A,m,\Delta,\iota,\epsilon, S)$ is a Hopf algebra and $A^*$ is its dual space; then using the structure maps of $A$ we define the structure maps $(m^*,\Delta^*,\iota^*,\epsilon^*, S^*)$ as follows:
\begin{itemize}
\item[-] $\langle m^*(f\otimes g),x\rangle= \langle f\otimes g, \Delta (x)\rangle$
\item[-] $\langle\Delta^*(f), x\otimes y\rangle=\langle f,xy\rangle$
\item[-] $\langle\iota^*(\alpha),x\rangle=\alpha\cdot\epsilon(x)$
\item[-] $\epsilon^*(f)=\langle f,1\rangle$
\item[-] $\langle S^*(f),x\rangle=\langle f, S(x)\rangle$
\end{itemize}
where $f,g\in A^*$ an $x,y\in A$. The brackets $\langle\cdot,\cdot\rangle$ are the pairing between $A^*$ and $A$ and $\langle f\otimes g, x\otimes y\rangle=\langle f,x\rangle \langle g,y\rangle$. It is 
easy to see that $A^*$ is also a Hopf algebra. 
%The duality between $\mathcal{U}(\mathfrak{g})$ and $C^{\infty}(G)$ is discussed in the Appendix \ref{app_dual}.

Since we are interested in the quantization of Poisson Lie groups, we need to introduce the concept of Poisson Hopf algebra. Recall that a Poisson algebra (\ref{1.2_pa}) is an algebra equipped with a
 bilinear map $\lbrace \cdot,\cdot\rbrace: A\otimes A\rightarrow A$ such that $(A,\lbrace \cdot,\cdot\rbrace)$ is a Lie algebra and $\lbrace \cdot,\cdot\rbrace$ is a derivation.

\begin{definition}\label{def: pha}
A Poisson algebra $(A,\lbrace \cdot,\cdot\rbrace)$ is called a \textbf{Poisson Hopf algebra} if it is also a Hopf algebra $(A,m,\Delta,\iota,\epsilon, S)$ over $k$ such that both structure are 
compatible, i.e.
\begin{equation}
\Delta\left( \lbrace a_1,a_2\rbrace\right) =\lbrace \Delta(a_1), \Delta(a_2)\rbrace_{A\otimes A}
\end{equation}
for all $a_1,a_2\in A$. Here the Poisson bracket $\lbrace \cdot, \cdot\rbrace_{A\otimes A}$ is defined as
\begin{equation}
\lbrace a\otimes a', b\otimes b'\rbrace_{A\otimes A}=\lbrace a,b\rbrace\otimes a'b'+ab\otimes \lbrace a',b'\rbrace
\end{equation}
for all $a,a',b,b'\in A$.
\end{definition}

Given a Poisson Lie group $(G,\pi)$, its Poisson algebra $(C^{\infty}(G), \lbrace\cdot,\cdot\rbrace)$ is a Poisson Hopf algebra with the Hopf structure given in Example \ref{ex_3.2.1_cg}.

For the quantization of the Lie bialgebras, we need the dual version of the above definition:

\begin{definition}\label{def: cpha}
A \textbf{co-Poisson Hopf algebra} is a co-commutative Hopf algebra $(A,m,\Delta,\iota,\epsilon, S)$ together with a map
\begin{equation}
\delta: A\rightarrow A\otimes A
\end{equation}
such that
\begin{enumerate}
\item $\tau\circ\delta=-\delta$ co-antisymmetry
\item $(1\otimes 1\otimes 1+(1\otimes\tau)(\tau\otimes 1))(1\otimes\delta)\delta=0$ co-Jacobi identity
\item $(\Delta\otimes 1)\delta=(1\otimes 1\otimes 1+\tau\otimes 1)(1\otimes \delta)\Delta$ co-Leibniz rule
\item $(m\otimes m)\circ \delta_{A\otimes A}=\delta\circ m$ m is co-Poisson homomorphism
\end{enumerate}
where $\delta_{A\otimes A}=(1\otimes\tau\otimes 1)(\delta\otimes\Delta+\Delta\otimes\delta)$ is the co-Poisson structure naturally associated to the tensor product space
\end{definition}

%As we said, a Poisson structure on $C^{\infty}(G)$ induces a co-Poisson structure on $\mathcal{U}(\mathfrak{g})$ because of the duality between these two spaces. We conclude that
The universal enveloping algebra of a Lie biagebra $\mathfrak{g}$ is a co-Poisson Hopf algebra with Hopf structure defined in Example \ref{ex_3.2.1_ug}.

\subsection{Quasi triangular Hopf algebras}
In the following we discuss a particular class of Hopf algebras and, in analogy with the classical case, the quantum Yang-Baxter equation.

As already mentioned, a Hopf algebra $H$ is cocommutative is $\tau\circ\Delta=\Delta$. Here we consider Hopf algebras that are only cocommutative up to conjugation by an element $R\in H\otimes H$. This element $R$ is called the quasi triangular structure.

\begin{definition}
A \textbf{quasi triangular} Hopf algebra is a pair $(H,R)$, where $H$ is a Hopf algebra and $R\in H\otimes H$ is invertible and such that
\begin{equation}
(\Delta\otimes id)R=R_{13}R_{23},\quad (id\otimes \Delta)R=R_{13}R_{12}
\end{equation}
\end{definition}

For sake of completeness we record  the following two lemmas, in analogy with the classical case.

\begin{lemma}
If $(H,R)$ is a quasitriangular bialgebra, then $R$ as an element of $H\otimes H$ obeys
\begin{equation}
(\epsilon\otimes id)R=(id\otimes \epsilon)R=1
\end{equation}
If $H$ is a Hopf algebra then one also has
\begin{equation}
(S\otimes id)R=R^{-1},\quad (id\otimes S)R^{-1}=R
\end{equation} 
and hence $(S\otimes S)R=R$.
\end{lemma}

\begin{lemma}
Let $(H,R)$ be a quasitriangular bialgebra. Then 
\begin{equation}
R_{12}R_{13}R_{23}=R_{23}R_{13}R_{12}
\end{equation}
is the quantum Yang-Baxter equation.
\end{lemma}

\subsection{Hopf algebra actions}
In this section we introduce the notion of Hopf algebra action. This will be used in the next sections to define a quantized action and a quantum momentum map.

\begin{definition}
For an algebra $A$, a (left) \textbf{A-module} is a $k$-space $M$ with a $k$-linear map $\gamma:A\otimes M\rightarrow M$ such that $\gamma (m\otimes id)=\gamma(id\otimes m)$ and $\gamma (u\otimes id)=$ scalar multiplication.
\end{definition}
We have that $A$ acts on the $k$-space $M$ if $M$ is a left $A$-module. The action is given by the map $\gamma$. The dual notion is the co-action of a coalgebra:
\begin{definition}
For a coalgebra $C$, a (right) \textbf{C-comodule} is a $k$-space $M$ with a $k$-linear map $\rho: M\rightarrow M\otimes C$ such that $(id\otimes \Delta)\rho=(\Delta\otimes id)\rho$ and $(id\otimes \epsilon)\rho=$ tensoring with 1. 
\end{definition}
We say that $\rho$ is a coaction of $C$ with $M$. Let us define the maps which preserve the module and comodule structures on the corresponding spaces.
\begin{definition}
Let $A$ be an algebra and $C$ a coalgebra.
\begin{enumerate}
\item Let $M$ and $N$ be (left) $A$-modules with structure maps $\gamma_M$ and $\gamma_N$ respectively. A map $f: M\rightarrow N$ is called an \textbf{A-module map} if $f\circ \gamma_M=\gamma_N\circ (id\otimes f)$.
\item Let $M$, $N$ be a (right) $C$-comodules, with structure maps $\rho_M$ and $\rho_N$ respectively.
A map $f:M\rightarrow N$ is called a \textbf{C-comodule map} if $\rho_N\circ f=(f\otimes id)\circ \rho_M$.
\end{enumerate}
\end{definition}

Finally, the Hopf algebra actions are defined as follows:
\begin{definition}\label{def: haa}
Let $H$ be a Hopf algebra. An algebra $A$ is a (left) \textbf{H-module algebra} if: 
\begin{enumerate}
\item $A$ is a (left) $H$-module via $h\otimes a\mapsto \gamma(h)(a)$ 
\item $\gamma(ab)=m(\gamma\otimes \gamma)(\Delta h)(a\otimes b)$ for any $a$ and $b$ in $A$
\item $\gamma(h)1_A=\epsilon(h)1_A$
\end{enumerate}
\end{definition}
In this case we have a Hopf algebra action of $H$ on $A$ if the algebraic structures of $A$ is compatible with this action. Similarly, a Hopf algebra co-action is defined by:
\begin{definition}
An algebra $A$ is a (right) \textbf{H-comodule algebra} if
\begin{enumerate}
\item $A$ is a (right) $H$-comodule via  $\rho:A\rightarrow A\otimes H$ for any $a$ and $b$ in $A$
\item $\rho(ab)=m(\Delta a, \Delta b)$
\item $\rho(1_A)=1_A\otimes 1_A$
\end{enumerate}
\end{definition}

\subsection{Quantization of Poisson structures}\label{sec_3.2.2}
In this section we discuss the quantization of a Lie bialgebra $\mathfrak{g}$ as studied in  \cite{Re} by Reshetikhin and
in \cite{EK} by Etingof and Kazhdan. We will see that the quantization of $\mathfrak{g}$
is provided by the quantum universal enveloping algebra $\mathcal{U}_{\hbar}(\mathfrak{g})$.
 Recall that, given a Poisson algebra $(A, \lbrace\cdot,\cdot\rbrace)$, its quantization is given
by defining a star product as in eq. (\ref{eq: star}). Then, we describe the quantization of a Poisson Hopf algebra, obtained as deformation of the Poisson algebra and the Hopf algebra structures.

\begin{definition}[Quantization of Hopf algebra]
A deformation of a Hopf algebra $(A,\iota, m, \epsilon, \Delta, S)$ over a field $k$ is a Hopf algebra $(A_{\hbar},\iota_{\hbar}, m_{\hbar}, \epsilon_{\hbar}, \Delta_{\hbar}, S_{\hbar})$ over the ring $k  \llbracket \hbar  \rrbracket $ of formal power series such that
\begin{enumerate}
\item $A_{\hbar}$ is isomorphic to $A  \llbracket \hbar  \rrbracket $ as a $k  \llbracket \hbar  \rrbracket $-module;
\item $m_{\hbar}\equiv m$ (mod $\hbar$) and $\Delta_{\hbar}\equiv\Delta$ (mod $\hbar$).
\end{enumerate}
\end{definition}

Next, let us define the quantization of a Poisson Hopf algebra:
\begin{definition}[Quantization of Poisson Hopf algebra]
Let $A$ be a Poisson Hopf algebra over $k$. A quantization of $A$ is a Hopf algebra deformation $A_{\hbar}$
such that
\begin{equation}
\lbrace q(a),q(b)\rbrace= q\left(\frac{[a,b]_{\hbar}}{\hbar}\right) \quad \forall a,b\in A
\end{equation}
where $q$ is the canonical quotient map $q:A_{\hbar}\rightarrow A_{\hbar}/\hbar A_{\hbar}\cong A$ and $[\cdot,\cdot]_{\hbar}$ is the usual commutator with respect to $m_{\hbar}$.
\end{definition} 

If $G$ is a Poisson Lie group, then the algebra of functions on $G$ is a Hopf algebra and the two structure are compatible as in Definition \ref{def: pha}. A quantization of a Poisson Lie group $G$ is a quantization of this Poisson Hopf algebra.

\begin{definition}[Quantization of co-Poisson Hopf algebra]\label{def_3.2.2_dcp}

Consider the co-Poisson Hopf algebra $(A,m,\Delta, \iota, \epsilon, S; \delta)$. A quantization of $A$ is a non-commutative Hopf algebra $(A_{\hbar},m_{\hbar},\Delta_{\hbar}, \iota_{\hbar},\epsilon_{\hbar}, S)$ over $k  \llbracket \hbar  \rrbracket $ such that

\begin{enumerate}
\item $A_{\hbar}/\hbar A_{\hbar}\cong A$
\item $m\circ (q\otimes q)=q\circ m_{\hbar}$
\item $q\circ \iota_{\hbar}=\iota$
\item $\epsilon_{\hbar}\circ q=\epsilon$
\item $\delta(q(a))=q(\frac{1}{\hbar}(\Delta_{\hbar} (a)-\tau\circ\Delta_{\hbar}(a)))$ for all $a\in A$.
\end{enumerate}
\end{definition}

The product and coproduct need to be related by the following condition:

\begin{equation}\label{3.2.2_ax}
\Delta_{\hbar}(a\star b)=\Delta_{\hbar}(a)\star \Delta_{\hbar}(b).
\end{equation}

Let us consider the universal enveloping algebra $\mathcal{U}(\mathfrak{g})$ of the Lie bialgebra $\mathfrak{g}$. The Poisson structure on $G$ induces a co-Poisson structure $\delta$ on its Lie bialgebra $\mathfrak{g}$, that can be easily extended to $\mathcal{U}(\mathfrak{g})$. Using the Hopf structure discussed in Example \ref{ex_3.2.1_ug} we get that $(\mathcal{U}(\mathfrak{g}),m,\Delta,\iota, \epsilon;\delta)$ is a co-Poisson structure. From Definition \ref{def_3.2.2_dcp}, its quantization is a non-commutative Hopf algebra $(\mathcal{U}_{\hbar}(\mathfrak{g}),m_{\hbar},\Delta_{\hbar},\iota_{\hbar}, \epsilon_{\hbar})$ with
\begin{equation}\label{3.2.2_br}
  \left[ \cdot,\cdot\right]_{\star}=\sum_{k=0}^{\infty} \hbar^{k+1} F_k(\cdot,\cdot)
\end{equation}
where $F_0(\cdot,\cdot)$ is the standard commutator, and
\begin{equation}\label{3.2.2_cp}
  \Delta_{\hbar}=\sum_k \hbar^k \Delta_k
\end{equation}
where $\Delta_0$ is the coproduct of $\mathcal{U}(\mathfrak{g})$ and the antisymmetrization of $\Delta_1$ is given by the structure $\delta$.

\begin{example}[Quantization of $\mathcal{U}(\mathfrak{sl})(2)$ \cite{Tj}]

Let us consider the Lie algebra $\mathfrak{sl}(2)$ with basis ${H,E,F}$ and commutation relations

\begin{equation}
\left[ H,E\right]=2E,\quad \left[ H,F\right]=-2F\quad \left[ E,F\right]=H.
\end{equation}

The co-Poisson structure of $\mathcal{U}(\mathfrak{sl}(2))$ is given by the extension to the whole universal enveloping algebra of the map

\begin{align}
\delta(H) &=0\\
\delta(E) &=\frac{1}{2}E\wedge H\\
\delta(F) &=\frac{1}{2}F\wedge H
\end{align}

The quantized space is the set of (formal) polynomials in $\hbar$ with coefficients 
in $\mathcal{U}(\mathfrak{sl}(2))$. The coproduct $\Delta_{\hbar}$ on this new space is determined by the following requirements: 

\begin{enumerate}
\item $\Delta_{\hbar}$ is co-associative
\item $\delta(q(a))=q(\frac{1}{\hbar}(\Delta_{\hbar} (a)-\tau\circ\Delta_{\hbar}(a)))$
\item In the classical limit $\hbar\rightarrow 0$ the coproduct $\Delta_{\hbar}$ reduces to the ordinary coproduct on $\mathcal{U}(\mathfrak{sl}(2))$
\end{enumerate}

The coproduct $\Delta_{\hbar}$ has the general form
\begin{equation}
\Delta_{\hbar}=\sum_{n=0}^{\infty}\frac{\hbar^n}{n!}\Delta_n.
\end{equation}
where the term $\Delta_n$ for an arbitrary $n$ and the coproduct of the quantized universal enveloping algebra is
\begin{align}
\Delta_{\hbar}(H) &=H\otimes 1+1\otimes H\\
\Delta_\hbar(E) &=E\otimes 1+q^{-H}\otimes E\\
\Delta_\hbar(F) &=F\otimes q^H+1\otimes F.
\end{align}
with $q= e^{\frac{\hbar}{4}}$. Now, imposing the condition (\ref{3.2.2_ax}) we get the commutation relations
\begin{align}
\left[ H,E\right]_{\star}&=2E\\
\left[ H,F\right]_{\star}&=-2F\\
\left[ E,F\right]_{\star}&=[H]_q
\end{align}
where $[H]_q=\frac{q^{2H}-q^{-2H}}{q-q^{-1}}$. Finally, we have
\begin{align}
\epsilon_{\hbar}(E)&=\epsilon_{\hbar}(F)=\epsilon_{\hbar}(H)=0\\
\epsilon_{\hbar}(1)&=1
\end{align}
and 
\begin{align}
S_{\hbar}(E)&=-qE\\
S_{\hbar}(F)&=-q^{-1}F\\
S_{\hbar}(H)&=-H.
\end{align}
This Hopf algebra is called the quantum universal enveloping algebra of $\mathfrak{sl}(2)$ and is denoted by $\mathcal{U}_{\hbar}(\mathfrak{sl}(2))$.
\end{example}

\section{Quantum Momentum Map}

The problem of quantizing the momentum map and the theory of reduction has been the main topic of many works, e.g. \cite{Fe1} and \cite{Lu4}. In the following we discuss two different methods that have been proposed to approach it.

The first one, due to Fedosov \cite{Fe1}, uses deformation quantization. In this approach,  given a canonical action of a Lie group $G$ on a symplectic manifold $M$, the quantum momentum map is defined as a Lie algebra homomorphism 
$\boldsymbol{\mu}_{\hbar}$ from the Lie algebra $\mathfrak{g}$ into the deformed algebra $C^{\infty}_{\hbar}(M)$. The corresponding quantum action is given by the quantization of the Hamiltonian vector field induced by $\boldsymbol{\mu}$. We notice that in this approach there is no quantization of the group. Fedosov defined the quantum reduced space as
\begin{equation}\label{eq: qrs1}
 C^{\infty}_{\hbar}(M)^G/\mathcal{I}_{\hbar},
\end{equation}
where $C^{\infty}_{\hbar}(M)^G$ is the set of the functions in $C^{\infty}_{\hbar}(M)$ which are invariant under the quantized action. Here $\mathcal{I}_{\hbar}$ is the ideal generated by the components $\boldsymbol{\mu}^{i}_{\hbar}$ of the quantum momentum map $\boldsymbol{\mu}_{\hbar}$.
Furthermore he proved that, under the assumptions of the Marsden-Weinstein Theorem, the quantum reduced algebra in eq. (\ref{eq: qrs1}) 
is isomorphic to the algebra obtained by canonical deformation quantization of $C^{\infty}(M_{\xi})$. 

A different approach has been developed in \cite{Lu4} by Lu. In this case the quantization procedure is carried on via quantum group techniques. Here the author considers a Hopf algebra $H$ with dual $H^*$ and assumes that $H$ is the quantization of the Poisson Lie group $G$.
Given an action $\Phi:H^*\otimes V\rightarrow V$ of $H^*$ on the algebra $V$, then the quantum momentum map is defined as an algebra homomorphism $\boldsymbol{\mu}: H^*\rightarrow V$, provided that the action $\Phi$ can be rewritten in terms of $\boldsymbol{\mu}$. We stress that this approach does not guarantee that the quantum action $\Phi$ is the quantization of the given Poisson action.

The idea we discuss in this section is the generalization of the deformation quantization approach to the Poisson reduction case. We  use quantum group techniques to quantize Poisson Lie groups and the 
quantum momentum map is basically defined as a linear map from the quantum group $\mathcal{U}_{\hbar}(\mathfrak{g})$ to the deformed algebra $C^{\infty}_{\hbar}(M)$.
The induced quantum action will be a Hopf algebra action, as defined in Definition \ref{def: haa}. This will allows us to discuss some
 examples of quantum momentum map and quantum reduction.

\subsection{Quantization of the momentum map}
\label{sec: def}

The main goal of this section is the quantization of the momentum map as defined in (\ref{def: mm}). 
Basically, we consider a Poisson action of $G$ on $M$, introduce the quantization of the structures using the techniques 
discussed in the previous sections and we give a definition of the corresponding quantum action in terms of Hopf algebra
 action.  The quantum momentum map will be defined as the map which factorizes such an action. 

Let $C^{\infty}_{\hbar}(M)$ be a deformation quantization of $(M,\pi)$ and let us denote
\begin{align}
& m_{\star} :C^\infty_\hbar (M) \times C^\infty_\hbar (M) \rightarrow C^\infty_\hbar (M)\\
&[f,g]_{\star}=\sum_{n=0}^{\infty} P_n(f,g), \quad\forall f,g\in C^{\infty}(M) 
\end{align}
the star-product and the deformed bracket in $C^{\infty}_{\hbar}(M)$.

As discussed in the previous section, $\mathcal{U}_{\hbar}(\mathfrak{g})$ denote the deformation quantization of the universal enveloping algebra of $\mk{g}$ and we denote with $m_{\hbar}$ and $\Delta_{\hbar}$ the deformed 
product and coproduct on $\mathcal{U}_{\hbar}(\mathfrak{g})$, resp. 

Given the quantization of all the structures, we define the quantum action as follows:
\begin{definition}
Given the infinitesimal generator $\Phi: \mathfrak{g}\rightarrow TM$ of a Poisson action of $(G,\pi_G)$ on $(M,\pi)$, the corresponding \textbf{quantum action} is the linear map
\begin{equation}\label{eq: qa}
 \Phi_{\hbar}  : \mathcal{U}_\hbar (\mk{g})\rightarrow End\; C^\infty_\hbar (M): \xi\mapsto  \Phi_{\hbar}(\xi)(f)
\end{equation}
continuous with respect to $C^\infty$-topology and such that it defines a Hopf algebra action, i.e. such that
\begin{equation}\label{eq: hac}
 \Phi_{\hbar}  (\xi)(f\star g)=m_{\star} ( \Phi_{\hbar} \otimes \Phi_{\hbar}  \circ \Delta_\hbar (\xi)(f\otimes g)).
\end{equation}
and
\begin{equation}
[\Phi_{\hbar}  (\xi),\Phi_{\hbar}  (\eta)](f)=\Phi_{\hbar}  ([\xi,\eta])(f)
\end{equation}
\end{definition}

An example can be constructed when the deformation of $\mathcal{U} (\mk{g})$ comes from a  twist $\tau\in (\mathcal{U} (\mk{g})\otimes \mathcal{U} (\mk{g}))  \llbracket \hbar  \rrbracket $ such
that $\Delta_{\hbar}=\tau\Delta\tau^{-1}$ is associative and the corresponding associator is trivial \cite{Xu1}. In fact, one can define a deformation quantization of $M$
setting, for any differential operator $\mathcal{D}$ on $M$.
\begin{equation}
\mathcal{D}(f\star g)=m ((\mathcal{L}_t \Delta (\mathcal{D}))f\otimes g),
\end{equation}
where $\Delta$ is the usual coproduct on differential operators. The $\star$-product defined by this formula is automatically consistent with the action of deformed Hopf algebra $\mathcal{U} (\mk{g})$.

In the following we define a quantum momentum map which, analogously to the classical case, factorizes the quantum action 
(\ref{eq: qa}). Let us recall from the previous chapter the composition of Lie algebra homomorphisms which define the momentum map:
\begin{align}\label{dia: cla1}
 \mathfrak{g} &\longrightarrow \Omega^1(M)\longrightarrow TM\\
 \xi &\longmapsto \alpha_{\xi}\longmapsto \pi^{\sharp}(\alpha_{\xi})
\end{align}

Using the arguments of Section \ref{sub: structure}, it should be clear that a quantum
momentum map can be defined as a quantization of the infinitesimal momentum map $\alpha$ (\ref{eq: imm}). Recall that, in the classical construction, the map $\pi^\sharp: \Omega^1(M)\rightarrow TM$ is defined by:
\begin{equation}
\pi^\sharp: \Omega^1(M)\ni adb \longmapsto  a\{b,\cdot\}\in TM
\end{equation}
where $a,b\in C^{\infty}(M)$. Hence, the classical construction can be rephrased as follows:
\begin{align}\label{dia: cla2}
 \mathfrak{g} &\longrightarrow C^{\infty}(M)\otimes C^{\infty}(M)\longrightarrow End\; C^{\infty} (M)\\
 \xi &\longmapsto a_{\xi}^i \otimes b_{\xi}^i\longmapsto \sum_i a_{\xi}^i \{b_{\xi}^i,f\}
\end{align} 

This motivates the following definition of quantum momentum map.

\begin{definition}\label{def: qmm1}
A \textbf{quantum momentum map} is defined to be a linear map 
\begin{equation}
\boldsymbol{\mu}_{\hbar}:\mathcal{U}_{\hbar}(\mathfrak{g})\rightarrow C_{\hbar}^{\infty}(M)\otimes C_{\hbar}^{\infty}(M): \xi\mapsto \sum_i a_{\xi}^i\otimes b_{\xi}^i.
\end{equation}
such that it is an algebra homomorphism and
\begin{equation}\label{eq: qa1}
\Phi_{\hbar}  (\xi)=\sum_i a_{\xi}^i \left[b_{\xi}^i,\cdot\right]_{\star}
\end{equation}
is a quantized action.
\end{definition}
To avoid cumbersome notation, we omitted the star notation $\star$ in the star-product. Moreover, we denoted the functions $\hat{a} = a \mod \hbar$, $\hat{b} = b \mod \hbar \in C_{\hbar}^{\infty}(M)$  simply by $a$ and $b$. 

It is easy to see that, the classical action (\ref{dia: cla2}) can be recovered in the limit $\hbar\rightarrow 0$ from 
eq. (\ref{eq: qa1}) and using  eq. (\ref{eq: pbd1}).
In other words, this definition gives the quantization of the construction (\ref{dia: cla2})  as
\begin{align}\label{dia: qua1}
 \mathcal{U}_{\hbar}(\mathfrak{g}) &\longrightarrow C_{\hbar}^{\infty}(M)\otimes C_{\hbar}^{\infty}(M)\longrightarrow End\; C^{\infty}_{\hbar} (M)\\
 \xi &\longmapsto a_{\xi}^i\otimes b_{\xi}^i\longmapsto \frac{1}{\hbar}\sum_i a_{\xi}^i[b_{\xi}^i,f]_{\star}
\end{align}

On the other hand, introducing the space $\Omega^1(\mathcal{A}_{\hbar})$ of differential forms on the algebra $\mathcal{A}_{\hbar}=C^{\infty}_{\hbar} (M)$ and identifying 
$\Omega^1(\mathcal{A}_{\hbar})$ with $\mathcal{A}_{\hbar}\otimes \mathcal{A}_{\hbar}$, we have:

\begin{equation}\label{dia: qua2}
 \mathcal{U}_{\hbar}(\mathfrak{g}) \longrightarrow \Omega^1(\mathcal{A}_{\hbar})\longrightarrow End\; \mathcal{A}_{\hbar}\\
\end{equation}

We can define a non commutative product on the space of differential forms $\Omega^1(\mathcal{A}_{\hbar})$, using  the map 
$\Omega^1(\mathcal{A}_{\hbar})\rightarrow End\; \mathcal{A}_{\hbar}: a db\mapsto a [b,f]$. It associates $[b,[c,f]]$ to the product of 
two closed forms $db\cdot dc$ and we have
\begin{equation}
 [b,[c,f]]= b [c,f]-[c,f] b= b [c,f]-[cb,f]+c [b,f].
\end{equation}
It is clear that the product $db\cdot dc$ on $\Omega^1(\mathcal{A}_{\hbar})$ has to be defined as follows:
\begin{equation}\label{eq: prf}
 db\cdot dc= bdc-d(cb)+cdb.
\end{equation}

As introduced in (\ref{eq: hcc}) the space $End\; \mathcal{A}_{\hbar}$ defines the Hochschild cochain of $\mathcal{A}_{\hbar}$ in itself $C^1(\mathcal{A}_{\hbar},\mathcal{A}_{\hbar})$ with
coboundary $b$. We notice here that from the definition of the product (\ref{eq: prf}), the map
 $\Omega^1(\mathcal{A}_{\hbar})\longrightarrow End\; \mathcal{A}_{\hbar}$ is a Lie algebra homomorphism only if the differential of the unit in $\mathcal{A}_{\hbar}$ does not vanish in $\Omega^1(\mathcal{A}_{\hbar})$, i.e. we work with the formal differential forms on the unitalization $\mathcal{A}_{\hbar}^{+}$ of $\mathcal{A}_{\hbar}$.

These observations allow us to rewrite the definition of quantum momentum map as follows:

\begin{definition}
A \textbf{quantum momentum map} for the quantum action $\Phi_{\hbar}:\mathcal{U}_{\hbar}(\mathfrak{g}) \rightarrow End\; \mathcal{A}_{\hbar}$ is a linear map
\begin{equation}
\boldsymbol{\mu}_{\hbar}:\mathcal{U}_{\hbar}(\mathfrak{g}) \longrightarrow \Omega^1(\mathcal{A}_{\hbar}): \xi\mapsto a_{\xi}^i d b_{\xi}^i
\end{equation}
 such that it is an algebra homomorphism and
\begin{equation}
\Phi_{\hbar}(\xi)(f)=\frac{1}{\hbar}\sum_i a_{\xi}^i[b_{\xi}^i,f]_{\star},
\end{equation}
where $ a_{\xi}^i, b_{\xi}^i\in \mathcal{A}_{\hbar}$.
\end{definition}

In Section (\ref{sub: structure}) we rephrased the classical construction (\ref{dia: cla1}) in terms of Gerstenhaber morphisms,
\begin{equation}\label{dia: cla3}
 \wedge^{\bullet}\mathfrak{g}\rightarrow\Omega^{\bullet}(M)\rightarrow  \wedge^{\bullet}TM
\end{equation}
In order to generalize the quantum construction (\ref{dia: qua1}) in a similar way we first notice that the map $\Omega^1(\mathcal{A}_{\hbar})\rightarrow C^1(\mathcal{A}_{\hbar},\mathcal{A}_{\hbar})$ extends naturally to the map
$\Omega^{\bullet}(\mathcal{A}_{\hbar})\rightarrow C^{\bullet}(\mathcal{A}_{\hbar},\mathcal{A}_{\hbar})$.

 Consider the tensor algebra $T(\mathcal{U}_{\hbar}(\mathfrak{g})[1])$, where the degree of $\xi_1\otimes\dots\otimes \xi_n$ is $n$.
The coproduct on $\mathcal{U}_{\hbar}(\mathfrak{g})$ extends naturally to $T(\mathcal{U}_{\hbar}(\mathfrak{g})[1])$, simply putting
$$\Delta_{\hbar}(\xi_1\otimes\xi_2)=\Delta_{\hbar}(\xi_1)\otimes\xi_2-\xi_1\otimes\Delta_{\hbar}(\xi_2)$$
(i. e. $\Delta$ is extended to an  odd derivation of the tensor algebra).
Then $\Delta_{\hbar}^2=0$ and   $(T(\mathcal{U}_{\hbar}(\mathfrak{g})[1]),\Delta_{\hbar})$ is a complex.
 The action $\mathcal{U}_{\hbar}(\mathfrak{g})\rightarrow C^{1}(\mathcal{A}_{\hbar},\mathcal{A}_{\hbar})$ extends to the cochain map
\begin{align}
& T(\mathcal{U}_{\hbar}(\mathfrak{g})[1])\longrightarrow C^{\bullet}(\mathcal{A}_{\hbar},\mathcal{A}_{\hbar})
\end{align}

These observations motivate the following rephrasing of the definition of quantum momentum map

\begin{definition}\label{def: qmm2}
A \textbf{quantum momentum map} is defined to be a linear map
\begin{equation}
\boldsymbol{\mu}_{\hbar}:T(\mathcal{U}_{\hbar}(\mathfrak{g})[1])\rightarrow \Omega^{\bullet}(\mathcal{A}_{\hbar}):  \xi_1\otimes\dots\otimes \xi_n\mapsto a_1 db_1\otimes\dots\otimes a_n db_n
\end{equation}
such that
\begin{equation}
\Phi_{\hbar}(\xi_1\otimes \dots\otimes \xi_n)(f_1,\ldots ,f_n)=\frac{1}{\hbar^n}a_1[b_1,f_1]\dots a_n[b_n,f_n]
\end{equation}
\end{definition}

\subsection{Examples}
\label{sec: exs}

In this section we apply the construction given above to some explicit examples.
We show that the existence of the quantum momentum map induces the quantization of the Lie algebra. In fact, in the examples studied in this section, the quantization of the Lie algebra is essentially uniquely determined by existence of universal formulas for the quantum momentum map.

\subsubsection*{Two-dimensional case}

Consider the Lie bialgebra $\mathfrak{g}=\mathbb{R}^2$ with generators $\xi,\eta$ and a deformation quantization $C^\infty_\hbar (M)$ of a Poisson manifold $M$. Assume that $\xi$ acts by
\begin{equation}
\Phi_{\hbar}(\xi)=\frac{1}{\hbar}a[b,\cdot\;]
\end{equation}
for some $a,b\in C^\infty_\hbar (M)$.
Let us impose that it is an Hopf algebra action; then we have
\begin{equation}\label{surprise}
\Phi_{\hbar}(\xi)(f g)=\frac{1}{\hbar}a[b,f g]=\frac{1}{\hbar}a[b,f]g+\frac{1}{\hbar}[a,f][b,g]+\frac{1}{\hbar}fa[b,g].
\end{equation}
Suppose that $a$ is invertible, then
$[a,f]=-a[a^{-1},f]a$ and, setting
\begin{equation}\label{eq: ac2}
\Phi_{\hbar}(\eta)=\frac{1}{\hbar}a[a^{-1},\cdot]
\end{equation}
the coproduct which satisfies the condition (\ref{def: haa}) is given by
\begin{equation}
\Delta_{\hbar}( \xi)=\xi\otimes 1-\hbar \;\eta\otimes \xi +1\otimes\xi.
\end{equation}

Similarly, for (\ref{eq: ac2}) we have

\begin{equation}
\begin{split}
\Phi_{\hbar}(\eta)(f g)&=\frac{1}{\hbar}a[a^{-1},f g]\\
&=\frac{1}{\hbar}a[a^{-1},f]g-\frac{1}{\hbar}a[a^{-1},f]a[a^{-1},g]+\frac{1}{\hbar}fa[a^{-1},g].
\end{split}
\end{equation}
hence
\begin{equation}
\Delta_{\hbar} (\eta)=\eta\otimes 1- \hbar\; \eta\otimes \eta+1\otimes \eta.
\end{equation}
Finally we calculate the bracket of the generators to get the deformed algebra structure of $\mathfrak{g}$:
\begin{equation}
\begin{split}\label{eq: br}
\left[\Phi_{\hbar}(\xi),\Phi_{\hbar}(\eta) \right] f &=\frac{1}{\hbar^2}(a[b,a[a^{-1},f]]-a[a^{-1},a[b,f]])\\
&=
a[b,a][a^{-1},f] +aa[b,[a^{-1},f]]-
a[a^{-1},a][b,f] -aa[a^{-1},[b,f]]\\
&=a[b,a][a^{-1},f]+a^2 [[b,a^{-1}],f].
\end{split}
\end{equation}

\begin{remark}
One should notice that the equation (\ref{surprise}) essentially says the following. Given an element in the image of $\Phi_\hbar$ of the form $\frac{1}{\hbar}a[b,\cdot ]$, $a$ is invertible on the support of $b$ and, assuming universality of our formulas, forces the image of $\Phi_\hbar$ also to contain an element of the form $\frac{1}{\hbar}a[a^{-1},\cdot ]$. In the case when ${\mathfrak g}$ is two (or three dimensional), this essentially forces  the formulas for the deformed coproduct in the examples below.
\end{remark}

We obtain different algebra structures that we discuss case by case

\paragraph{Case 1: $[a,b]=0$.}

Under this assumption, from the relation (\ref{eq: br}) we obtain $\left[\Phi_{\hbar}(\xi),\Phi_{\hbar}(\eta) \right] =0$ and imposing that $\Phi_{\hbar}$ is a Lie algebra homomorphism we get
\begin{equation}
[\xi,\eta]=0.
\end{equation}
Hence, the quantum group given by the universal enveloping algebra $\mathcal{U}_{\hbar}(\mathbb{R}^2)$ generated by the commuting elements $\xi,\eta$ with coproduct
\begin{align}
\Delta_{\hbar} (\xi) &=\xi\otimes 1-\hbar \;\eta\otimes \xi +1\otimes\xi\\
\Delta_{\hbar} (\eta) &=\eta\otimes 1- \hbar\; \eta\otimes \eta+1\otimes \eta.
\end{align}
is the deformation quantization of the abelian Lie bialgebra $\mathfrak{g}=\mathbb{R}^2$, with cobracket
\begin{equation}
\delta (\xi)=-\frac{1}{2} \eta\wedge \xi \qquad \delta(\eta)=0.
\end{equation}
The corresponding Poisson Lie group is $(\R^2,\pi)$, where the Poisson bivector is given by
\begin{equation}
    \pi=-\frac{1}{2}x\partial_x\wedge\partial_y
\end{equation}
Setting $a_0=a\mbox{ mod }\hbar$ and $b_0=b\mbox{ mod }\hbar$, the quantum actions
\begin{equation}\label{eq: qa2}
\Phi_{\hbar}(\xi)=\frac{1}{\hbar}a[b,\cdot]\qquad \Phi_{\hbar}(\eta)=\frac{1}{\hbar}a[a^{-1},\cdot]
\end{equation}
give the quantization of the Poisson action of $(\R^2,\pi)$ on $M$ given by
\begin{equation}\label{eq: ca}
\Phi(\xi)=a_0\{b_0,\cdot\}\qquad \Phi(\eta)=a_0\{a^{-1}_0,\cdot\}.
\end{equation}
The Poisson reduction extends to the quantized version immediately. Given $\lambda ,\mu\in\mathcal{A}_{\hbar}$ with $\lambda \neq 0$ and considering the ideal $\mathcal{I}_{\hbar}$ of functions generated by $a-\lambda$ and $b-\mu$, the algebra
\begin{equation}
(C^\infty_\hbar (M) /\mathcal{I}_{\hbar})^{\mathcal{U}_{\hbar}(\mathbb{R}^2)}
\end{equation}
is a quantization of the Poisson algebra
\begin{equation}
\{ a_0 =\lambda , b_0=\mu \}^{\R^2}.
\end{equation}

\paragraph{Case 2: $[a,b]=-\hbar$.}

In this case we have hence $[b,a^{-1}]=a^{-2}\hbar$; using eq. (\ref{eq: br}) we obtain
\begin{equation}
\begin{split}
\left[\Phi_{\hbar}(\xi),\Phi_{\hbar}(\eta) \right] (f)&= Y(f) +a^2[a^{-2},f]\\
&= Y(f)+2a[a^{-1},f]-a^2[a^{-1},[a^{-1},f]]\\
&=3\hbar^2 Y(f)-\hbar Y^2 (f).
\end{split}
\end{equation}
Hence the quantum group $\mathcal{U}_{\hbar}(\mathfrak{g})$
has the following structures
\begin{align}
[\xi,\eta]&=3\eta-\hbar \eta^2\\
\Delta_{\hbar} (\xi)= &=\xi\otimes 1-\hbar \eta\otimes \xi +1\otimes \xi\\
\Delta_{\hbar}( \eta) &=\eta\otimes 1- \hbar \eta\otimes \eta+1\otimes \eta
\end{align}
and defines a deformation quantization of the Lie bialgebra $\mathfrak{g}$ generated by $\xi$ and $\eta$ with
\begin{align}
[\xi,\eta]&=3\eta\\
\delta (\xi)&=-\frac{1}{2} \eta\wedge \xi\\
\delta(\eta)&=0
\end{align}

The action of $\mathfrak{g}$ on $M$ is factorized by the momentum map determined by the forms
\begin{equation}
\boldsymbol{\mu}(\xi)=a_0db_0 \quad \boldsymbol{\mu}(\eta)=d\log (a_0)
\end{equation}
and is given by (\ref{eq: ca}). Its quantization is given by (\ref{eq: qa2}) and it is factorized by the quantum momentum map
\begin{equation}
\boldsymbol{\mu}_{\hbar}(\xi)=adb \quad \boldsymbol{\mu}_{\hbar}(\eta)=d\log (a).
\end{equation}
The quantum reduction is given by
\begin{equation}
(C^{\infty}_{\hbar}(M)[b^{-1}])^{\mathcal{U}_{\hbar}(\mathfrak{g})}
\end{equation}
and it is the quantization of the example discussed in Section \ref{sec: ex}. 

In this case it is easy to check that  $C^{\infty}_{\hbar}(M)[b^{-1}]$ is still an algebra.

\paragraph{Case 3: $[a,b]=-\hbar ba$.}

This example is the direct quantization of the Poisson reduction discussed in Section \ref{sec: ex}.
Similarly we discuss here the different cases that classically give rise to different dressing orbits. First, if $b>0$	
 we can define $b=e^q$ and $a=e^p$, then we recover the previous case, since we get
\begin{equation}
[p,q]=\hbar
\end{equation}	
i.e. the quantum plane. In this case, unfortunately $C^{\infty}_{\hbar}(M)[b^{-1}]$ is not an algebra but we observe that
we easily get a well defined algebra by replacing it with $C^{\infty}_{\hbar}(M)[[\hbar b^{-1}]]$. Hence the quantum reduction is given by
 \begin{equation}
(C^{\infty}_{\hbar}(M)[[\hbar b^{-1}]])^{\mathcal{U}_{\hbar}(\mathfrak{g})}
\end{equation}

If $b=0$ we recover the result of the abelian case.

\subsubsection*{Three-dimensional case}

The second example we discuss here is the Hopf algebra action of $U_{\hbar}(\mathfrak{su}(2))$ on the deformed algebra
 $C^\infty_{\hbar} (M)$. Consider $a,b,c\in C^\infty_{\hbar} (M)$  satisfying
\begin{align}
aba^{-1}&=e^{2\hbar}b\\
aca^{-1}&=e^{-2\hbar}c\\
\left[ b,c\right] &= \frac{\hbar^2}{e^{-\hbar}-e^{\hbar}} a^{-2}-(1-e^{2\hbar})cb
\end{align}
and the generators $\xi,\eta,\zeta$ acting respectively by
\begin{align}
\Phi_{\hbar}(\xi)f&=\frac{1}{\hbar}a[b,f]\\
\Phi_{\hbar}(\eta)f&=\frac{1}{\hbar}[c,f]a\\
\Phi_{\hbar}(\zeta)f&=afa^{-1}.
\end{align}
Then by calculating the commutation relations of these generators and imposing the Lie algebra homomorphism of $\Phi_{\hbar}$ we obtain that $\xi,\eta,\zeta$ satisfy the commutation relations:
\begin{align}
\zeta\xi\zeta^{-1}&=e^{2\hbar}\xi\\
\zeta\eta\zeta^{-1}&=e^{-2\hbar}\eta\\
 \left[ \xi,\eta\right] &=\frac{\zeta^{-1}-\zeta}{e^{-\hbar}-e^{\hbar}}
\end{align}

Checking the condition (\ref{eq: hac}) for any generator, we get
\begin{align}
\Phi_{\hbar}(\zeta)(fg)&=\Phi_{\hbar}(\zeta)(f)\Phi_{\hbar}(\zeta)(g)\\
\Phi_{\hbar}(\xi)(fg)&=\Phi_{\hbar}(\xi)(f)g+ \Phi_{\hbar}(\zeta)(f)\Phi_{\hbar}(\xi)(g)\\
\Phi_{\hbar}(\eta)(fg)&=f\Phi_{\hbar}(\eta)(g)+ \Phi_{\hbar}(\eta)(f)\Phi_{\hbar}(\zeta)^{-1}(g)
\end{align}
hence
\begin{align}
\Delta_{\hbar}(\zeta)&=\zeta\otimes \zeta\\
\Delta_{\hbar}(\xi)&=\xi\otimes 1+\zeta\otimes\xi\\
\Delta_{\hbar}(\eta)&=1\otimes \eta+\eta\otimes\zeta^{-1}.
\end{align}

Finally, we have that $\xi,\eta$ and $\zeta$ generate a Hopf algebra action of $U_{\hbar}(\mathfrak{su}(2))$ on $C^{\infty}_{\hbar}(M)$.

%% file: Conclusions/conclusions.tex
\chapter*{Conclusions}
\addcontentsline{toc}{chapter}{Conclusions}
\markboth{Conclusions}{}

The work described in this thesis suggests a number of extensions and directions for
future work. Here we collect some possibilities.

\section*{Momentum map and Reduction in Poisson geometry}

The contributions of this thesis in the theory of momentum map and reduction in Poisson geometry have been discussed in detail in Chapter \ref{ch: two}. We give here a summary of the main results.

We introduced a definition of momentum map in infinitesimal terms and we proved the theory of reconstruction
	of momentum map from the infinitesimal one (Sections \ref{sub: structure} and \ref{sec: rec}). The reconstruction theory 
	allowed us to prove the existence of the momentum map in two explicit cases.
	
	We studied the uniqueness of the
	 momentum map, proving the Theorem \ref{thm: idef} on the infinitesimal deformations of a momentum map. We analyzed
	 the uniqueness of the momentum map in the case of a compact and semisimple Poisson Lie group acting on a generic
	  Poisson manifold. Finally, in Section \ref{sec: poisson reduction} we introduced  the construction of a theory of Poisson reduction.

These results motivate the study of many open problems and we introduce and briefly discuss in the following those we are
interested in approaching:
\begin{enumerate}
	\item[--] The reconstruction problem has been discussed in the Section \ref{sec: rec} only for the abelian case and the Heisenberg
	 group. An interesting question would be the possibility of extend our result to an arbitrary two-step nilpotent group
	  $G^*$. Moreover, since the computation performed is a kind of spectral sequence computation associated to the central 
	series of a nilpotent Lie algebra, it seems possible to extend the above result to arbitrary \textit{nilpotent groups}.
	\item[--]  The connection of the Poisson reduction with the Lu's point reduction defined in \cite{Lu1} can be investigated.
	More precisely, the Poisson reduction can be regarded as an orbit reduction, hence we aim to generalize the
	\textit{ Reduction diagram theorem}, proved in \cite{RO} for canonical actions of Lie groups. This would complete the
	  analogy with the symplectic theory.
	\item[--] As suggested by Rui Loja Fernandes, Poisson reduction can be rephrased in terms of \textit{Dirac structures}
	   \cite{Co}. Since $M$ is a Poisson manifold, it is known that a Dirac structure on it can be defined by the graph of the
	    map $\pi^{\sharp}:T^*M\rightarrow TM$. We would like to prove that the reduced space, as defined in Section 
	    \ref{sec: poisson reduction}, inherits an integrable Dirac structure from $M$ and hence, a Poisson structure. Explicitly,
	    we are interested to demonstrate the following claim:
	    
\textit{
Consider the Poisson action $G\times M\rightarrow M$ with equivariant momentum map $\bs{\mu}:M\ra G\st$. Let $x\in G\st$ be a regular value of $\bs{\mu}$ and assume that
the action is proper and free on $\bs{\mu}\inv(x)$. Then one has a natural isomorphism
\begin{equation}
 \boldsymbol{\mu}^{-1}(x)/G_{x}\simeq \boldsymbol{\mu}^{-1}(\mathcal{O}_{x})/G,
\end{equation}
where $\mathcal{O}_{x}\subset G^*$ denotes the dressing orbit of $G$ through $x$ and $G_{x}$ denotes the isotropy group of $x$. This isomorphism is a Poisson diffeomorphism for the unique Poisson structures on these
quotients which arise from the diagram
\begin{diagram}
&              &M \\
&\ruTo     &            &\luTo\\
\boldsymbol{\mu}^{-1}(x) &             &        &        &\boldsymbol{\mu}^{-1}(\mathcal{O}_{x})\\
   &\rdTo    &          &\ldTo\\
&              &\boldsymbol{\mu}^{-1}(x)/G_{x}\simeq \boldsymbol{\mu}^{-1}(\mathcal{O}_{x})/G
\end{diagram}
where the inclusions are backward Dirac maps and the projections are forward Dirac maps.}

This new formulation of the Poisson reduction leads to a new question, i.e. the relation with the theory of Dirac reduction. 
\end{enumerate}

\section*{Quantum Momentum map and Quantum Reduction}

The contribution of Chapter \ref{ch: three} is the study of a new definition of the quantum momentum map
associated to the quantized action. Some directions for the future work are discussed below:
\begin{enumerate}
	\item[--] The first open question we would approach is a formal definition of the \textit{quantum reduction}. First we have to
	 complete the example of the Hopf algebra action of $U_{\hbar}(\mathfrak{su}(2))$ on $C^{\infty}_{\hbar}(M)$.
	 Let $H=a^{-2}-e^\hbar\frac{(1-e^{2\hbar})^2}{\hbar^2}cb$. Then we obtain
	 \begin{align}
	 a^{-1}H a&=H\\
	 [b,H]&=-(1-e^{2\hbar})H b\\
	 [c,H]&=c(1-e^{2\hbar})H
	 \end{align}
 In particular, the ideal $\mathcal{I}$ generated by $H$ in $C^{\infty}_{\hbar}(M)$ is $\mc{U}_{\hbar}(\mathfrak{su}(2))$-invariant, and
 \begin{equation}
(C^{\infty}_{\hbar}(M)/\mathcal{I})^{\mc{U}_{\hbar}(\mathfrak{su}(2))}
\end{equation}
is a deformation quantization of the Poisson reduction
\begin{equation}
M/\!/SU(2)
\end{equation}
corresponding to the symplectic leaf $a_0^{-2}-4b_0 c_0=0$ in $SU(2)^*=SB(2,{\mathbb C} )$.
Notice that the above leaf is not compact, hence the action cannot integrate to the action of SU(2). It would be interesting to
investigate on the structure of the other leaves. This would help us to understand how to define correctly the quantum
 reduction.

	\item[--] The procedure of quantization can applied to the Poisson actions lifted to symplectic groupoids.	
It is known that a Poisson action often does not admit a momentum map. For this reason Fernandes and Ponte in \cite{FP}
 define a \textit{symplectization functor} which turns this action into a Hamiltonian action. Given a Poisson manifold there
  is a canonical symplectic groupoid $\Sigma(M)\rightrightarrows M$ and the action lifted to symplectic groupoids always
   admits a momentum map $\boldsymbol{\mu}:\Sigma(M)\rightarrow G^*$. It would be interesting to analyze the
    quantization of this momentum map. Similarly to the case of a Poisson Lie group, the quantization of the symplectic
     groupoid is given by a quantum groupoid, using the theories studied by Lu and Xu \cite{Lu2}, \cite{Xu1}.

	\item[--] Another significant future direction of research aims at comparing other approaches to quantization different than the
 deformation one, in particular we are interested in \textit{geometrical quantization} \cite{WH}.
This is motivated by an example of the Gelfand-Cetlin system considered by Guillemin and Sternberg in
\cite{GS}. In this example geometric quantization  is defined via higher
cohomology groups.
More precisely, Gelfand-Cetlin system is an integrable system on
$\mathfrak{u(n)}^*$ obtained by Thimm's method. This integrable system can be
viewed as an integrable system on the coadjoint orbits of $\mathfrak{u(n)}^*$
and therefore as a collection of real polarization on the symplectic
leaves. There is a natural Hamiltonian action of a compact Lie group
in this setting: a toric action is associated to a natural choice of
action coordinates.
 In the case of  geometric quantization of the Gelfand-Cetlin system,
 Guillemin and Sternberg obtain results in representation
theory.

We would like to apply a similar scheme to other examples in Poisson
Lie groups which inherit real polarizations using a Thimm's method \cite{M}, \cite{HM} and
compare the results obtained via deformation quantization.

\end{enumerate}

%%% ----------------------------------------------------------------------

% ------------------------------------------------------------------------

%%% Local Variables: 
%%% mode: latex
%%% TeX-master: "../thesis"
%%% End: 

%% file: PhDThesis.bbl
\begin{thebibliography}{10}

\bibitem{A}
V.I. Arnol'd.
\newblock Small denominators {III}. {Small} denominators and problems of
  stability of motion in classical and celestial mechanics.
\newblock {\em Russ. Math. Surveys}, 1963.

\bibitem{BFLS}
F.~Bayen, M.~Flato, C.~Fronsdal, A.~Lichnerowicz, and D.~Sternheimer.
\newblock Deformation theory and {Q}uantization {I}-{II}.
\newblock {\em Annals of Physics}, 111:61--110, 111--151, 1978.

\bibitem{Bo}
N.~Bourbaki.
\newblock {\em Lie groups and {L}ie algebras}.
\newblock Springer-{Verlag}, 1989.

\bibitem{CW}
A.~Cannas~de Silva and A.~Weinstein.
\newblock {\em Geometrical models of {N}on {C}ommutative {A}lgebras}.
\newblock American Mathematical Society, 1999.

\bibitem{CE}
E.~Cartan and S.~Eilenberg.
\newblock {\em Homological algebra}.
\newblock Princeton University Press, 1956.

\bibitem{CI}
A.~S. Cattaneo and D.~Indelicato.
\newblock Formality and star product.
\newblock In {\em Poisson geometry, deformation quantization and group
  representation}, 2004.
\newblock Cambridge University Press.

\bibitem{C}
V.~Chari.
\newblock {\em A guide to {Q}uantum {G}roups}.
\newblock Cambridge University Press, 1994.

\bibitem{Co}
T.~Courant.
\newblock Dirac manifolds.
\newblock {\em Trans. A.M.S.}, 319:631--661, 1990.

\bibitem{DL}
M.~De~Wilde and P.~B.~A. Lecomte.
\newblock Existence of star-products and of formal deformations of the
  {Poisson} {Lie} algebra of arbitrary symplectic manifolds.
\newblock {\em Lett. Math. Phys.}, 7:487--496, 1983.

\bibitem{Drinfeld}
V.G. Drinfel'd.
\newblock Hamiltonian structures on {L}ie groups, {L}ie bialgebras, and the
  geometric meaning of the classical {Y}ang-{B}axter equations.
\newblock {\em Soviet Math. Dokl.}, 27, 1983.

\bibitem{EK}
P.~Etingof and D.~Kazhdan.
\newblock Quantization of {Lie} bialgebras {I}.
\newblock {\em Selecta Math.}, 2:1--41, 1996.

\bibitem{Fe}
B.V. Fedosov.
\newblock A simple {G}eometrical {C}onstruction of {D}eformation
  {Q}uantization.
\newblock {\em Journal of Differential Geometry}, 40:213--238, 1994.

\bibitem{Fe1}
B.V. Fedosov.
\newblock Non abelian reduction in deformation quantization.
\newblock {\em Letters in Mathematical Physics}, 43(2):137--154, 1998.

\bibitem{FP}
R.~L. Fernandes and D.~I. Ponte.
\newblock Integrability of {P}oisson {L}ie group actions.
\newblock {\em Letters in Mathematical Physics}, 90, 2009.

\bibitem{FOR}
R.~L. Fernandes, D.~I. Ponte, and T.~Ratiu.
\newblock Momentum map in {P}oisson geometry.
\newblock {\em Americal Journal of Mathematics}, 131(5), 2009.

\bibitem{FLS1}
M.~Flato, A.~Lichnerowicz, and D.~Sternheimer.
\newblock D\'eformations 1-diff\'erentiales des alg\`ebres de {L}ie attach\'ees
  \`a une vari\'et\'e symplectique ou de contact.
\newblock {\em Compositio Math.}, 31:47--82, 1975.

\bibitem{FLS2}
M.~Flato, A.~Lichnerowicz, and D.~Sternheimer.
\newblock Crochet de {M}oyal-{V}ey et quantification.
\newblock {\em C. R. Acad. Sci. Paris S\'er. A-B}, 283:61--110, 111--151, 1976.

\bibitem{GD}
I.~M. Gel'fand and L.~A. Dickey.
\newblock A family of {Hamiltonian} structures connected with integrable
  nonlinear differential equations.
\newblock Preprint 136, IPM AN SSSR, Moscow, 1978.

\bibitem{G}
M.~Gerstenhaber.
\newblock The cohomology structure of an associative ring.
\newblock {\em Ann. Math.}, 78(2):267--288, 1963.

\bibitem{GDMS}
Ph.~A. Griffiths, P.~Deligne, J.~W. Morgan, and D.~Sullivan.
\newblock Real homotopy theory of {K}\"{a}hler manifolds.
\newblock {\em Uspekhi Mat. Nauki}, 32, 1997.

\bibitem{Gro}
H.~J. Groenewold.
\newblock On the principles of elementary quantum mechanics.
\newblock {\em Physics}, 12:405--460, 1946.

\bibitem{GS}
V.~Guillemin and S.~Sternberg.
\newblock The {G}el'fand-{C}etlin system and quantization of the complex flag
  manifolds.
\newblock {\em J. Funct. Anal.}, 52:106--128, 1983.

\bibitem{Gu}
S.~Gutt and J.~Rawnsley.
\newblock Equivalence of star products on a symplectic manifold: an
  introduction to {D}eligne's {C}ech cohomology classes.
\newblock {\em Journ. of Geom. and Phys}, 29:347--392, 1999.

\bibitem{HM}
M.~Hamilton and E.~Miranda.
\newblock Geometric quantization of integrable systems with hyperbolic
  singularities.
\newblock {\em Annales de l'Institut Fourier}, 60:51--85, 2010.

\bibitem{HKR}
G.~Hochschild, B.~Kostant, and A.~Rosenberg.
\newblock Differential forms on regular affine algebras.
\newblock {\em Trans. Amer. Math. Soc.}, 102(3):383--408, 1962.

\bibitem{Ko1}
M.~Kontsevich.
\newblock Formality conjecture.
\newblock In J.~Rawnsley D.~Sternheimer and S.~Gutt, editors, {\em Deformation
  theory and symplectic geometry}, pages 139--156, Ascona, 1996.
\newblock Math. Phys. Stud.

\bibitem{Ko}
M.~Kontsevich.
\newblock Deformation {Q}uantization of {P}oisson {M}anifolds.
\newblock {\em Letters in Mathematical Physics}, 66:157--216, 2003.

\bibitem{YK}
Y.~Kosmann-Schwarzbach.
\newblock Lie bialgebras, {P}oisson {L}ie groups and dressing transformations.
\newblock In Y.~Kosmann-Schwarzbach, B.~Grammaticos, and K.~M. Tamizhmani,
  editors, {\em Integrability of Nonlinear Systems}. Springer, 2004.

\bibitem{Ks}
B.~Kostant.
\newblock Orbits, symplectic structures and representation theory.
\newblock {\em Proc. US-Japan Seminar on Diff. Geom., Kyoto. Nippon Hyronsha,
  Tokyo}, 77, 1965.

\bibitem{Ks1}
B.~Kostant.
\newblock The solution to the generalized {Toda} lattice and representation
  theory.
\newblock {\em Adv. Math.}, 34, 1979.

\bibitem{Li}
A.~Lichnerowicz.
\newblock Les vari\'et\'es de {P}oisson et leurs alg\'ebres de {L}ie
  associ\'ees.
\newblock {\em Journal of Differential Geometry}, 12, 1977.

\bibitem{Lie}
S.~Lie.
\newblock {\em Theorie der {T}ransformationgruppen. {Z}weiter {A}bschnitt}.
\newblock Teubner, Leipzig, 1890.

\bibitem{Lu3}
J-H. Lu.
\newblock {\em Multiplicative and {A}ffine {P}oisson structure on {L}ie
  groups}.
\newblock PhD thesis, University of California (Berkeley), 1990.

\bibitem{Lu1}
J-H. Lu.
\newblock Momentum mappings and reduction of {P}oisson actions.
\newblock In {\em Symplectic geometry, groupoids, and integrable systems},
  pages 209--226. Math. Sci. Res. Inst. Publ., Berkeley, California, 1991.

\bibitem{Lu4}
J-H. Lu.
\newblock Moment {M}aps at the {Q}uantum level.
\newblock {\em Communications in Mathematical Physics}, pages 389--404, 1993.

\bibitem{Lu2}
J-H. Lu.
\newblock Hopf {A}lgebroids and quantum groupoids.
\newblock {\em International Journal of Mathematics}, 7, 1995.
\newblock arXiv:q-alg/9505024v1.

\bibitem{Ma}
S.~Majid.
\newblock {\em Foundations of Quantum Group Theory}.
\newblock Cambridge University Press, 1995.

\bibitem{MR}
J.~Marsden and T.~Ratiu.
\newblock {\em Introduction to {M}echanics and {S}ymmetry}, volume~17.
\newblock Springer-{Verlag}, 1999.

\bibitem{Md}
J.~E. Marsden and T.~Ratiu.
\newblock Reduction of {P}oisson manifolds.
\newblock {\em Letters in Mathematical Physics}, 11(2):161--169, 1986.

\bibitem{MsWe}
J.~E. Marsden and A.~Weinstein.
\newblock Reduction of symplectic manifolds with symmetry.
\newblock {\em Rep. Math. Phys.}, 1974.

\bibitem{MsWe1}
J.~E. Marsden and A.~Weinstein.
\newblock The {H}amiltonian structure of the {M}axwell-{V}lasov equations.
\newblock {\em Physica}, 1982.

\bibitem{M}
E.~Miranda.
\newblock From action-angle coordinates to geometric quantization: a round
  trip.
\newblock {\em Oberwolfach report, Geometric quantization in the non-compact
  setting, MFO Oberwolfach reports}, 2011.

\bibitem{Mo}
J.~E. Moyal.
\newblock Quantum mechanics as a statistical theory.
\newblock {\em Proc. Cambridge Phyl. Soc.}, 45:99--124, 1949.

\bibitem{NT}
R.~Nest and B.~Tsygan.
\newblock Algebraic index theorem for families.
\newblock {\em Advances in Math.}, 113:151--205, 1995.

\bibitem{N}
E.~Noether.
\newblock Invariant variation problems.
\newblock {\em Transport Theory and Statistical Physics}, 1971.

\bibitem{P}
S.-D. Poisson.
\newblock Sur la variation des constantes arbitraires dans les questions de
  m\'ecanique.
\newblock {\em J. Ecole Polytechnique}, 1809.

\bibitem{RO}
T.~Ratiu and J-P. Ortega.
\newblock {\em Momentum {M}aps and {H}amiltonian {R}eduction}, volume 222.
\newblock Birkh{\"a}user, 003.

\bibitem{Re}
N.~Yu. Reshetikhin.
\newblock Quantization of {L}ie bialgebras.
\newblock {\em Duke Math. J. Intern. Math. Res. Notices}, 7, 1992.

\bibitem{ST1}
M.~A. Semenov-Tian-Shansky.
\newblock What is a classical r-matrix?
\newblock {\em Funct. Anal. Appl.}, 4, 1983.

\bibitem{STS}
M.~A. Semenov-Tian-Shansky.
\newblock Dressing transformations and {P}oisson {L}ie group actions.
\newblock {\em Publ. RIMS, Kyoto University}, 21:1237--1260, 1985.

\bibitem{SW}
J.~Sniatycki and A.~Weinstein.
\newblock Reduction and quantization for singular momentum map.
\newblock {\em Letters in mathematical physics}, 1983.

\bibitem{So}
J-M. Souriau.
\newblock {\em Geom\'etrie de l'espace de phases, calcul des variations et
  m\'ecanique quantique.}
\newblock Facult\'e des {S}ciences de {M}arseille, 1965.

\bibitem{Tj}
T.~Tjin.
\newblock An introduction to quantized {Lie} groups and algebras.
\newblock {\em Int. J. Mod. Phys. A}, 7:6175--6213, 1992.

\bibitem{V}
I.~Vaisman.
\newblock {\em Lectures on the {G}eometry of {P}oisson {M}anifolds}.
\newblock Birkh{\"a}user, 1994.

\bibitem{We1}
A.~Weinstein.
\newblock The local structure of {P}oisson {M}anifolds.
\newblock {\em Journal of Differential Geometry}, 1983.

\bibitem{We}
A.~Weinstein.
\newblock Some remarks on dressing transformations.
\newblock {\em J. Fac. Sci. Univ. Tokyo}, 1988.

\bibitem{Wy}
H.~Weyl.
\newblock Quantenmechanik und {G}ruppentheorie.
\newblock {\em Z. Physics}, 46:1--46, 1927.

\bibitem{Wi}
E.~P. Wigner.
\newblock Quantum corrections for thermodynamic equilibrium.
\newblock {\em Phys. Rev.}, 40:749--759, 1932.

\bibitem{WH}
N.~M.~J. Woodhouse.
\newblock {\em Geometric quantization}.
\newblock Oxford Mathematical Monographs. Clarendon Press, 2nd edition, 1992.

\bibitem{Xu1}
P.~Xu.
\newblock Quantum groupoids.
\newblock {\em Comm. Math. Phys.}, 216:539--581, 2001.
\newblock arXiv:math/9905192v2 [math.QA].

\end{thebibliography}
